\documentclass{amsart}
\usepackage{amssymb}
\usepackage[nobysame]{amsrefs}
\usepackage{mathpazo}
\usepackage{enumerate}
\usepackage{subcaption}
\usepackage{ifthen}
\usepackage[figuresright]{rotating}
\usepackage{longtable}
\usepackage[normalem]{ulem}
\usepackage{tikz,pgf,pgfmath}
	\tikzstyle{edge}=[line width=.75pt]
	\pgfdeclarelayer{background}
	\pgfdeclarelayer{middle}
	\pgfsetlayers{background,middle,main}
\usepackage{mathtools}
\usepackage[colorlinks=true,
	urlcolor=blue!20!black,
	citecolor=green!50!black]{hyperref}
\definecolor{c1}{rgb}{1, 0.8, 0.6}
\definecolor{c2}{rgb}{0.5, 1, 0.8}
\definecolor{c3}{rgb}{1, 0.4, 0.9}
\definecolor{c4}{rgb}{0.5, 0.9, 1}
\definecolor{c5}{RGB}{213,94,0}		
\newtheorem{theorem}{Theorem}[section]
\newtheorem{proposition}[theorem]{Proposition}
\newtheorem{conjecture}[theorem]{Conjecture}
\newtheorem{lemma}[theorem]{Lemma}
\newtheorem{corollary}[theorem]{Corollary}
\newtheorem{question}[theorem]{Question}
\theoremstyle{remark}
\newtheorem{definition}[theorem]{Definition}
\newtheorem{example}[theorem]{Example}
\newtheorem{remark}[theorem]{Remark}
\newtheorem{note}[theorem]{Note}
\newcommand{\defn}[1]{{\color{green!50!black}\textit{#1}}}
\newcommand{\defs}{\stackrel{\mathsf{def}}{=}}
\newcommand{\ie}{\textit{i.e.},\;}
\newcommand{\eg}{\textit{e.g.},\;}
\renewcommand{\th}{^{\mathsf{th}}}
\newcommand{\st}{^{\mathsf{st}}}
\newcommand{\rsp}{-.16cm}
\newcommand*\circled[1]{\tikz[baseline=(char.base)]{
            \node[shape=circle,fill=gray!50!black, minimum size=11pt, inner sep=-3pt, outer sep = 0pt] (char) {\textcolor{white}{$\vphantom{\overline{\mathbf{8}}} \mathbf{#1}$}};}}
\newcommand{\hgl}[1]{\circled{#1}}
\newcommand{\Poset}{\mathbf{P}}
\newcommand{\least}{\hat{0}}
\newcommand{\grtst}{\hat{1}}
\newcommand{\JI}{\mathsf{J}}
\newcommand{\MI}{\mathsf{M}}
\newcommand{\Con}{\mathsf{Con}}
\newcommand{\cg}{\mathsf{cg}}
\newcommand{\id}{\mathsf{e}}
\newcommand{\wo}{\omega_{\mathsf{o}}}
\newcommand{\woa}{\omega_{\mathsf{o};\alpha}}
\newcommand{\woab}{\omega_{\mathsf{o};\alpha'}}
\newcommand{\Weak}{\mathsf{Weak}}
\newcommand{\weakorder}{\leq_{\mathsf{weak}}}
\newcommand{\weakless}{<_{\mathsf{weak}}}
\newcommand{\weakcover}{\lessdot_{\mathsf{weak}}}
\newcommand{\Tamari}{\mathsf{Tam}}
\newcommand{\Align}{\mathsf{Align}}
\newcommand{\Skew}{\mathsf{skew}}
\newcommand{\invset}{\mathsf{Inv}}
\newcommand{\covset}{\mathsf{Cov}}
\newcommand{\Hyper}{\mathfrak{H}}
\newcommand{\sleft}{[\![}
\newcommand{\sright}{]\!]}
\newcommand{\lleft}{(\!(}
\newcommand{\rright}{)\!)}
\newcommand{\invs}[1]{\sleft {#1} \sright}
\newcommand{\invd}[2]{\lleft {#1} \; {#2} \rright}
\newcommand{\invorder}{\mathbf{Inv}}
\newcommand{\wb}{\mathbf{w}}
\newcommand{\sbo}{\mathbf{s}}
\newcommand{\wbo}{\mathbf{w}_{\mathsf{o}}}
\newcommand{\wboa}{\mathbf{w}_{\mathsf{o};\alpha}}
\newcommand{\linc}{\vec{c}}
\newcommand{\Part}{\mathsf{Part}}
\newcommand{\indicator}{\varrho}
\title{Parabolic Tamari Lattices in Linear Type $B$}
\author{Wenjie Fang}
	\address{WF: LIGM, Univ. Gustave Eiffel, CNRS, ESIEE Paris, F-77454 Marne-la-Vallée, France}
	\email{wenjie.fang@u-pem.fr}
\author{Henri M{\"u}hle}
	\address{HM: Qoniac GmbH, Dr. K{\"u}lz-Ring 15, 01067 Dresden, Germany.}
	\email{henri.muehle@tu-dresden.de}
\author{Jean-Christophe Novelli}
	\address{JCN: LIGM, Univ. Gustave Eiffel, CNRS, ESIEE Paris, F-77454 Marne-la-Vallée, France}
	\email{novelli@univ-mlv.fr}
\thanks{HM has received funding from the European Research Council (Grant Agreement no. 681988, CSP-Infinity). JCN was partially supported by the
CARPLO project of the Agence Nationale de la recherche (ANR-20-CE40-0007).}
\newcounter{i}
\newcounter{n}

\newcommand{\crossout}[3]{
    \draw[line width=.95pt,red]({(#1-.25)*#3},{(#2-.25)*#3}) -- ({(#1+.25)*#3},{(#2+.25)*#3});
    \draw[line width=.95pt,red]({(#1-.25)*#3},{(#2+.25)*#3}) -- ({(#1+.25)*#3},{(#2-.25)*#3});
}

\newcommand{\AlphaPerm}[7]{
	\edef\alphaList{{#1}}
	\begin{tikzpicture}
		\def\d{#5};
		\def\dx{#4};
		\setcounter{i}{0}
		\setcounter{n}{0}
		\foreach \cl in {#2}{
            \pgfmathsetmacro{\ax}{\alphaList[\thei]}
            \addtocounter{n}{\ax}
            \stepcounter{i}
		}
		\begin{pgfonlayer}{background}
            \draw[#6,fill=white,rounded corners,line width=2pt]({(0+1.5*\d)*\dx},-2.2*\d*\dx) -- ({(\then+1-1.5*\d)*\dx},-2.2*\d*\dx) -- ({(\then+1-1.5*\d)*\dx},2.2*\d*\dx) -- ({(0+1.5*\d)*\dx},2.2*\d*\dx) -- cycle;
        \end{pgfonlayer}{background}

        \setcounter{i}{0}
		\setcounter{n}{0}
		\foreach \cl in {#2}{
			\pgfmathsetmacro{\ax}{\alphaList[\thei]}
			\begin{pgfonlayer}{background}
				\fill[\cl]({(\then+1-1.5*\d)*\dx},-1.55*\d*\dx) -- ({(\then+\ax+1.5*\d)*\dx},-1.55*\d*\dx) -- ({(\then+\ax+1.5*\d)*\dx},1.55*\d*\dx) -- ({(\then+1-1.5*\d)*\dx},1.55*\d*\dx) -- cycle;
			\end{pgfonlayer}
			\addtocounter{n}{\ax}
			\stepcounter{i}
		}
		\setcounter{i}{1}
		\foreach \k in {#3}{
			\draw(\thei*\dx,.65*\dx) node[anchor=north,scale=#7]{$\k\vphantom{\overline{\uwave{k}}}$};
			\addtocounter{i}{1};
		}
	\end{tikzpicture}
}

\begin{document}

\begin{abstract}
  We study parabolic aligned elements associated with the type-$B$ Coxeter group and the so-called linear Coxeter element.  These elements were introduced algebraically in (M{\"u}hle and Williams, 2019) for parabolic quotients of finite Coxeter groups and were characterized by a certain forcing condition on inversions.  We focus on the type-$B$ case and give a combinatorial model for these elements in terms of pattern avoidance.  Moreover, we describe an equivalence relation on parabolic quotients of the type-$B$ Coxeter group whose equivalence classes are indexed by the aligned elements.  We prove that this equivalence relation extends to a congruence relation for the weak order. The resulting quotient lattice is the type-$B$ analogue of the parabolic Tamari lattice introduced for type $A$ in (M{\"u}hle and Williams, 2019).  
\end{abstract}

\maketitle

\section{Introduction}
    \label{sec:introduction}
Knuth introduced in \cite[Section~2.2.1]{knuth97aop_vol1} the family of \textit{stack-sortable permutations}.  These are permutations that can be converted into the identity permutation by passing through a stack. Combinatorially, stack-sortable permutations are characterized by avoiding the pattern $231$. More precisely, a permutation $\pi$ of $\{1,2,\ldots,n\}$ avoids the pattern $231$ if and only if for every $i<j<k$, when $\pi(i)>\pi(k)$, we have $\pi(i)>\pi(j)$. Following Bj{\"o}rner and Wachs this pattern-avoidance can be interpreted as a certain \textit{forcing} of inversions~\cite[Section~9]{bjorner97shellable}.

Using a root-theoretic approach, Reading generalized this condition to all finite irreducible Coxeter groups $W$ and all Coxeter elements $c$.  These \textit{$c$-aligned elements} of $W$ play an important role in the very active stream of Coxeter--Catalan combinatorics.  First, the number of $c$-aligned elements of $W$ is the $W$-Catalan number (for any choice of Coxeter element $c$).  Second, the $c$-aligned elements provide a bijective bridge between $c$-noncrossing partitions and $c$-clusters associated with $W$~\cite{reading07clusters}.

On top of that, the $c$-aligned elements behave nicely from a lattice-theoretic point of view.  Inside the weak order, they form the \textit{$c$-Cambrian lattice}; a particular semidistributive and trim lattice. This family of lattices generalizes the famous \textit{Tamari lattice}~\cite{reading06cambrian}.  Another remarkable feature is that these lattices arise from a certain orientation of the $1$-skeleton of the \textit{$c$-associahedron}, a simple polytope associated with the Coxeter group $W$.  In fact, the boundary complex of the dual polytope of the $c$-associahedron is the $c$-cluster complex associated with $W$, which arises from the \textit{almost positive roots} associated with $W$ by means of a certain \textit{compatibility} condition.

In \cite{muehle19tamari}, $c$-aligned elements, $c$-noncrossing partitions and $c$-clusters were generalized to \textit{parabolic quotients} of finite irreducible Coxeter groups.  In this generalization, some properties of these objects were conjectured to remain, such as the lattice property of parabolic $c$-aligned elements under weak order (\cite[Conjecture~35]{muehle19tamari}), and---when $W$ is of \textit{coincidental type}---the equinumerosity of parabolic $c$-aligned elements, parabolic $c$-noncrossing partitions and parabolic $c$-clusters (\cite[Conjecture~41]{muehle19tamari}).  The case of the symmetric group with $c$ the increasing long cycle was settled in the same article.  Further research has exhibited remarkable connections between parabolic Coxeter--Catalan objects associated with the symmetric group, certain Hopf algebras, and the theory of diagonal harmonics~\cite{ceballos20the}.  For more properties of these objects, readers are referred to \cite{fang21consecutive,krattenthaler22the,muehle21noncrossing}.

\medskip

In this article, we present a first study of parabolic Coxeter--Catalan objects associated with the hyperoctahedral group, \ie the Coxeter group of type $B$.  We start by recalling the necessary basics on posets and lattices (Section~\ref{sec:posets_lattices}) and Coxeter groups (Section~\ref{sec:coxeter_groups}).  Then we turn our attention to the Coxeter group of type $B$, which can be viewed as the symmetry group of the hyperoctahedron.  We recall a permutation representation of this group in terms of sign-symmetric permutations (Section~\ref{sec:sign_symmetric_permutations}), and then consider a particular realization of the roots associated with this group (Section~\ref{sec:type_b_roots}).  With the basics set up we then recall Reading's construction of the $\linc$-aligned elements associated with the Coxeter group of type $B$, where $\linc$ is a \textit{linear} Coxeter element, \ie a particular long cycle (Section~\ref{sec:type_b_alignment}).

In the main part of this article, we consider parabolic quotients of the hyperoctahedral group of degree $n$.  These parabolic quotients consist of minimal-length representatives of the left cosets with respect to certain subgroups.  We give a simple criterion for the membership of a sign-symmetric permutation to such a parabolic quotient (Section~\ref{sec:representatives}). In particular, it turns out that the parabolic quotients are determined by a type-$B$ composition $\alpha$ of $n$, \ie an integer composition of $n$ with a possible additional $0$-component.  We denote such a parabolic quotient by $\Hyper_{\alpha}$.

Next, we study the longest element in $\Hyper_{\alpha}$ and describe its $\linc$-sorting word together with the associated inversion order by means of skew shapes (Section~\ref{sec:parabolic_longest_elements}).  We then characterize the $\linc$-aligned elements associated with a parabolic quotient of the hyperoctahedral group algebraically in terms of a forcing condition on the inversions and combinatorially in terms of a pattern-avoidance condition (Section~\ref{sec:parabolic_aligned_elements}).  It shall be remarked here that there are two types of type-$B$ compositions: those with a $0$-component and those without.  This dichotomy has an impact on the pattern-avoidance condition, because the avoided patterns differ slightly depending on whether the $0$-component is present or not.  Nevertheless, we may write $\Align\bigl(\Hyper_{\alpha},\wboa(\linc)\bigr)$ uniformly for the set of $\linc$-aligned elements in $\Hyper_{\alpha}$.

In the last section of this article (Section~\ref{sec:parabolic_tamari_lattice}), we prove that $\Align\bigl(\Hyper_{\alpha},\wboa(\linc)\bigr)$ induces a quotient lattice of the weak order on $\Hyper_{\alpha}$, called the \textit{type-$B$ parabolic Tamari lattice} $\Tamari_{B}(\alpha)$.

\begin{theorem}\label{thm:parabolic_tamari_lattice}
    For every type-$B$ composition $\alpha$, $\Tamari_{B}(\alpha)$ is a lattice. Moreover, it is a quotient lattice of the weak order on the parabolic quotient $\Hyper_{\alpha}$.
\end{theorem}

These lattices enjoy some remarkable structural properties.

\begin{theorem}\label{thm:parabolic_tamari_properties}
    For every type-$B$ composition $\alpha$, $\Tamari_{B}(\alpha)$ is congruence uniform (hence semidistributive) and trim.
\end{theorem}

In fact, congruence-uniformity of $\Tamari_{B}(\alpha)$ follows right away from the fact that $\Tamari_{B}(\alpha)$ is a quotient lattice of the weak order because the latter is itself congruence uniform. Since trimness of $\Tamari_{B}(\alpha)$ cannot be deduced so easily, we think that it is justified to state these important structural properties together as a theorem.

We conclude this article with a brief outlook on future research in this direction (Section~\ref{sec:outlook}).

\section{Posets and lattices}
    \label{sec:posets_lattices}
\subsection{Partially ordered sets}

Throughout this section, let $P$ denote a finite set. If $P$ is equipped with a reflexive, antisymmetric and transitive relation $\leq$, then $\Poset={(P,\leq)}$ is a partially ordered set (or \defn{poset} for short).  An induced \defn{subposet} of $\Poset$ is a poset $(Q,\leq')$ for $Q\subseteq P$, where $q\leq'q'$ if and only if $q\leq q'$ for all $q,q'\in Q$.  The \defn{direct product} of $\Poset$ with a poset $(P',\leq')$ is the poset $(P\times P',\preceq)$ with $(p_{1},p'_{1})\preceq(p_{2},p'_{2})$ if and only if $p_{1}\leq p_{2}$ and $p'_{1}\leq'p'_{2}$.

An element $p\in P$ is \defn{minimal} if for every $q\in P$ with $q\leq p$ it follows that $p=q$. It is \defn{maximal} if for every $q\in P$ with $p\leq q$ it follows that $p=q$. If $\Poset$ has a unique minimal and a unique maximal element, then $\Poset$ is \defn{bounded}. Then, we usually write $\least$ (resp. $\grtst$) for the unique minimal (resp. maximal) element of $\Poset$.

An \defn{interval} of $\Poset$ is a set $[p,q]\defs\{r\in P\colon p\leq r\leq q\}$ for some $p,q\in P$ with $p\leq q$. A \defn{covering pair} (or \defn{cover} for short) of $\Poset$ is a pair $(p,q)\in P\times P$ such that $p<q$ and there does not exist an $r\in P$ such that $p<r<q$. We usually write $p\lessdot q$ in this situation. In other words, a covering pair induces an interval of cardinality two.

A subset $C\subseteq P$ is a \defn{chain} if we can write $C=\{c_{1},c_{2},\ldots,c_{k}\}$ such that $c_{1}<c_{2}<\cdots<c_{k}$.  A chain is \defn{maximal} if there exists no strictly larger chain $C'$ with $C\subsetneq C'$.  In other words, a maximal chain $C$ can be written as $C=\{c_{1},c_{2},\ldots,c_{k}\}$ with $c_{1}\lessdot c_{2}\lessdot\cdots\lessdot c_{k}$, where $c_{1}$ is minimal and $c_{k}$ is maximal.  The \defn{length} of $\Poset$ is
\begin{displaymath}
	\ell(\Poset) \defs \max\bigl\{\lvert C\rvert-1\colon C\;\text{is a maximal chain of}\;\Poset\bigr\}.
\end{displaymath}

If $(P',\leq')$ is another poset, then a map $f\colon P\to P'$ is \defn{order preserving} if for all $p,q\in P$ with $p\leq q$ it holds that $f(p)\leq' f(q)$.

\subsection{Lattices}
	
An element $p\in P$ is a \defn{lower bound} for a subset $A\subseteq P$ if $p\leq a$ for all $a\in A$.  The \defn{greatest lower bound} of $A$ is (if it exists) the unique maximal element of the set of lower bounds of $A$ (considered as a subposet of $\Poset$).  Upper bounds and least upper bounds are defined analogously.

Then, $\Poset$ is a \defn{lattice} if for all $p,q\in P$ the set $\{p,q\}$ has a greatest lower bound (the \defn{meet}, denoted by $p\wedge q$) and a least upper bound (the \defn{join}, denoted by $p\vee q$).  It is easy to see that every finite lattice is bounded.  A \defn{sublattice} of $\Poset$ is a subposet $\mathbf{Q}=(Q,\leq)$ which is itself a lattice and for every $p,q\in Q$ the join (and meet) of $p$ and $q$ in $\mathbf{Q}$ is the same as the join (and meet) in $\Poset$.

An element $j\in P\setminus\{\least\}$ is \defn{join-irreducible} if $j=p\vee q$ implies $j\in\{p,q\}$ for all $p,q\in P$.  Since $P$ is finite, for every join-irreducible element $j\in P$ there exists a unique element $j_{*}\in P$ such that $j_{*}\lessdot j$.  \defn{Meet-irreducible} elements are defined analogously.  Let $\JI(\Poset)$ and $\MI(\Poset)$ denote respectively the set of join- and meet-irreducible elements of $\Poset$.
 
\subsection{Congruence-uniform lattices}

Let $\Poset=(P,\leq)$ be a lattice.  A \defn{lattice congruence} of $\Poset$ is an equivalence relation $\Theta$ on $P$ that respects the lattice operations.  More precisely, for $p_{1},p_{2},q_{1},q_{2}\in P$ with $(p_{1},q_{1})\in\Theta$ and $(p_{2},q_{2})\in\Theta$, we have $(p_{1}\vee p_{2},q_{1}\vee q_{2})\in\Theta$ and $(p_{1}\wedge p_{2},q_{1}\wedge q_{2})\in\Theta$.  The set of congruence classes of $\Theta$ forms a lattice again, called the \defn{quotient lattice} $\Poset/\Theta$.  The order relation is given by comparing representatives of each congruence class.

The following characterization of lattice congruences will be useful later.

\begin{lemma}[{\cite[Section~3]{reading06cambrian}}]    \label{lem:congruences_combin}
    Let $\Poset=(P,\leq)$ be a lattice and let $\Theta$ be an equivalence relation on $P$.  Then $\Theta$ is a lattice congruence of $\Poset$ if and only if the following conditions are satisfied:
    \begin{enumerate}[\rm (i)]
        \item Every equivalence class of $\Theta$ is an interval of $\Poset$;
        \item The map that sends $p\in P$ to the unique minimal element in $[p]_{\Theta}$ is order preserving;
        \item The map that sends $p\in P$ to the unique maximal element in $[p]_{\Theta}$ is order preserving.
    \end{enumerate}
\end{lemma}

Given another lattice $\mathbf{Q}={(Q,\leq_{Q})}$, a map $f\colon P\to Q$ is a \defn{lattice homomorphism} if for all $p,p'\in P$ we have $f(p\vee p')=f(p)\vee_{Q}f(p')$ and $f(p\wedge p')=f(p)\wedge_{Q}f(p')$. If $f$ is surjective, then the preimages of $f$ induce a lattice congruence on $\Poset$.  Moreover, given a lattice congruence $\Theta$, the map that sends $p\in P$ to the minimal element of $[p]_{\Theta}$ is a surjective lattice homomorphism.

We may as well consider the set of \textit{all} congruence relations on $\Poset$ ordered by refinement. This is again a lattice~\cite{funayama42distributivity}; the \defn{congruence lattice} of $\Poset$, denoted by $\Con(\Poset)$.  From this perspective, a lattice congruence $\Theta$ of $\Poset$ is join-irreducible if it is a join-irreducible element of $\Con(\Poset)$.  For $p\lessdot q$, let $\cg(p,q)$ denote the finest lattice congruence on $\Poset$ in which $p$ and $q$ are congruent, \ie the intersection of \textit{all} lattice congruences in which $p$ and $q$ are congruent.  Interestingly, the covering pairs of $\Poset$ determine the join-irreducible lattice congruences.

\begin{theorem}[{\cite[Section~2.14]{gratzer11lattice}~and~\cite[Theorem~3.20]{freese95free}}]
  Let $\Poset$ be a finite lattice and let $\Theta\in\Con(\Poset)$. The following are equivalent.
    \begin{enumerate}[\rm (i)]
        \item $\Theta$ is join-irreducible in $\Con(\Poset)$.
        \item $\Theta=\cg(p,q)$ for some $p\lessdot q$.
        \item $\Theta=\cg(j_{*},j)$ for some $j\in\JI(\Poset)$.
    \end{enumerate}
\end{theorem}

As a consequence, any finite lattice $\Poset$ admits a surjective map
\begin{displaymath}
    \cg\colon\JI(\Poset)\to\JI\bigl(\Con(\Poset)\bigr), \quad j\mapsto\cg(j_{*},j),
\end{displaymath}
which in general fails to be injective.  A lattice $\Poset$ is \defn{congruence uniform} if the map $\cg$ is a bijection for $\Poset$ and its dual $\Poset^{d}=(P,\geq)$.

\begin{proposition}[\cite{pudlak74yeast}]\label{prop:congruence_uniform_preserving}
    The class of congruence-uniform lattices is preserved under lattice homomorphisms, under passing to sublattices, and under taking (finitely many) direct products.
\end{proposition}

Congruence-uniform lattices have other remarkable properties.  Recall that a lattice $\Poset=(P,\leq)$ is \defn{semidistributive} if for all $p,q,r\in P$ it holds that 
\begin{align*}
    p\vee q=p\vee r & \quad\text{implies}\quad p\vee q=p\vee(q\wedge r),\\
    p\wedge q=p\wedge r & \quad\text{implies}\quad p\wedge q=p\wedge(q\vee r).
\end{align*}
Lattices satisfying only the first implication are \defn{join-semidistributive}.  We record two helpful, auxiliary results.

\begin{lemma}[{\cite[Lemma~4.2]{day79characterizations}}]\label{lem:congruence_uniform_semidistributive}
    Every congruence-uniform lattice is semidistributive.
\end{lemma}

\begin{lemma}[{\cite[Corollary~2.55]{freese95free}}]\label{lem:semidistributive_irreducibles}
    If $\Poset$ is a semidistributive lattice, then $\bigl\lvert\JI(\Poset)\bigr\rvert=\bigl\lvert\MI(\Poset)\bigr\rvert$.
\end{lemma}

\subsection{Trim lattices}

According to \cite{markowsky92primes}, a lattice $\Poset$ is \defn{extremal} if $\bigl\lvert\JI(\Poset)\bigr\rvert=\ell(\Poset)=\bigl\lvert\MI(\Poset)\bigr\rvert$. However, every lattice can be embedded as an interval into an extremal lattice; see \cite[Theorem~14(ii)]{markowsky92primes}.  Thus, extremality alone is not necessarily inherited by intervals.  

A lattice property, which is stronger than extremality and behaves well with respect to sublattices and quotients, was introduced in \cite{thomas06analogue}.  An element $p\in P$ is \defn{left modular} if for all $r,q\in P$ with $r\leq q$ it holds that
\begin{displaymath}
    (r\vee p)\wedge q = r\vee(p\wedge q).
\end{displaymath}
Then, $\Poset$ is \defn{left modular} if it has a maximal chain consisting of $\ell(\Poset)+1$ left-modular elements.  Left-modular lattices have interesting topological properties, though not in the scope of this paper; see for instance~\cite{liu00left,mcnamara06poset}.  We are rather interested in lattices that are both extremal and left modular; following \cite{thomas06analogue} we call such lattices \defn{trim}.

\begin{proposition}[\cite{thomas06analogue,thomas19rowmotion}]\label{prop:trim_preserving}
    The class of trim lattices is preserved under lattice homomorphisms, under passing to intervals, and under taking (finitely many) direct products.
\end{proposition}

It is, however, not necessarily true that any sublattice of a trim lattice is trim again, see \cite[Theorem~3]{thomas06analogue}.  For semidistributive lattices it is somewhat easier to verify trimness, as we only need to establish extremality.

\begin{theorem}[{\cite[Theorem~1.4]{thomas19rowmotion}}]\label{thm:semidistributive_extremal_is_trim}
    Every semidistributive and extremal lattice is trim.
\end{theorem}

\section{Basics on Coxeter groups}
    \label{sec:coxeter_groups}
\subsection{Coxeter groups}

We start with some algebraic background on Coxeter groups.  A \defn{Coxeter group} is a group $W$ with identity element $\id$ and a distinguished set $S=\{s_{1},s_{2},\ldots,s_{n}\}$ of (Coxeter) generators that admits a presentation of the following form:
\begin{displaymath}
    W = \bigl\langle S\colon (s_{i}s_{j})^{m_{i,j}}=\id\bigr\rangle
\end{displaymath}
such that for all $i,j\in[n]\defs\{1,2,\ldots,n\}$ it holds that
\begin{itemize}
    \item $m_{i,j}\geq 1$,
    \item $m_{i,j}=1$ if and only if $i=j$,
    \item $m_{i,j}=m_{j,i}$.
\end{itemize}
The parameters $m_{i,j}$ form the \defn{Coxeter matrix} of $W$.  Note that $m_{i,j}=\infty$ is allowed if there is no relation between the generators $s_{i}$ and $s_{j}$.

A \defn{word} for $w\in W$ is a representation of $w$ by a group product $w_{i_{1}}w_{i_{2}}\cdots w_{i_{k}}$, where $w_{i_{j}}\in W$ for all $j$.  Such a word is \defn{$S$-reduced} if $w_{i_{j}}\in S$ for all $j$ and $w$ cannot be expressed as a product of fewer than $k$ generators. The number of letters in an $S$-reduced word for $w$ is its \defn{Coxeter length}, denoted by $\ell_{S}(w)$.  If $W$ is finite, then there exists a unique element $\wo$, called the \defn{longest element}, which maximizes $\ell_{S}$.

\subsection{The weak order}

The \defn{(left) weak order} on $W$ is defined by 
\begin{displaymath}
    u\weakorder v \quad \text{if and only if}\quad
    \ell_{S}(u)+\ell_{S}(vu^{-1}) = \ell_{S}(v).
\end{displaymath}
The weak order on $W$ constitutes a graded poset whose rank function is precisely $\ell_{S}$.  This order is called \textit{left} weak order, because we have a covering pair $u\weakcover v$ if and only if there exists $s\in S$ such that $v=su$ and $\ell_{S}(v)=\ell_{S}(u)+1$.  In other words, covering pairs are determined by left multiplication with a Coxeter generator.  We sometimes write $\Weak(W)$ instead of $(W,\weakorder)$.

\begin{remark}
    We may as well define a \textit{right} weak order on $W$ in an analogous fashion.  The two orders are isomorphic via reversal of $S$-reduced words.  For our purposes, it will be more convenient to work with the left weak order, even though the right weak order is more frequently used in the literature.
\end{remark}

For our purposes, the following two properties of the weak order are of major interest.

\begin{theorem}[{\cite[Theorem~3.2.1]{bjorner05combinatorics}}]\label{thm:weak_order_lattice}
    $\Weak(W)$ is a meet-semilattice. It is a lattice if and only if $W$ is finite.
\end{theorem}

\begin{theorem}[\cite{caspard04cayley,reading11sortable}]\label{thm:weak_order_congruence_uniform}
    $\Weak(W)$ is join-semidistributive. If $W$ is finite, then $\Weak(W)$ is congruence uniform.
\end{theorem}

\subsection{Reflections and inversions}

The weak order on $W$ can be expressed equivalently via inclusion on special set systems. To explain this connection, we define the set of \defn{reflections} of $W$ by
\begin{displaymath}
	T \defs \bigl\{wsw^{-1}\colon s\in S, w\in W\bigr\}.
\end{displaymath}
Clearly, $S\subseteq T$ and we usually call the elements of $S$ the \defn{simple} reflections of $W$.

A \defn{(right) inversion} of $w\in W$ is a reflection $t\in T$ such that $\ell_{S}(wt)<\ell_{S}(w)$, and we define the \defn{(right) inversion set} of
$w$ by
\begin{displaymath}
    \invset(w) \defs \bigl\{t\in T\colon \ell_{S}(wt)<\ell_{S}(w)\bigr\}.
\end{displaymath}
We have the following correspondence between inversion sets and length of $S$-reduced words.

\begin{lemma}[{\cite[Corollary~1.4.5]{bjorner05combinatorics}}]\label{lem:length_inversions}
    For $w\in W$, we have $\bigl\lvert\invset(w)\bigr\rvert=\ell_{S}(w)$.
\end{lemma}

Proposition~3.1.3 in \cite{bjorner05combinatorics} states that
\begin{displaymath}
    u\weakorder v \quad\text{if and only if}\quad \invset(u)\subseteq\invset(v).
\end{displaymath}
A \defn{cover inversion} of $w$ is $t\in\invset(w)$ such that there exists $s\in S$ with $sw=wt$.  The name comes from the fact that, by definition, the cover inversions of $w$ are in bijection with the elements covered by $w$ in $\Weak(W)$.  The \defn{cover inversion set} of $w$ is denoted by $\covset(w)$.

\subsection{Roots and aligned elements}

A Coxeter group admits a representation as a group of reflections acting on a Euclidean vector space~\cite{humphreys90reflection}; these reflections correspond bijectively to the elements of $T$.  Such a reflection sends a non-zero vector to its negative and fixes a hyperplane pointwise.  We may choose two (mutually inverse) normal vectors to every such hyperplane and obtain the \defn{root system} associated with $W$.  Moreover, by choosing an appropriate half-space, we may partition the root system into \defn{positive} and \defn{negative} roots; the collection of positive roots is denoted by $\Phi_{W}^{+}$. This sets up a bijection $t\mapsto \alpha_{t}$ from $T$ to $\Phi_{W}^{+}$. Let us, conversely, denote by $t_{\alpha}$ the reflection corresponding to the positive root $\alpha\in\Phi_{W}^{+}$.

Let $\wb=a_{1}a_{2}\cdots a_{k}$ be an $S$-reduced word for $w\in W$. For $i\in[k]$, we define $t_{i}\defs a_{k}a_{k-1}\cdots a_{k-i+1}\cdots a_{k-1}a_{k}$. Then, as explained in \cite[Section~1.3]{bjorner05combinatorics}, 
\begin{displaymath}
    \invset(w) = \{t_{1},t_{2},\ldots,t_{k}\}.
\end{displaymath}
In fact, we obtain a \textit{linear order} on $\invset(w)$ in this way, called the \defn{inversion order} of $w$, denoted by $\invorder(\wb)$.

\begin{remark}
	Let us emphasize that throughout this article we use letters in a normal font (\eg $w$) for group elements, while boldface letters (\eg $\wb$) indicate words. 
    Likewise, $\invset(w)$ is a \emph{set} of inversions, while $\invorder(\wb)$ indicates a \emph{tuple} (or linear order) of the inversions of $w$ determined by the $S$-reduced word $\wb$.
\end{remark}

The next definition is central for this article.

\begin{definition}[{\cite[Definition~43]{muehle19tamari}}]\label{def:w_alignment}
    An element $x\in W$ is \defn{$\wb$-aligned} if $x\weakorder w$ and whenever $t_{\alpha}\prec t_{a\alpha+b\beta}\prec t_{\beta}$ with respect to $\invorder(\wb)$ for $a,b\in\mathbb{N}$, then $t_{a\alpha+b\beta}\in\covset(x)$ implies $t_{\alpha}\in\invset(x)$.
\end{definition}

Notice that the condition ``$t_{\alpha}\prec t_{a\alpha+b\beta}\prec t_{\beta}$ with respect to $\invorder(\wb)$'' in Definition~\ref{def:w_alignment} requires that $\{t_{\alpha},t_{a\alpha+b\beta},t_{\beta}\}\subseteq\invset(w)$.  We denote the set of $\wb$-aligned elements of $W$ by $\Align(W,\wb)$. We shall see examples of aligned elements when $W=B_n$ (see Example~\ref{ex-B-aligned}).

\begin{remark}\label{rem:cover_vs_inversions}
    Note that, if we plug in $w=\wo$ in Definition~\ref{def:w_alignment}, then we obtain the usual definition of Coxeter-alignment as in \cite[Section~4]{reading07clusters} except that we check only cover inversions (instead of arbitrary inversions). This difference plays a crucial role when establishing a ``Coxeter-Catalan''-like theory for parabolic quotients of Coxeter groups. We give a more concrete illustration in Example~\ref{ex:aligned_elements_subtlety}.
\end{remark}

\section{The Coxeter group of type $B$}

\subsection{A permutation representation}
    \label{sec:sign_symmetric_permutations}
The finite Coxeter groups were classified by Coxeter~\cite{coxeter35complete}, and we are mainly concerned with the Coxeter groups of type $B$. If we consider the set $S=\{s_{0},s_{1},\ldots,s_{n-1}\}$, then the group
\begin{displaymath}
    \Hyper_{n} \defs \bigl\langle S\colon s_{i}^{2}=\id, (s_{0}s_{1})^{4}=\id, (s_{i},s_{i+1})^{3}=\id\;\text{for}\;i>0, (s_{i}s_{j})^{2}=\id\;\text{for}\;j>i+1\bigr\rangle.
\end{displaymath}
is the \defn{Coxeter group of type $B$}.

Combinatorially, we may realize $\Hyper_{n}$ as follows. For $n>0$ we define the set of \defn{signed integers} by
\begin{displaymath}
    \pm[n] \defs \{{-}n,-{n}{+}1,\ldots,{-}2,{-}1,1,2,\ldots,n{-}1,n\}.
\end{displaymath}
A permutation $\pi$ of $\pm[n]$ is \defn{sign-symmetric} if $\pi(-i)=-\pi(i)$ for all $i\in[n]$. Then, equivalently, $\Hyper_{n}$ is the group of all sign-symmetric permutations of $\pm[n]$; we usually call it the \defn{hyperoctahedral group} of degree $n$. It is easily checked that $\bigl\lvert\Hyper_{n}\bigr\rvert=2^{n}n!$.

Let $\sleft i\sright$ denote the sign-symmetric permutation that exchanges $i$ and $-i$, and $\lleft i\;j\rright$ the sign-symmetric permutation that exchanges the values $i$ and $j$ (and simultaneously $-i$ and $-j$), where we do not make assumptions about the sign of $i$ and $j$.  Then, the assignment $s_{0}\mapsto\sleft 1\sright$ and $s_{i}\mapsto\lleft i\;i{+}1\rright$ for $i\in[n-1]$ allows us to switch from the Coxeter representation of $\Hyper_{n}$ to its permutation representation (when multiplying by $s_i$ on the left of a sign-symmetric permutation). Via the induced isomorphism, we may transfer the notions of reflection set, Coxeter length, (cover) inversions, and weak order to sign-symmetric permutations.

\begin{note}
    The set of reflections of $\Hyper_{n}$ can be partitioned into three types: the ones that exchange $i$ with $-i$, the ones that exchange two positive values (and their opposite values), and the ones that exchange a positive and a negative value (and their opposite).
    They will respectively be denoted by $\sleft i\sright$, $\lleft i\;j\rright$, and $\lleft -j\;i \rright$. Since changing the signs of all elements or exchanging both elements represent the same reflection, we shall write the reflections in the following non-ambiguous way
    \begin{equation}\label{eq:reflection_convention}
        \begin{split}
            \sleft i\sright &\text{ with } i>0, \\
            \lleft i\;j\rright &\text{ with } 0<i<j, \\
            \lleft -j\;i\rright &\text{ with } 0<i<j. \\
        \end{split}
    \end{equation}

    Note that, by definition, every reflection is itself an involution.
\end{note}

\begin{lemma}\label{lem:type_b_inversions}
    Let $\pi\in\Hyper_{n}$ and let $i,j\in[n]$.  Then
    \begin{itemize}
        \item $\sleft i\sright\in\invset(\pi)$ if and only if $\pi(i)<0$;
        \item $\lleft i\;j\rright\in\invset(\pi)$ with $i<j$ if and only if $\pi(i)>\pi(j)$;
        \item $\lleft {-}j\;i\rright\in\invset(\pi)$ with $i<j$ if and only if $\pi(-j)>\pi(i)$.
    \end{itemize}
    The length of $\pi$ is as usual its number of inversions.
    Moreover, 
    \begin{itemize}
        \item $\sleft i\sright\in\covset(\pi)$ if and only if $\pi(i)=-1$;
        \item $\lleft i\;j\rright\in\covset(\pi)$ with $i<j$ if and only if $\pi(i)=\pi(j)+1$;
        \item $\lleft {-}j\;i\rright\in\covset(\pi)$ with $i<j$ if and only if $\pi(-j)=\pi(i)+1$.
    \end{itemize}
\end{lemma}
\begin{proof}
    For the definition of the (right) inversions, consider a sign-symmetric permutation $\pi$. Then one easily checks that multiplying it by a generator $s_i$ on its right either adds or subtracts one to its length. Moreover, if its length is not zero, there is always a way to subtract one to its length by applying a suitable generator, that is, if there is an inversion between $i$ and $j$, then there exists an index $k$ such that $k$ and $k+1$ is an inversion.

    Since composing $\pi\in\Hyper_{n}$ with $\sleft i\sright$ on the right effectively swaps the sign of $\pi(i)$; composing $\pi\in\Hyper_{n}$ with $\lleft i\;j\rright$ effectively exchanges the entries of $\pi$ in the $i\th$ and $j\th$ position, the first part of the claim holds.

    Now, concerning the cover inversions, one just has to decode the action of the $s_i$ on the \textit{left} of a sign-symmetric permutation: $s_0$ exchanges the values $1$ and $-1$, and $s_i$ exchanges the values $i$ and $i+1$ along with $-i$ and $-i-1$, hence the statement.
\end{proof}

\begin{corollary}\label{cor-si-less-invs}
    Let $\pi\in\Hyper_n$. Then $\ell_{S}(\pi s_{i})<\ell_{S}(\pi)$ if and only if
    \begin{itemize}
        \item $i=0$ and $\pi(1)<0$; or
        \item $i\neq 0$ and either $\pi(i)>0$ and $\pi(i)>\pi(i+1)$, or $\pi(i)<0$ and $\pi(i)<\pi(i+1)$.
    \end{itemize}
\end{corollary}

We usually represent sign-symmetric permutations via their \defn{long one-line notation}, \ie for $\pi\in\Hyper_{n}$ we write down the values
\begin{displaymath}
    \pi({-}n),\pi({-}n{+}1),\ldots,\pi({-}1),\pi(1),\ldots,\pi(n{-}1),\pi(n)
\end{displaymath}
in that order. We represent negative values by an overbar rather than a minus sign for stylistic reasons, and we add a vertical bar between $\pi(-1)$ and $\pi(1)$ to emphasize the symmetry.  If no ambiguities may arise, we omit the commas in this sequence. The \defn{right part} of $\pi\in\Hyper_{n}$ is the sequence $\pi(1),\pi(2),\ldots,\pi(n)$.  Clearly, the right part determines $\pi$ due to the sign-symmetry.
We shall make use of the following notation:
\begin{displaymath}
    i^{+} \defs \begin{cases} i+1, & \text{if}\;i\neq\overline{1},\\ 1, & \text{otherwise.}\end{cases}
\end{displaymath}
One easily finds the cover inversions of a permutation in this notation: start with $\overline1$ and then read $1$, $2$, and so on. If $i^+$ is to the left of $i$, their positions form a cover inversion.

\begin{example}
    Consider $\pi\in\Hyper_{n}$ given by
    \begin{displaymath}
        \pi = \overline{9}\;\overline{7}\;\overline{8}\;\overline{5}\;\overline{6}\;1\;\overline{3}\;\overline{4}\;2\mid \overline{2}\;4\;3\;\overline{1}\;6\;5\;8\;7\;9.
    \end{displaymath}
    One easily checks that $\pi=s_{0}s_{1}s_{2}s_{7}s_{5}s_{3}s_{2}s_{0}$. Then, 
    \begin{displaymath}
        \invset(w) = \Bigl\{\sleft 1\sright,\sleft 4\sright,\lleft 2\;3\rright, \lleft 2\;4\rright,\lleft 3\;4\rright,\lleft 5\;6\rright, \lleft 7\;8\rright,\lleft {-4}\;1\rright\Bigr\},
    \end{displaymath}
    and
    \begin{displaymath}
        \covset(w) = \Bigl\{\sleft 4\sright,\lleft 2\;3\rright,\lleft 5\;6\rright,\lleft 7\;8\rright\Bigr\}.
    \end{displaymath}
\end{example}

\subsection{Roots for the Coxeter group of type $B$}
    \label{sec:type_b_roots}
For $n>0$, and $i\in[n]$, we denote by $\varepsilon_{i}$ the $i\th$ unit vector in $\mathbb{R}^{n}$.  Then, it is well-known that the following is indeed a root system for the Coxeter group of type $B$:
\begin{displaymath}
    \Phi = \bigl\{\pm\varepsilon_{i}\colon i\in[n]\bigr\} \uplus \bigl\{\pm(\varepsilon_{i}\pm\varepsilon_{j})\colon 1\leq i<j\leq n\bigr\},
\end{displaymath}
see for instance \cite[Section~5.3]{grove85finite} for a detailed explanation.  Let us consider the set
\begin{displaymath}
    \Phi^{+} \defs \bigl\{\varepsilon_{i}\colon i\in[n]\bigr\} \uplus \bigl\{-\varepsilon_{i}+\varepsilon_{j}\colon 1\leq i<j\leq n\bigr\}
     \uplus\bigl\{\varepsilon_{i}+\varepsilon_{j}\colon 1\leq i<j\leq n\bigr\}.
\end{displaymath}
Then, the subset
\begin{displaymath}
    \Pi = \bigl\{\varepsilon_{1},-\varepsilon_{1}+\varepsilon_{2},-\varepsilon_{2}+\varepsilon_{3},\ldots,-\varepsilon_{n-1}+\varepsilon_{n}\bigr\}
\end{displaymath}
of \defn{simple} roots consists of linearly independent vectors and spans $\Phi^{+}$. In particular, expanding any $\alpha\in\Phi^{+}$ in terms of $\Pi$ yields a linear combination with nonnegative coefficients. This is enough to conclude that $\Phi^{+}$ is indeed a choice of positive roots for $\Hyper_{n}$. We may equivalently see this by observing that if $\mathbf{v}=(1,2,\ldots,n)\in\mathbb{R}^{n}$, then any $\alpha\in\Phi^{+}$ has $(\alpha,\mathbf{v})>0$ and any $\alpha\in\Phi\setminus\Phi^{+}$ has $(\alpha,\mathbf{v})<0$ with respect to the standard scalar product of $\mathbb{R}^{n}$.

In the following lemma, we abbreviate $\alpha_{i,j}\defs -\varepsilon_{i}+\varepsilon_{j}$ and $\overline{\alpha}_{i,j}\defs\varepsilon_{i}+\varepsilon_{j}$.

\begin{lemma}\label{lem:type_b_decompositions}
    Let $\alpha\in\Phi_{+}$, and let $i,k\in[n]$ with $i<k$.  The following list of decompositions of $\alpha$ as nonnegative linear combinations of two positive roots is exhaustive.
    \begin{itemize}
        \item If $\alpha=\varepsilon_{i}$, then 
            \begin{displaymath}
                \alpha = \alpha_{j,i}+\varepsilon_{j}\quad\text{for}\;1\leq j<i.
            \end{displaymath}
        \item If $\alpha=\alpha_{i,k}$, then
            \begin{displaymath}
                \alpha = \alpha_{i,j}+\alpha_{j,k}\quad\text{for}\;i<j<k.
            \end{displaymath}
        \item If $\alpha=\overline{\alpha}_{i,k}$, then 
            \begin{displaymath}
                \alpha = \begin{cases}
                    \varepsilon_{i}+\varepsilon_{k}, & \\
                    \overline{\alpha}_{i,j}+\alpha_{j,k}, & \text{for}\;1\leq j<k,\\
                    \alpha_{j,i}+\overline{\alpha}_{j,k}, & \text{for}\;1\leq j<i.
                \end{cases}
            \end{displaymath}
    \end{itemize}
\end{lemma}
\begin{proof}
This is a straightforward computation. Note that for $i=j$, we get $\overline{\alpha}_{i,j}=2\varepsilon_{i}$ in the second decomposition of the third case.
\end{proof}

In order to match $\Phi^{+}$ with the reflections of $\Hyper_{n}$, we identify a positive root $\varepsilon_{i}$ with the reflection $\sleft i\sright$, a positive root $-\varepsilon_{i}+\varepsilon_{j}$ with the reflection $\lleft i\;j\rright$, and a positive root $\varepsilon_{i}+\varepsilon_{j}$ with the reflection $\lleft {-}j\;i\rright$.

\subsection{A special family of aligned sign-symmetric permutations}
    \label{sec:type_b_alignment}
We consider the \defn{linear Coxeter element} $\linc\defs s_{0}s_{1}\cdots s_{n-1}$, whose long one-line notation is
\begin{displaymath}
	1,\overline n,\ldots,\overline3,\overline2\mid 2,3,\ldots,n,\overline1.
\end{displaymath}
The \defn{$\linc$-sorting word} of $w\in\Hyper_{n}$ is the $S$-reduced word for $w$ that appears as rightmost as possible in the half-infinite word
\begin{displaymath}
    ~^{\infty}\linc \defs \ldots \mid s_{n-1}\cdots s_{1}s_{0}\mid s_{n-1}\cdots s_{1}s_{0}
\end{displaymath}
consisting of infinitely many copies of the reverse of $\linc$. The $\linc$-sorting word of $w$ is denoted by $\wb(\linc)$.  In order to describe the inversion order of $\wbo(\linc)$, we prove the following general lemma regarding $\linc$-sorting words of particular elements of $\Hyper_{n}$. For $i<k$, let us abbreviate
\begin{displaymath}
    s_{k\dots i} \defs s_{k}s_{k-1}\cdots s_{i+1}s_{i}
\end{displaymath}
and
\begin{displaymath}
    s_{i\dots k} \defs s_{i}s_{i+1}\cdots s_{k-1}s_{k}.
\end{displaymath}
Moreover, we use the notation $[i,k]\defs\{i,i{+}1,\ldots,k\}$.

\begin{lemma}\label{lem:sorting_word_negative}
    If the right part of a sign-symmetric permutation $w\in\Hyper_{n}$ can be written as the concatenation of intervals in the following form
    \begin{equation}
        [\overline{i_1{+}i_2},\overline{i_1{+}1}], [\overline{i_1{+}i_2{+}i_3},\overline{i_1{+}i_2{+}1}]\dots [\overline{i_1{+}\cdots{+}i_r},\overline{i_1{+}\cdots{+}i_{r-1}{+}1}], [1,i_1],
    \end{equation}
    then its $\linc$-sorting word is obtained by reading the following list from top to bottom, left to right:
    \begin{equation}\label{prods}\begin{aligned}
        & s_{n-i_{r}\dots 0} && s_{n-i_{r}+1\dots 0} && \cdots && s_{n-1\dots 0}\\
        & s_{n-i_{r-1}\dots 0} && s_{n-i_{r-1}+1\dots 0} && \cdots && s_{n-1\dots 0}\\
        & \cdots\\
        & s_{n-i_{2}\dots 0} && s_{n-i_{2}+1\dots 0} && \cdots && s_{n-1\dots 0}.
    \end{aligned}\end{equation}
\end{lemma}
\begin{proof}
    Given that $~^{\infty}\linc$ is the repetition of the same sequence $s_{n-1}\cdots s_1s_0$, the rightmost reduced word of $w$ in it is obtained by computing, given a sign-symmetric permutation $\pi$ and a value $j$, the first value $k$ within the ordered sequence $j,\ldots,n-1,1,\ldots,j-1$ such that $\ell_{S}(\pi s_k)<\ell_{S}(\pi)$. The induction begins with the pair $(w,0)$.

    Now, the property holds by an easy induction on $r$: one gets again such an element by multiplying $w$ on the right by the sequence of $s_i$ of the last line of~\eqref{prods} read from right to left. This product is the rightmost one in $~^{\infty}\linc$ since each product $s_{0\ldots n-j}$ first changes the sign of the first letter then moves it to the right until it is followed by a letter greater than itself or at the right extremity of the permutation. Since at the right of its final position we have an increasing sequence, the next possible generator we can use is $s_0$, hence the induction.

    Note that each product by an $s_k$ either changes the first value into a positive one or exchanges it with a letter smaller than itself so that the length of the element decreases at each step, hence showing that the final result is indeed a reduced word of $w$.
\end{proof}

\begin{example}\label{ex:longest_word_3121}
    Consider the sign-symmetric permutation $w\in\Hyper_{7}$ given by its right part as follows:
    \begin{displaymath}
        \overline{3}\;\overline{2}\;\overline{1}\;\overline{4}\;\overline{6}\;\overline{5}\;\overline{7} = [\overline{3},\overline{1}],[\overline{4},\overline{4}],[\overline{6},\overline{5}],[\overline{7},\overline{7}].
    \end{displaymath}
    The parameters used in Lemma~\ref{lem:sorting_word_negative} are $r=5$, $i_{1}=0$, $i_{2}=3$, $i_{3}=1$, $i_{4}=2$, $i_{5}=1$.  We get
    \begin{displaymath}
        \wb(\linc) = s_{6\dots 0} s_{5\dots 0} s_{6\dots 0} s_{6\dots 0} s_{4\dots 0} s_{5\dots 0} s_{6\dots 0},
    \end{displaymath}
    since
    \begin{align*}
        (\overline{3}\;\overline{2}\;\overline{1}\;\overline{4}\;\overline{6}\;\overline{5}\;\overline{7}) & 
        s_{0\dots6} s_{0\dots5}s_{0\dots4}s_{0\dots6}s_{0\dots6}s_{0\dots5}s_{0\dots6}\\
        & = (\overline{2}\;\overline{1}\;\overline{4}\;\overline{6}\;\overline{5}\;\overline{7}\;3) s_{0\dots5}s_{0\dots4}s_{0\dots6}s_{0\dots6}s_{0\dots5}s_{0\dots6} \\
        & = (\overline{1}\;\overline{4}\;\overline{6}\;\overline{5}\;\overline{7}\;2\;3) s_{0\dots4}s_{0\dots6}s_{0\dots6}s_{0\dots5}s_{0\dots6} \\
        & = (\overline{4}\;\overline{6}\;\overline{5}\;\overline{7}\;1\;2\;3) s_{0\dots6}s_{0\dots6}s_{0\dots5}s_{0\dots6} \\
        & = (\overline{6}\;\overline{5}\;\overline{7}\;1\;2\;3\;4) s_{0\dots6}s_{0\dots5}s_{0\dots6} \\
        & = (\overline{7}\;1\;2\;3\;4\;5\;6) s_{0\dots6} \\
        & = (1\;2\;3\;4\;5\;6\;7),
    \end{align*}
    as one can check by hand (or computer).
\end{example}

\begin{corollary}\label{cor:linc_sorting_longest}
    The $\linc$-sorting word of the longest element $\wo\in\Hyper_{n}$ is given by $\wbo(\linc)=(s_{n-1}\cdots s_{1}s_{0})^{n}$.
\end{corollary}
\begin{proof}
    Note that $\wo\in\Hyper_{n}$ is determined by the right part $\overline{1}\;\overline{2}\;\ldots\;\overline{n}$.  Thus, we may apply Lemma~\ref{lem:sorting_word_negative} with $r=n+1$, $i_{1}=0$ and $i_{2}=i_{3}=\cdots=i_{r}=1$.  Then, each line of \eqref{prods} consists of the word $s_{n-1\ldots 0}$ and there is a total of $n$ rows.
\end{proof}

Starting from the $\linc$-sorting word of $\wo\in\Hyper_{n}$ we can explicitly compute the inversion order $\invorder\bigl(\wbo(\linc)\bigr)$.  We defer a formal proof to Section~\ref{sec:parabolic_longest_elements}.

For $i\in[n]$, we consider the ordered list of pairs
\begin{displaymath}\begin{aligned}
    & (i,i{-}1), && \ldots, && (i,2), && (i,1), && (i,{-}n), && \ldots, && (i,{-}i{-}1), && (i,{-}i).
\end{aligned}\end{displaymath}
The $n$ possible rows are now arranged such that first components increase from top to bottom and pairs with the same second component appear in the same column.  If we now read from top to bottom, right to left, then we obtain the inversion order $\invorder\bigl(\wbo(\linc)\bigr)$ by identifying a pair $(i,{-}i)$ with the reflection $\sleft i\sright$, a pair $(i,j)$ for $i>j>0$ with the transposition $\lleft j\;i\rright$ and a pair $(i,{-}j)$ for $i<j$ with the transposition $\lleft {-}j\;i\rright$.
\begin{example}
    Let $n=4$.  Then $\linc=s_{0}s_{1}s_{2}s_{3}$ and
    \begin{displaymath}
        \wbo(\linc) = s_{3}s_{2}s_{1}s_{0}s_{3}s_{2}s_{1}s_{0}s_{3}s_{2}s_{1}s_{0}s_{3}s_{2}s_{1}s_{0}.
    \end{displaymath}
    The previously mentioned arrangement of the sixteen possible pairs is:
    \begin{displaymath}\begin{aligned}
        & && && && (1,-4) && (1,-3) && (1,-2) && (1,-1)\\
        & && && (2,1) && (2,-4) && (2,-3) && (2,-2)\\
        & && (3,2) && (3,1) && (3,-4) && (3,-3)\\
        & (4,3) && (4,2) && (4,1) && (4,-4)
    \end{aligned}\end{displaymath}
    The replacement from above produces the following inversion order:
    \begin{displaymath}\begin{aligned}
        && && && && && \succ && \lleft -4\;1\rright && \succ && \lleft -3\;1\rright && \succ && \lleft -2\;1\rright && \succ && \sleft 1\sright \\
        && && && \succ && \lleft 1\;2\rright && \succ && \lleft -4\;2\rright && \succ && \lleft -3\;2\rright && \succ && \sleft 2\sright \\
        && \succ && \lleft 2\;3\rright && \succ && \lleft 1\;3\rright && \succ && \lleft -4\;3\rright && \succ && \sleft 3\sright \\
        \lleft 3\;4\rright && \succ && \lleft 2\;4\rright && \succ && \lleft 1\;4\rright && \succ && \sleft 4\sright  \\
    \end{aligned}\end{displaymath}
\end{example}

We may thus characterize the members of $\Align\bigl(\Hyper_{n},\wbo(\linc)\bigr)$ via their inversion sets.

\begin{lemma}\label{lem:type_b_forcing}
    Let $\pi\in\Hyper_{n}$. Then $\pi\in\Align\bigl(\Hyper_{n},\wbo(\linc)\bigr)$ if and only if for every $1\leq i<k\leq n$ it holds that
    \begin{itemize}
        \item if $\sleft i\sright\in\covset(\pi)$, then $\sleft j\sright\in\invset(\pi)$ for all $1\leq j<i$;
        \item if $\lleft i\;k\rright\in\covset(\pi)$, then $\lleft i\;j\rright\in\invset(\pi)$ for all $i<j<k$;
        \item if $\lleft {-}k\;i\rright\in\covset(\pi)$, then 
        \begin{itemize}
            \item $\sleft i\sright\in\invset(\pi)$,
            \item $\lleft {-}j\;i\rright\in\invset(\pi)$ for all $1\leq j<k$, $j\neq i$,
            \item $\lleft {-}k\;j\rright\in\invset(\pi)$ for all $1\leq j<i$.
        \end{itemize}
    \end{itemize}
\end{lemma}
\begin{proof}
    Let $t$ be a reflection of $\Hyper_{n}$ such that $\alpha_{t}$ can be written as $\alpha_{t}=b_{1}\beta_{1}+b_{2}\beta_{2}$ for positive roots $\beta_{1},\beta_{2}$ and positive integers $b_{1},b_{2}$. Assume further that $t_{\beta_{1}}\prec t\prec t_{\beta_{2}}$ with respect to $\invorder\bigl(\wbo(\linc)\bigr)$. Then, by Definition~\ref{def:w_alignment}, it is enough to show that $t\in\covset(\pi)$ implies $t_{\beta_{1}}\in\invset(\pi)$.

    \medskip

    We use Lemma~\ref{lem:type_b_decompositions} and the description of $\invorder\bigl(\wbo(\linc)\bigr)$ from above. To avoid lengthy conversions from reflections to roots, we use the notation ``$t_{1}+t_{2}$'' as a short-hand for ``$t_{\alpha_{t_{1}}+\alpha_{t_{2}}}$''. We consider the possible cases, and pick $i,k\in[n]$ with $i<k$.

    \begin{enumerate}[(i)]
        \item If $t=\sleft i\sright$, then by Lemma~\ref{lem:type_b_decompositions}, we have $t=\lleft j\;i\rright+\sleft j\sright$ for $0<j<i$.  In the inversion order, we clearly have $\sleft j\sright\prec\lleft j\;i\rright$.
        \item If $t=\lleft i\;k\rright$, then by Lemma~\ref{lem:type_b_decompositions}, we have $t=\lleft i\;j\rright+\lleft j\;k\rright$ for $i<j<k$.  In the inversion order, we clearly have $\lleft i\;j\rright\prec\lleft j\;k\rright$.
        \item If $t=\lleft {-}k\;i\rright$, then by Lemma~\ref{lem:type_b_decompositions}, we have the following three options.
        \begin{enumerate}[(a)]
            \item Say $t=\sleft i\sright+\sleft k\sright$. In the inversion order, we have $\sleft i\sright\prec\sleft k\sright$.
            \item Say $t=\lleft {-}j\;i\rright+\lleft j\;k\rright$ for $1\leq j<k$.  If $j=i$, then we actually have $t=2\sleft i\sright+\lleft i\;k\rright$.  In the inversion order, we have $\sleft i\sright\prec\lleft i\;k\rright$.  If $j\neq i$, then we have $\lleft{-}j\;i\rright\prec\lleft j\;k\rright$ in the inversion order.
            \item Say $t=\lleft j\;i\rright+\lleft {-}k\;j\rright$ for $1\leq j<i$. Since $j<i<k$, we get $\lleft j\;i\rright\prec\lleft {-}k\;j\rright$ in the inversion order.
        \end{enumerate}
    \end{enumerate}

    These considerations show that if $t\in\covset(\pi)$, then $\invset(\pi)$ must contain the specified inversions for $\pi$ to be $\wbo(\linc)$-aligned.
\end{proof}

Using the permutation representation of $\Hyper_{n}$, we can identify $\wbo(\linc)$-aligned elements in terms of pattern avoidance. A permutation $\pi\in\Hyper_{n}$ has a \defn{type-$B$ $231$-pattern} if there exist indices $-n\leq i<j<k\leq n$ such that $\pi(i)=\pi(k)^+$ as well as $\pi(j)>\pi(i)$ and $j,k>0$. In other words, we look for subwords of the long one-line notation of $\pi$ standardizing to $231$ with the requirement that the positions of the '$2$' and the '$1$' form a cover inversion and and the positions of the '$3$' and the '$1$' are positive. Let $\Hyper_{n}(231)$ be the set of sign-symmetric permutations without a type-$B$ $231$-pattern.

\begin{example}\label{ex-B-aligned}
    Let $n=5$ and consider the sign-symmetric permutation
    \begin{displaymath}
        \pi = \overline{4}\;\overline{3}\;\overline{5}\;1\;2\mid\overline{2}\;\overline{1}\;5\;3\;4.
    \end{displaymath}
    We get
    \begin{displaymath}
        \invset(\pi) = \Bigl\{\sleft 1\sright,\sleft 2\sright,\lleft {-}2\;1\rright,\lleft 3\;4\rright,\lleft 3\;5\rright\Bigr\}\quad\text{and}\quad\covset(\pi) = \Bigl\{\sleft 2\sright,\lleft 3\;5\rright\Bigr\}.
    \end{displaymath}
    Comparing this with Lemma~\ref{lem:type_b_forcing}, we note that $\pi$ is $\wbo(\linc)$-aligned, and that $\pi$ does not have any type-$B$ $231$-pattern. There are three ``usual'' $231$ patterns, determined by the positions $(\overline{4},\overline{3},\overline{5})$, $(1,2,\overline{2})$ and $(1,2,\overline{1})$. However, in each case the middle position is negative, and therefore none of them is a type-$B$ $231$-pattern.

    \medskip

    On the other hand, consider the sign-symmetric permutation
    \begin{displaymath}
        \sigma = 2\;\overline{3}\;5\;\overline{1}\;\hgl{\overline{4}} \mid \hgl{4}\;1\;\hgl{\overline{5}}\;3\;\overline{2}.
    \end{displaymath}
    We get
    \begin{align*}
        \invset(\sigma)
        & = \Bigl\{\lleft 1\;2\rright, \lleft 1\;3\rright, \lleft 1\;4\rright,\lleft 1\;5\rright, \lleft 2\;3\rright, \lleft 2\;5\rright,\lleft 4\;5\rright, \lleft {-}3\;1\rright, \lleft {-}3\;2\rright,\\
        & \kern1cm \lleft {-}3\;4\rright, \lleft {-}3\;5\rright,
              \lleft {-}5\;2\rright, \sleft 3\sright, \sleft 5\sright\Bigr\},\\
        \covset(\sigma)
        & = \Bigl\{\lleft 1\;4\rright, \lleft {-}3\;1\rright,\lleft {-}5\;2\rright\Bigr\}.
    \end{align*}
    Comparing this with Lemma~\ref{lem:type_b_forcing}, we observe that $\sigma$ is not $\wbo(\linc)$-aligned, because for instance $\lleft {-3}\;1\rright\in\covset(v)$ but $\sleft 1\sright\notin\invset(v)$. This manifests in the type-$B$ $231$-pattern in positions $(-1,1,3)$ highlighted above in boldface, because $\sigma(-1)=-4=-5+1=\sigma(3)+1$ (thus $(-1,3)$ is a cover inversion of $\sigma$) and $\sigma(1)=4>-4=\sigma(-1)$ (thus $(-1,1)$ is not an inversion of $\sigma$).
\end{example}

\begin{lemma}[{\cite[Lemma~4.9]{reading07clusters}}]\label{lem:type_b_alignment}
    A sign-symmetric permutation $\pi\in\Hyper_{n}$ is $\wbo(\linc)$-aligned if and only if $\pi$ does not have a type-$B$ $231$-pattern.
\end{lemma}
\begin{proof}
    This is immediate from Lemmas~\ref{lem:type_b_inversions} and \ref{lem:type_b_forcing}.
\end{proof}

The weak order on the set $\Hyper_{n}(231)$ inherits some remarkable properties from the weak order on the full group $\Hyper_{n}$.  Let us call the poset
\begin{displaymath}
    \Tamari_{B}(n)\defs\Weak\bigl(\Hyper_{n}(231)\bigr)
\end{displaymath}
the \defn{type-$B$ Tamari lattice}.

Figure~\ref{fig:weak_order_b3} shows the lattice $\Weak\bigl(\Hyper_{3}\bigr)$ where we have highlighted the intervals characterized by the projection map that sends an element to the largest aligned permutation smaller than itself.  Figure~\ref{fig:tamari_b3} shows $\Tamari_{B}(3)$, which is the quotient lattice induced by the highlighted intervals in Figure~\ref{fig:weak_order_b3}.  In these figures, we have colored the positions in the long one-line notation by symmetrically placed colors.  This foreshadows our representation of the elements in parabolic quotients of $\Hyper_{n}$ in the next section.

\begin{sidewaysfigure}
    \vspace*{12cm}
    \begin{tikzpicture}[scale=.7]
		\def\x{3};
		\def\y{2};
		\def\s{.7};
		\def\t{.8};	
        \coordinate(n1) at (6*\x,1*\y);
		\coordinate(n2) at (5*\x,2*\y);
		\coordinate(n3) at (6*\x,2*\y);
		\coordinate(n4) at (7*\x,2*\y);
		\coordinate(n5) at (4*\x,3*\y);
		\coordinate(n6) at (5*\x,3*\y);
		\coordinate(n7) at (6*\x,3*\y);
		\coordinate(n8) at (7*\x,3*\y);
		\coordinate(n9) at (8*\x,3*\y);
		\coordinate(n10) at (3*\x,4*\y);
		\coordinate(n11) at (4*\x,4*\y);
		\coordinate(n12) at (5*\x,4*\y);
		\coordinate(n13) at (6*\x,4*\y);
		\coordinate(n14) at (7*\x,4*\y);
		\coordinate(n15) at (8*\x,4*\y);
		\coordinate(n16) at (9*\x,4*\y);
		\coordinate(n17) at (2*\x,5*\y);
		\coordinate(n18) at (3*\x,5*\y);
		\coordinate(n19) at (4*\x,5*\y);
		\coordinate(n20) at (5*\x,5*\y);
		\coordinate(n21) at (6*\x,5*\y);
		\coordinate(n22) at (7*\x,5*\y);
		\coordinate(n23) at (8*\x,5*\y);
		\coordinate(n24) at (9*\x,5*\y);
		\coordinate(n25) at (1*\x,6*\y);
		\coordinate(n26) at (2*\x,6*\y);
		\coordinate(n27) at (3*\x,6*\y);
		\coordinate(n28) at (4*\x,6*\y);
		\coordinate(n29) at (5*\x,6*\y);
		\coordinate(n30) at (6*\x,6*\y);
		\coordinate(n31) at (7*\x,6*\y);
		\coordinate(n32) at (8*\x,6*\y);
		\coordinate(n33) at (1*\x,7*\y);
		\coordinate(n34) at (2*\x,7*\y);
		\coordinate(n35) at (3*\x,7*\y);
		\coordinate(n36) at (4*\x,7*\y);
		\coordinate(n37) at (5*\x,7*\y);
		\coordinate(n38) at (6*\x,7*\y);
		\coordinate(n39) at (7*\x,7*\y);
		\coordinate(n40) at (2*\x,8*\y);
		\coordinate(n41) at (3*\x,8*\y);
		\coordinate(n42) at (4*\x,8*\y);
		\coordinate(n43) at (5*\x,8*\y);
		\coordinate(n44) at (6*\x,8*\y);
		\coordinate(n45) at (3*\x,9*\y);
		\coordinate(n46) at (4*\x,9*\y);
		\coordinate(n47) at (5*\x,9*\y);
		\coordinate(n48) at (4*\x,10*\y);
		\draw[thick](n1) -- (n2);
		\draw[thick](n1) -- (n3);
		\draw[thick](n1) -- (n4);
		\draw[thick](n2) -- (n5);
		\draw[thick](n2) -- (n7);
		\draw[thick](n3) -- (n6);
		\draw[thick](n3) -- (n8);
		\draw[thick](n4) -- (n7);
		\draw[thick](n4) -- (n9);
		\draw[thick](n5) -- (n10);
		\draw[thick](n5) -- (n11);
		\draw[thick](n6) -- (n11);
		\draw[thick](n6) -- (n13);
		\draw[thick](n7) -- (n12);
		\draw[thick](n8) -- (n13);
		\draw[thick](n8) -- (n15);
		\draw[thick](n9) -- (n14);
		\draw[thick](n9) -- (n16);
		\draw[thick](n10) -- (n17);
		\draw[thick](n10) -- (n18);
		\draw[thick](n11) -- (n18);
		\draw[thick](n12) -- (n19);
		\draw[thick](n12) -- (n21);
		\draw[thick](n13) -- (n20);
		\draw[thick](n14) -- (n21);
		\draw[thick](n14) -- (n23);
		\draw[thick](n15) -- (n22);
		\draw[thick](n15) -- (n24);
		\draw[thick](n16) -- (n23);
		\draw[thick](n16) -- (n24);
		\draw[thick](n17) -- (n25);
		\draw[thick](n17) -- (n27);
		\draw[thick](n18) -- (n26);
		\draw[thick](n19) -- (n27);
		\draw[thick](n19) -- (n29);
		\draw[thick](n20) -- (n28);
		\draw[thick](n20) -- (n30);
		\draw[thick](n21) -- (n29);
		\draw[thick](n22) -- (n30);
		\draw[thick](n22) -- (n32);
		\draw[thick](n23) -- (n31);
		\draw[thick](n24) -- (n32);
		\draw[thick](n25) -- (n33);
		\draw[thick](n25) -- (n34);
		\draw[thick](n26) -- (n33);
		\draw[thick](n26) -- (n35);
		\draw[thick](n27) -- (n34);
		\draw[thick](n28) -- (n35);
		\draw[thick](n28) -- (n37);
		\draw[thick](n29) -- (n36);
		\draw[thick](n30) -- (n37);
		\draw[thick](n31) -- (n38);
		\draw[thick](n31) -- (n39);
		\draw[thick](n32) -- (n39);
		\draw[thick](n33) -- (n40);
		\draw[thick](n34) -- (n41);
		\draw[thick](n35) -- (n40);
		\draw[thick](n36) -- (n41);
		\draw[thick](n36) -- (n43);
		\draw[thick](n37) -- (n42);
		\draw[thick](n38) -- (n43);
		\draw[thick](n38) -- (n44);
		\draw[thick](n39) -- (n44);
		\draw[thick](n40) -- (n45);
		\draw[thick](n41) -- (n46);
		\draw[thick](n42) -- (n45);
		\draw[thick](n42) -- (n47);
		\draw[thick](n43) -- (n46);
		\draw[thick](n44) -- (n47);
		\draw[thick](n45) -- (n48);
		\draw[thick](n46) -- (n48);
		\draw[thick](n47) -- (n48);
		\draw(n1) node[scale=\s]{\AlphaPerm{1,1,1,1,1,1}{c3,c2,c1,c1,c2,c3}{\overline{3},\overline{2},\overline{1},1,2,3}{.5}{.25}{white!50!gray}{\t}};
		\draw(n2) node[scale=\s]{\AlphaPerm{1,1,1,1,1,1}{c3,c2,c1,c1,c2,c3}{\overline{2},\overline{3},\overline{1},1,3,2}{.5}{.25}{white!50!gray}{\t}};
		\draw(n3) node[scale=\s]{\AlphaPerm{1,1,1,1,1,1}{c3,c2,c1,c1,c2,c3}{\overline{3},\overline{1},\overline{2},2,1,3}{.5}{.25}{white!50!gray}{\t}};
		\draw(n4) node[scale=\s]{\AlphaPerm{1,1,1,1,1,1}{c3,c2,c1,c1,c2,c3}{\overline{3},\overline{2},1,\overline{1},2,3}{.5}{.25}{white!50!gray}{\t}};
		\draw(n5) node[scale=\s]{\AlphaPerm{1,1,1,1,1,1}{c3,c2,c1,c1,c2,c3}{\overline{1},\overline{3},\overline{2},2,3,1}{.5}{.25}{white!50!gray}{\t}};
		\draw(n6) node[scale=\s]{\AlphaPerm{1,1,1,1,1,1}{c3,c2,c1,c1,c2,c3}{\overline{2},\overline{1},\overline{3},3,1,2}{.5}{.25}{white!50!gray}{\t}};
		\draw(n7) node[scale=\s]{\AlphaPerm{1,1,1,1,1,1}{c3,c2,c1,c1,c2,c3}{\overline{2},\overline{3},1,\overline{1},3,2}{.5}{.25}{white!50!gray}{\t}};
		\draw(n8) node[scale=\s]{\AlphaPerm{1,1,1,1,1,1}{c3,c2,c1,c1,c2,c3}{\overline{3},1,\overline{2},2,\overline{1},3}{.5}{.25}{white!50!gray}{\t}};
		\draw(n9) node[scale=\s]{\AlphaPerm{1,1,1,1,1,1}{c3,c2,c1,c1,c2,c3}{\overline{3},\overline{1},2,\overline{2},1,3}{.5}{.25}{white!50!gray}{\t}};
		\draw(n10) node[scale=\s]{\AlphaPerm{1,1,1,1,1,1}{c3,c2,c1,c1,c2,c3}{1,\overline{3},\overline{2},2,3,\overline{1}}{.5}{.25}{white!50!gray}{\t}};
		\draw(n11) node[scale=\s]{\AlphaPerm{1,1,1,1,1,1}{c3,c2,c1,c1,c2,c3}{\overline{1},\overline{2},\overline{3},3,2,1}{.5}{.25}{white!50!gray}{\t}};
		\draw(n12) node[scale=\s]{\AlphaPerm{1,1,1,1,1,1}{c3,c2,c1,c1,c2,c3}{\overline{1},\overline{3},2,\overline{2},3,1}{.5}{.25}{white!50!gray}{\t}};
		\draw(n13) node[scale=\s]{\AlphaPerm{1,1,1,1,1,1}{c3,c2,c1,c1,c2,c3}{\overline{2},1,\overline{3},3,\overline{1},2}{.5}{.25}{white!50!gray}{\t}};
		\draw(n14) node[scale=\s]{\AlphaPerm{1,1,1,1,1,1}{c3,c2,c1,c1,c2,c3}{\overline{2},\overline{1},3,\overline{3},1,2}{.5}{.25}{white!50!gray}{\t}};
		\draw(n15) node[scale=\s]{\AlphaPerm{1,1,1,1,1,1}{c3,c2,c1,c1,c2,c3}{\overline{3},2,\overline{1},1,\overline{2},3}{.5}{.25}{white!50!gray}{\t}};
		\draw(n16) node[scale=\s]{\AlphaPerm{1,1,1,1,1,1}{c3,c2,c1,c1,c2,c3}{\overline{3},1,2,\overline{2},\overline{1},3}{.5}{.25}{white!50!gray}{\t}};
		\draw(n17) node[scale=\s]{\AlphaPerm{1,1,1,1,1,1}{c3,c2,c1,c1,c2,c3}{2,\overline{3},\overline{1},1,3,\overline{2}}{.5}{.25}{white!50!gray}{\t}};
		\draw(n18) node[scale=\s]{\AlphaPerm{1,1,1,1,1,1}{c3,c2,c1,c1,c2,c3}{1,\overline{2},\overline{3},3,2,\overline{1}}{.5}{.25}{white!50!gray}{\t}};
		\draw(n19) node[scale=\s]{\AlphaPerm{1,1,1,1,1,1}{c3,c2,c1,c1,c2,c3}{1,\overline{3},2,\overline{2},3,\overline{1}}{.5}{.25}{white!50!gray}{\t}};
		\draw(n20) node[scale=\s]{\AlphaPerm{1,1,1,1,1,1}{c3,c2,c1,c1,c2,c3}{\overline{1},2,\overline{3},3,\overline{2},1}{.5}{.25}{white!50!gray}{\t}};
		\draw(n21) node[scale=\s]{\AlphaPerm{1,1,1,1,1,1}{c3,c2,c1,c1,c2,c3}{\overline{1},\overline{2},3,\overline{3},2,1}{.5}{.25}{white!50!gray}{\t}};
		\draw(n22) node[scale=\s]{\AlphaPerm{1,1,1,1,1,1}{c3,c2,c1,c1,c2,c3}{\overline{2},3,\overline{1},1,\overline{3},2}{.5}{.25}{white!50!gray}{\t}};
		\draw(n23) node[scale=\s]{\AlphaPerm{1,1,1,1,1,1}{c3,c2,c1,c1,c2,c3}{\overline{2},1,3,\overline{3},\overline{1},2}{.5}{.25}{white!50!gray}{\t}};
		\draw(n24) node[scale=\s]{\AlphaPerm{1,1,1,1,1,1}{c3,c2,c1,c1,c2,c3}{\overline{3},2,1,\overline{1},\overline{2},3}{.5}{.25}{white!50!gray}{\t}};
		\draw(n25) node[scale=\s]{\AlphaPerm{1,1,1,1,1,1}{c3,c2,c1,c1,c2,c3}{3,\overline{2},\overline{1},1,2,\overline{3}}{.5}{.25}{white!50!gray}{\t}};
		\draw(n26) node[scale=\s]{\AlphaPerm{1,1,1,1,1,1}{c3,c2,c1,c1,c2,c3}{2,\overline{1},\overline{3},3,1,\overline{2}}{.5}{.25}{white!50!gray}{\t}};
		\draw(n27) node[scale=\s]{\AlphaPerm{1,1,1,1,1,1}{c3,c2,c1,c1,c2,c3}{2,\overline{3},1,\overline{1},3,\overline{2}}{.5}{.25}{white!50!gray}{\t}};
		\draw(n28) node[scale=\s]{\AlphaPerm{1,1,1,1,1,1}{c3,c2,c1,c1,c2,c3}{1,2,\overline{3},3,\overline{2},\overline{1}}{.5}{.25}{white!50!gray}{\t}};
		\draw(n29) node[scale=\s]{\AlphaPerm{1,1,1,1,1,1}{c3,c2,c1,c1,c2,c3}{1,\overline{2},3,\overline{3},2,\overline{1}}{.5}{.25}{white!50!gray}{\t}};
		\draw(n30) node[scale=\s]{\AlphaPerm{1,1,1,1,1,1}{c3,c2,c1,c1,c2,c3}{\overline{1},3,\overline{2},2,\overline{3},1}{.5}{.25}{white!50!gray}{\t}};
		\draw(n31) node[scale=\s]{\AlphaPerm{1,1,1,1,1,1}{c3,c2,c1,c1,c2,c3}{\overline{1},2,3,\overline{3},\overline{2},1}{.5}{.25}{white!50!gray}{\t}};
		\draw(n32) node[scale=\s]{\AlphaPerm{1,1,1,1,1,1}{c3,c2,c1,c1,c2,c3}{\overline{2},3,1,\overline{1},\overline{3},2}{.5}{.25}{white!50!gray}{\t}};
		\draw(n33) node[scale=\s]{\AlphaPerm{1,1,1,1,1,1}{c3,c2,c1,c1,c2,c3}{3,\overline{1},\overline{2},2,1,\overline{3}}{.5}{.25}{white!50!gray}{\t}};
		\draw(n34) node[scale=\s]{\AlphaPerm{1,1,1,1,1,1}{c3,c2,c1,c1,c2,c3}{3,\overline{2},1,\overline{1},2,\overline{3}}{.5}{.25}{white!50!gray}{\t}};
		\draw(n35) node[scale=\s]{\AlphaPerm{1,1,1,1,1,1}{c3,c2,c1,c1,c2,c3}{2,1,\overline{3},3,\overline{1},\overline{2}}{.5}{.25}{white!50!gray}{\t}};
		\draw(n36) node[scale=\s]{\AlphaPerm{1,1,1,1,1,1}{c3,c2,c1,c1,c2,c3}{2,\overline{1},3,\overline{3},1,\overline{2}}{.5}{.25}{white!50!gray}{\t}};
		\draw(n37) node[scale=\s]{\AlphaPerm{1,1,1,1,1,1}{c3,c2,c1,c1,c2,c3}{1,3,\overline{2},2,\overline{3},\overline{1}}{.5}{.25}{white!50!gray}{\t}};
		\draw(n38) node[scale=\s]{\AlphaPerm{1,1,1,1,1,1}{c3,c2,c1,c1,c2,c3}{1,2,3,\overline{3},\overline{2},\overline{1}}{.5}{.25}{white!50!gray}{\t}};
		\draw(n39) node[scale=\s]{\AlphaPerm{1,1,1,1,1,1}{c3,c2,c1,c1,c2,c3}{\overline{1},3,2,\overline{2},\overline{3},1}{.5}{.25}{white!50!gray}{\t}};
		\draw(n40) node[scale=\s]{\AlphaPerm{1,1,1,1,1,1}{c3,c2,c1,c1,c2,c3}{3,1,\overline{2},2,\overline{1},\overline{3}}{.5}{.25}{white!50!gray}{\t}};
		\draw(n41) node[scale=\s]{\AlphaPerm{1,1,1,1,1,1}{c3,c2,c1,c1,c2,c3}{3,\overline{1},2,\overline{2},1,\overline{3}}{.5}{.25}{white!50!gray}{\t}};
		\draw(n42) node[scale=\s]{\AlphaPerm{1,1,1,1,1,1}{c3,c2,c1,c1,c2,c3}{2,3,\overline{1},1,\overline{3},\overline{2}}{.5}{.25}{white!50!gray}{\t}};
		\draw(n43) node[scale=\s]{\AlphaPerm{1,1,1,1,1,1}{c3,c2,c1,c1,c2,c3}{2,1,3,\overline{3},\overline{1},\overline{2}}{.5}{.25}{white!50!gray}{\t}};
		\draw(n44) node[scale=\s]{\AlphaPerm{1,1,1,1,1,1}{c3,c2,c1,c1,c2,c3}{1,3,2,\overline{2},\overline{3},\overline{1}}{.5}{.25}{white!50!gray}{\t}};
		\draw(n45) node[scale=\s]{\AlphaPerm{1,1,1,1,1,1}{c3,c2,c1,c1,c2,c3}{3,2,\overline{1},1,\overline{2},\overline{3}}{.5}{.25}{white!50!gray}{\t}};
		\draw(n46) node[scale=\s]{\AlphaPerm{1,1,1,1,1,1}{c3,c2,c1,c1,c2,c3}{3,1,2,\overline{2},\overline{1},\overline{3}}{.5}{.25}{white!50!gray}{\t}};
		\draw(n47) node[scale=\s]{\AlphaPerm{1,1,1,1,1,1}{c3,c2,c1,c1,c2,c3}{2,3,1,\overline{1},\overline{3},\overline{2}}{.5}{.25}{white!50!gray}{\t}};
		\draw(n48) node[scale=\s]{\AlphaPerm{1,1,1,1,1,1}{c3,c2,c1,c1,c2,c3}{3,2,1,\overline{1},\overline{2},\overline{3}}{.5}{.25}{white!50!gray}{\t}};
		\begin{pgfonlayer}{background}
			\filldraw[draw=gray,fill=white!50!gray,rounded corners](5.55*\x,.85*\y) -- (6.45*\x,.85*\y) -- (6.45*\x,1.25*\y) -- (5.55*\x,1.25*\y) -- cycle;
			\filldraw[draw=gray,fill=white!50!gray,rounded corners](6.55*\x,1.85*\y) -- (7.45*\x,1.85*\y) -- (7.45*\x,2.25*\y) -- (6.55*\x,2.25*\y) -- cycle;
			\filldraw[draw=gray,fill=white!50!gray,rounded corners](5.55*\x,1.85*\y) -- (6.45*\x,1.85*\y) -- (6.45*\x,2.25*\y) -- (7.45*\x,2.85*\y) -- (7.45*\x,3.25*\y) -- (8.45*\x,3.85*\y) -- (8.45*\x,4.25*\y) -- (7.45*\x,4.85*\y) -- (7.45*\x,5.25*\y) -- (6.55*\x,5.25*\y) -- (6.55*\x,4.85*\y) -- (7.55*\x,4.25*\y) -- (7.55*\x,3.85*\y) -- (6.55*\x,3.25*\y) -- (6.55*\x,2.85*\y) -- (5.55*\x,2.25*\y) -- cycle;
			\filldraw[draw=gray,fill=white!50!gray,rounded corners](4.55*\x,1.85*\y) -- (5.45*\x,1.85*\y) -- (5.45*\x,2.25*\y) -- (4.45*\x,2.85*\y) -- (4.45*\x,3.25*\y) -- (3.45*\x,3.85*\y) -- (3.45*\x,4.25*\y) -- (2.45*\x,4.85*\y) -- (2.45*\x,5.25*\y) -- (1.45*\x,5.85*\y) -- (1.45*\x,6.25*\y) -- (.55*\x,6.25*\y) -- (.55*\x,5.85*\y) -- (1.55*\x,5.25*\y) -- (1.55*\x,4.85*\y) -- (2.55*\x,4.25*\y) -- (2.55*\x,3.85*\y) -- (3.55*\x,3.25*\y) -- (3.55*\x,2.85*\y) -- (4.55*\x,2.25*\y) -- cycle;
			\filldraw[draw=gray,fill=white!50!gray,rounded corners](7.55*\x,2.85*\y) -- (8.45*\x,2.85*\y) -- (8.45*\x,3.25*\y) -- (7.55*\x,3.25*\y) -- cycle;
			\filldraw[draw=gray,fill=white!50!gray,rounded corners](5.55*\x,2.85*\y) -- (6.45*\x,2.85*\y) -- (6.45*\x,3.25*\y) -- (5.45*\x,3.85*\y) -- (5.45*\x,4.25*\y) -- (4.45*\x,4.85*\y) -- (4.45*\x,5.25*\y) -- (3.45*\x,5.85*\y) -- (3.45*\x,6.25*\y) -- (2.45*\x,6.85*\y) -- (2.45*\x,7.25*\y) -- (1.55*\x,7.25*\y) -- (1.55*\x,6.85*\y) -- (2.55*\x,6.25*\y) -- (2.55*\x,5.85*\y) -- (3.55*\x,5.25*\y) -- (3.55*\x,4.85*\y) -- (4.55*\x,4.25*\y) -- (4.55*\x,3.85*\y) -- (5.55*\x,3.25*\y) -- cycle;
			\filldraw[draw=gray,fill=white!50!gray,rounded corners](8.55*\x,3.85*\y) -- (9.45*\x,3.85*\y) -- (9.45*\x,4.25*\y) -- (8.55*\x,4.25*\y) -- cycle;
			\filldraw[draw=gray,fill=white!50!gray,rounded corners](6.55*\x,3.85*\y) -- (7.45*\x,3.85*\y) -- (7.45*\x,4.25*\y) -- (6.55*\x,4.25*\y) -- cycle;
			\filldraw[draw=gray,fill=white!50!gray,rounded corners](8.55*\x,4.85*\y) -- (9.45*\x,4.85*\y) -- (9.45*\x,5.25*\y) -- (8.45*\x,5.85*\y) -- (8.45*\x,6.25*\y) -- (7.55*\x,6.25*\y) -- (7.55*\x,5.85*\y) -- (8.55*\x,5.25*\y) -- cycle;
			\filldraw[draw=gray,fill=white!50!gray,rounded corners](5.55*\x,4.85*\y) -- (6.45*\x,4.85*\y) -- (6.45*\x,5.25*\y) -- (5.45*\x,5.85*\y) -- (5.45*\x,6.25*\y) -- (4.45*\x,6.85*\y) -- (4.45*\x,7.25*\y) -- (3.45*\x,7.85*\y) -- (3.45*\x,8.25*\y) -- (2.55*\x,8.25*\y) -- (2.55*\x,7.85*\y) -- (3.55*\x,7.25*\y) -- (3.55*\x,6.85*\y) -- (4.55*\x,6.25*\y) -- (4.55*\x,5.85*\y) -- (5.55*\x,5.25*\y) -- cycle;
			\filldraw[draw=gray,fill=white!50!gray,rounded corners](7.55*\x,4.85*\y) -- (8.45*\x,4.85*\y) -- (8.45*\x,5.25*\y) -- (7.55*\x,5.25*\y) -- cycle;
			\filldraw[draw=gray,fill=white!50!gray,rounded corners](6.55*\x,5.85*\y) -- (7.45*\x,5.85*\y) -- (7.45*\x,6.25*\y) -- (6.55*\x,6.25*\y) -- cycle;
			\filldraw[draw=gray,fill=white!50!gray,rounded corners](6.55*\x,6.85*\y) -- (7.45*\x,6.85*\y) -- (7.45*\x,7.25*\y) -- (6.55*\x,7.25*\y) -- cycle;
			\filldraw[draw=gray,fill=white!50!gray,rounded corners](5.55*\x,6.85*\y) -- (6.45*\x,6.85*\y) -- (6.45*\x,7.25*\y) -- (5.55*\x,7.25*\y) -- cycle;
			\filldraw[draw=gray,fill=white!50!gray,rounded corners](5.55*\x,7.85*\y) -- (6.45*\x,7.85*\y) -- (6.45*\x,8.25*\y) -- (5.55*\x,8.25*\y) -- cycle;
			\filldraw[draw=gray,fill=white!50!gray,rounded corners](4.55*\x,7.85*\y) -- (5.45*\x,7.85*\y) -- (5.45*\x,8.25*\y) -- (4.45*\x,8.85*\y) -- (4.45*\x,9.25*\y) -- (3.55*\x,9.25*\y) -- (3.55*\x,8.85*\y) -- (4.55*\x,8.25*\y) -- cycle;
			\filldraw[draw=gray,fill=white!50!gray,rounded corners](4.55*\x,8.85*\y) -- (5.45*\x,8.85*\y) -- (5.45*\x,9.25*\y) -- (4.55*\x,9.25*\y) -- cycle;
			\filldraw[draw=gray,fill=white!50!gray,rounded corners](3.55*\x,9.85*\y) -- (4.45*\x,9.85*\y) -- (4.45*\x,10.25*\y) -- (3.55*\x,10.25*\y) -- cycle;
			\filldraw[draw=gray,fill=white!50!gray,rounded corners](4.55*\x,2.85*\y) -- (5.45*\x,2.85*\y) -- (5.45*\x,3.25*\y) -- (6.45*\x,3.85*\y) -- (6.45*\x,4.25*\y) -- (5.45*\x,4.85*\y) -- (5.45*\x,5.25*\y) -- (6.45*\x,5.85*\y) -- (6.45*\x,6.25*\y) -- (5.45*\x,6.85*\y) -- (5.45*\x,7.25*\y) -- (4.45*\x,7.85*\y) -- (4.45*\x,8.25*\y) -- (3.55*\x,8.25*\y) -- (3.55*\x,7.85*\y) -- (4.55*\x,7.25*\y) -- (4.55*\x,6.85*\y) -- (5.55*\x,6.25*\y) -- (5.55*\x,5.85*\y) -- (4.55*\x,5.25*\y) -- (4.55*\x,4.85*\y) -- (5.55*\x,4.25*\y) -- (5.55*\x,3.85*\y) -- (4.55*\x,3.25*\y) -- cycle;
			\fill[white!50!gray,rounded corners](4.55*\x,4.85*\y) -- (5.45*\x,4.85*\y) -- (5.45*\x,5.25*\y) -- (4.45*\x,5.85*\y) -- (4.45*\x,6.25*\y) -- (5.45*\x,6.85*\y) -- (5.45*\x,7.25*\y) -- (4.55*\x,7.25*\y) -- (4.55*\x,6.85*\y) -- (3.55*\x,6.25*\y) -- (3.55*\x,5.85*\y) -- (4.55*\x,5.25*\y) -- cycle;
			\draw[gray,rounded corners](4.55*\x,5*\y) -- (4.55*\x,5.25*\y) -- (3.55*\x,5.85*\y) -- (3.55*\x,6.25*\y) -- (4.55*\x,6.85*\y) -- (4.55*\x,7*\y);
			\draw[gray,rounded corners](5*\x,5.525*\y) -- (4.45*\x,5.85*\y) -- (4.45*\x,6.25*\y) -- (5*\x,6.59*\y);
			\filldraw[draw=gray,fill=white!50!gray,rounded corners](3.55*\x,3.85*\y) -- (4.45*\x,3.85*\y) -- (4.45*\x,4.25*\y) -- (3.45*\x,4.85*\y) -- (3.45*\x,5.25*\y) -- (2.45*\x,5.85*\y) -- (2.45*\x,6.25*\y) -- (1.45*\x,6.85*\y) -- (1.45*\x,7.25*\y) -- (2.45*\x,7.85*\y) -- (2.45*\x,8.25*\y) -- (3.45*\x,8.85*\y) -- (3.45*\x,9.25*\y) -- (2.55*\x,9.25*\y) -- (2.55*\x,8.85*\y) -- (1.55*\x,8.25*\y) -- (1.55*\x,7.85*\y) -- (.55*\x,7.25*\y) -- (.55*\x,6.85*\y) -- (1.55*\x,6.25*\y) -- (1.55*\x,5.85*\y) -- (2.55*\x,5.25*\y) -- (2.55*\x,4.85*\y) -- (3.55*\x,4.25*\y) -- cycle;
			\fill[white!50!gray,rounded corners](2.45*\x,8.25*\y) -- (2.45*\x,7.85*\y) -- (3.45*\x,7.25*\y) -- (3.45*\x,6.85*\y) -- (2.45*\x,6.25*\y) -- (2.45*\x,5.85*\y) -- (1.55*\x,5.85*\y) -- (1.55*\x,6.25*\y) -- (2.55*\x,6.85*\y) -- (2.55*\x,7.25*\y) -- (1.55*\x,7.85*\y) -- (1.55*\x,8.25*\y) -- cycle;
			\draw[gray,rounded corners](2.45*\x,6*\y) -- (2.45*\x,6.25*\y) -- (3.45*\x,6.85*\y) -- (3.45*\x,7.25*\y) -- (2.45*\x,7.85*\y) -- (2.45*\x,8*\y);
			\draw[gray,rounded corners](2*\x,6.525*\y) -- (2.55*\x,6.85*\y) -- (2.55*\x,7.25*\y) -- (2*\x,7.59*\y);
		\end{pgfonlayer}
	\end{tikzpicture}
    \caption{The weak order on $\Hyper_{3}$.} 
    \label{fig:weak_order_b3}
\end{sidewaysfigure}

\begin{figure}
    \centering
    \begin{tikzpicture}[scale=.7]\small
		\def\x{2};
		\def\y{2};
		\def\s{.7};
		\def\t{.8};	
		\coordinate(n1) at (6*\x,1*\y);
		\coordinate(n2) at (6*\x,2*\y);
		\coordinate(n3) at (7.5*\x,2.5*\y);
		\coordinate(n4) at (8*\x,3*\y);
		\coordinate(n5) at (6.75*\x,4.25*\y);
		\coordinate(n6) at (9*\x,4*\y);
		\coordinate(n7) at (7.75*\x,5.25*\y);
		\coordinate(n8) at (9*\x,5*\y);
		\coordinate(n9) at (1*\x,6*\y);
		\coordinate(n10) at (2.25*\x,5.75*\y);
		\coordinate(n11) at (7.25*\x,5.75*\y);
		\coordinate(n12) at (2.5*\x,7.5*\y);
		\coordinate(n13) at (1*\x,7*\y);
		\coordinate(n14) at (5.75*\x,7.25*\y);
		\coordinate(n15) at (7.25*\x,6.75*\y);
		\coordinate(n16) at (3*\x,8*\y);
		\coordinate(n17) at (5.75*\x,8.25*\y);
		\coordinate(n18) at (4*\x,9*\y);
		\coordinate(n19) at (5.25*\x,8.75*\y);
		\coordinate(n20) at (4*\x,10*\y);
		\draw[thick](n1) -- (n2);
		\draw[thick](n1) -- (n3);
		\draw[thick](n1) -- (n9);
		\draw[thick](n2) -- (n8);
		\draw[thick](n2) -- (n10);
		\draw[thick](n3) -- (n4);
		\draw[thick](n3) -- (n12);
		\draw[thick](n4) -- (n5);
		\draw[thick](n4) -- (n6);
		\draw[thick](n5) -- (n7);
		\draw[thick](n5) -- (n16);
		\draw[thick](n6) -- (n7);
		\draw[thick](n6) -- (n8);
		\draw[thick](n7) -- (n11);
		\draw[thick](n8) -- (n15);
		\draw[thick](n9) -- (n12);
		\draw[thick](n9) -- (n13);
		\draw[thick](n10) -- (n13);
		\draw[thick](n10) -- (n19);
		\draw[thick](n11) -- (n14);
		\draw[thick](n11) -- (n15);
		\draw[thick](n12) -- (n16);
		\draw[thick](n13) -- (n20);
		\draw[thick](n14) -- (n17);
		\draw[thick](n14) -- (n18);
		\draw[thick](n15) -- (n17);
		\draw[thick](n16) -- (n18);
		\draw[thick](n17) -- (n19);
		\draw[thick](n18) -- (n20);
		\draw[thick](n19) -- (n20);
		\draw(n1) node[scale=\s]{\AlphaPerm{1,1,1,1,1,1}{c3,c2,c1,c1,c2,c3}{\overline{3},\overline{2},\overline{1},1,2,3}{.5}{.25}{white}{\t}};
		\draw(n2) node[scale=\s]{\AlphaPerm{1,1,1,1,1,1}{c3,c2,c1,c1,c2,c3}{\overline{3},\overline{1},\overline{2},2,1,3}{.5}{.25}{white}{\t}};
		\draw(n3) node[scale=\s]{\AlphaPerm{1,1,1,1,1,1}{c3,c2,c1,c1,c2,c3}{\overline{3},\overline{2},1,\overline{1},2,3}{.5}{.25}{white}{\t}};
		\draw(n4) node[scale=\s]{\AlphaPerm{1,1,1,1,1,1}{c3,c2,c1,c1,c2,c3}{\overline{3},\overline{1},2,\overline{2},1,3}{.5}{.25}{white}{\t}};
		\draw(n5) node[scale=\s]{\AlphaPerm{1,1,1,1,1,1}{c3,c2,c1,c1,c2,c3}{\overline{2},\overline{1},3,\overline{3},1,2}{.5}{.25}{white}{\t}};
		\draw(n6) node[scale=\s]{\AlphaPerm{1,1,1,1,1,1}{c3,c2,c1,c1,c2,c3}{\overline{3},1,2,\overline{2},\overline{1},3}{.5}{.25}{white}{\t}};
		\draw(n7) node[scale=\s]{\AlphaPerm{1,1,1,1,1,1}{c3,c2,c1,c1,c2,c3}{\overline{2},1,3,\overline{3},\overline{1},2}{.5}{.25}{white}{\t}};
		\draw(n8) node[scale=\s]{\AlphaPerm{1,1,1,1,1,1}{c3,c2,c1,c1,c2,c3}{\overline{3},2,1,\overline{1},\overline{2},3}{.5}{.25}{white}{\t}};
		\draw(n9) node[scale=\s]{\AlphaPerm{1,1,1,1,1,1}{c3,c2,c1,c1,c2,c3}{\overline{2},\overline{3},\overline{1},1,3,2}{.5}{.25}{white}{\t}};
		\draw(n10) node[scale=\s]{\AlphaPerm{1,1,1,1,1,1}{c3,c2,c1,c1,c2,c3}{\overline{2},\overline{1},\overline{3},3,1,2}{.5}{.25}{white}{\t}};
		\draw(n11) node[scale=\s]{\AlphaPerm{1,1,1,1,1,1}{c3,c2,c1,c1,c2,c3}{\overline{1},2,3,\overline{3},\overline{2},1}{.5}{.25}{white}{\t}};
		\draw(n12) node[scale=\s]{\AlphaPerm{1,1,1,1,1,1}{c3,c2,c1,c1,c2,c3}{\overline{2},\overline{3},1,\overline{1},3,2}{.5}{.25}{white}{\t}};
		\draw(n13) node[scale=\s]{\AlphaPerm{1,1,1,1,1,1}{c3,c2,c1,c1,c2,c3}{\overline{1},\overline{2},\overline{3},3,2,1}{.5}{.25}{white}{\t}};
		\draw(n14) node[scale=\s]{\AlphaPerm{1,1,1,1,1,1}{c3,c2,c1,c1,c2,c3}{1,2,3,\overline{3},\overline{2},\overline{1}}{.5}{.25}{white}{\t}};
		\draw(n15) node[scale=\s]{\AlphaPerm{1,1,1,1,1,1}{c3,c2,c1,c1,c2,c3}{\overline{1},3,2,\overline{2},\overline{3},1}{.5}{.25}{white}{\t}};
		\draw(n16) node[scale=\s]{\AlphaPerm{1,1,1,1,1,1}{c3,c2,c1,c1,c2,c3}{\overline{1},\overline{2},3,\overline{3},2,1}{.5}{.25}{white}{\t}};
		\draw(n17) node[scale=\s]{\AlphaPerm{1,1,1,1,1,1}{c3,c2,c1,c1,c2,c3}{1,3,2,\overline{2},\overline{3},\overline{1}}{.5}{.25}{white}{\t}};
		\draw(n18) node[scale=\s]{\AlphaPerm{1,1,1,1,1,1}{c3,c2,c1,c1,c2,c3}{2,1,3,\overline{3},\overline{1},\overline{2}}{.5}{.25}{white}{\t}};
		\draw(n19) node[scale=\s]{\AlphaPerm{1,1,1,1,1,1}{c3,c2,c1,c1,c2,c3}{2,3,1,\overline{1},\overline{3},\overline{2}}{.5}{.25}{white}{\t}};
		\draw(n20) node[scale=\s]{\AlphaPerm{1,1,1,1,1,1}{c3,c2,c1,c1,c2,c3}{3,2,1,\overline{1},\overline{2},\overline{3}}{.5}{.25}{white}{\t}};
	\end{tikzpicture}
    \caption{The lattice $\Tamari_{B}(3)$.}
    \label{fig:tamari_b3}
\end{figure}

\begin{theorem}[\cite{reading06cambrian,ingalls09noncrossing}]\label{thm:tamari_type_b}
    For $n>0$, the poset $\Tamari_{B}(n)$ is both a sublattice and a quotient lattice (for the projection sending a sign-symmetric permutation to the largest aligned permutation smaller than itself) of $\Weak(\Hyper_{n})$. Moreover, it is congruence uniform and trim.
\end{theorem}

\section{Parabolic quotients in type $B$}
    \label{sec:parabolic_type_b}
\subsection{A distinguished set of representatives}
    \label{sec:representatives}
Let $W$ be a Coxeter group, whose set of simple reflections is $S$. Any $J\subseteq S$ generates the \defn{parabolic} subgroup $W_{J}$ of $W$. The \defn{parabolic} quotient $W^{J}\defs W/W_{J}$ consists of the left cosets of $W_J$ in $W$. By construction, every such coset has a member of minimal Coxeter length. In the following, we identify the quotient $W^{J}$ with this set of minimal length representatives. Equivalently, we can define the set $W^{J}$ as follows:
\begin{displaymath}
    W^{J} = \bigl\{w\in W\mid \ell_{S}(ws)>\ell_{S}(w)\;\text{for all}\;s\in J\bigr\}.
\end{displaymath}

If we consider the Coxeter group of type $B$ in its permutation representation defined above, we may describe the parabolic quotients neatly as follows. Let $n>0$, and consider a composition $\alpha=(\alpha_{1},\alpha_{2},\ldots,\alpha_{r})$ of $n$, \ie $\alpha_{1}+\alpha_{2}+\cdots+\alpha_{r}=n$.  Let us define $p_{0}\defs 0$ and $p_{i}\defs\alpha_{1}+\alpha_{2}+\cdots+\alpha_{i}$.  The number of compositions of $n$ is $2^{n-1}$, because we may identify $\alpha$ with the subset $\{p_{1},p_{2},\ldots,p_{r-1}\}$ of $[n-1]$.  Conversely, with $\{k_{1},k_{2},\ldots,k_{r-1}\}\subseteq [n-1]$ we may associate the composition $(k_{1},k_{2}-k_{1},k_{3}-k_{2},\ldots,k_{r-1}-k_{r-2},n-k_{r-1})$ of $n$.

A \defn{type-$B$ composition} of $n>0$ is a composition of $n$ with a possible \defn{zero-component} $\alpha_{0}\defs 0$ in the beginning. A type-$B$ composition is \defn{split} if it has a zero-component, and it is \defn{join} otherwise.

The number of type-$B$ compositions of $n$ is $2^{n}$, because any type-$B$ composition $\alpha$ is associated with a unique subset $J_{\alpha}\subseteq S$. The non-zero-components of $\alpha$ determine a subset of $\{s_{1},s_{2},\ldots,s_{n-1}\}$ as before, and we add $s_{0}$ to this set if and only if $\alpha$ is split.

Therefore, type-$B$ compositions are in natural bijection with the subsets of the simple reflections of $\Hyper_{n}$, so that if $\alpha$ is a type-$B$ composition, then there exists a unique parabolic quotient of $\Hyper_{n}$ determined by the \textit{complement} $S\setminus J_{\alpha}$. We write $\Hyper_{\alpha}$ for the resulting set of minimal-length representatives (see Example~\ref{ex-parab}).

\begin{remark}
	Let us emphasize the difference between the notions $\Hyper_{n}$ where the subscript is an \textit{integer}, and $\Hyper_{\alpha}$ where the subscript is a type-$B$ \textit{composition}.  The first denotes the full hyperoctahedral group of degree $n$ and the second denotes the set of minimal length representatives of the parabolic quotient associated with $\alpha$.  Clearly, if $\alpha=(0,1,1,\ldots,1)$ is a composition of $n$, then $\Hyper_{(0,1,1,\ldots,1)}=\Hyper_{n}$.
\end{remark}

Let $\alpha$ be a type-$B$ composition with $r$ non-zero-components $\alpha_{1},\alpha_{2},\ldots,\alpha_{r}$ appearing in that order. As before, we define $p_{i}=\alpha_{1}+\alpha_{2}+\cdots+\alpha_{i}$, and we set $\overline{p}_{i}\defs -p_{i}$.  We define $\Part(\alpha)$ as the following partition of $\pm[n]$:
\begin{displaymath}
    \begin{cases}
        \bigl\{[\overline{n},\overline{p}_{r-1}{-}1],\ldots,[\overline{p}_{2},\overline{p}_{1}{-}1],[\overline{p}_{1},\overline{1}],[1,p_{1}],[p_{1}{+}1,p_{2}],\ldots, [p_{r-1}{+}1,n]\bigr\},
        & \text{if}\;\alpha\;\text{is split},\\
        \bigl\{[\overline{n},\overline{p}_{r-1}{-}1],\ldots,[\overline{p}_{2},\overline{p}_{1}{-}1],[\overline{p}_{1},p_{1}],[p_{1}{+}1,p_{2}],\ldots,[p_{r-1}{+}1,n]\bigr\},
        & \text{if}\;\alpha\;\text{is join}.
    \end{cases}
\end{displaymath}
This means, if $\alpha$ is split, that we partition $\pm[n]$ into $2r$ symmetrically placed blocks, and if $\alpha$ is join, we partition it into $2r-1$ blocks, where the middle block contains positive and negative positions. This partition also explains the terminology ``join'' and ``split'' for the type-$B$ compositions. The partition corresponding to a join composition has a block joining positive and negative positions, the partition corresponding to a split composition does not.

For an index $a\in[n]$ of the right part, we say that $a$ is \defn{in the $i\th$ $\alpha$-region} if $p_{i-1}\leq a<p_{i}$, where $p_{0}=0$. We sometimes write $\indicator_{\alpha}(a)=i$ in this situation.

\begin{lemma}\label{lem:increasing_blocks}
    Let $\alpha$ be a type-$B$ composition. A sign-symmetric permutation $\pi\in\Hyper_{n}$ belongs to $\Hyper_{\alpha}$ if and only if the long one-line notation of $\pi$, when partitioned according to $\Part(\alpha)$, is increasing in each block.
\end{lemma}
\begin{proof}
    Let $\alpha$ be a type-$B$ composition, and let $J_{\alpha}$ be the corresponding subset of the simple reflections $S=\{s_{0},s_{1},\ldots,s_{n-1}\}$.  Note that we use the set $S\setminus J_{\alpha}$ to construct the parabolic quotient $ \Hyper_{\alpha}$.

    Let $\pi\in\Hyper_{n}$ and $s\in S$ be such that $\ell_{S}(\pi')>\ell_{S}(\pi)$ for the product $\pi'=\pi s$.  Since $s$ is an involution, we get $s\in\invset(\pi')$, and according to Lemma~\ref{lem:type_b_inversions} we have $\pi'(1)<0$ if $s=\sleft 1\sright$ or $\pi'(i)>\pi'(i{+}1)$ if $s=\lleft i\;i{+}1\rright$. Comparing this with the definition of $\Part(\alpha)$ yields the claim.
\end{proof}

\begin{example}\label{ex-parab}
    Let $n=3$ and consider $\alpha=(0,1,2)$.  Then $\alpha$ is split and we obtain $J_{\alpha}=\{s_{0},s_{1}\}$, which means that $S\setminus J_{\alpha}=\{s_{2}\}$.  The long one-line notation is therefore partitioned into four parts: $\{-3,-2\},\{-1\},\{1\},\{2,3\}$. We highlight the first and last part in green and the second and third part in orange.  Then, $\Hyper_{\alpha}$ has the following 24 elements:
    \begin{center}
	   $\begin{aligned}
		& \hspace*{-.3cm}\raisebox{\rsp}{\AlphaPerm{2,1,1,2}{c2,c1,c1,c2}{\overline{3},\overline{2},\overline{1},1,2,3}{.5}{.25}{white}{1}}\hspace*{-.3cm},
		&& \hspace*{-.3cm}\raisebox{\rsp}{\AlphaPerm{2,1,1,2}{c2,c1,c1,c2}{\overline{3},2,\overline{1},1,\overline{2},3}{.5}{.25}{white}{1}}\hspace*{-.3cm},
		&& \hspace*{-.3cm}\raisebox{\rsp}{\AlphaPerm{2,1,1,2}{c2,c1,c1,c2}{\overline{2},3,\overline{1},1,\overline{3},2}{.5}{.25}{white}{1}}\hspace*{-.3cm},
		&& \hspace*{-.3cm}\raisebox{\rsp}{\AlphaPerm{2,1,1,2}{c2,c1,c1,c2}{2,3,\overline{1},1,\overline{3},\overline{2}}{.5}{.25}{white}{1}}\hspace*{-.3cm},\\
		& \hspace*{-.3cm}\raisebox{\rsp}{\AlphaPerm{2,1,1,2}{c2,c1,c1,c2}{\overline{3},\overline{2},1,\overline{1},2,3}{.5}{.25}{white}{1}}\hspace*{-.3cm},
		&& \hspace*{-.3cm}\raisebox{\rsp}{\AlphaPerm{2,1,1,2}{c2,c1,c1,c2}{\overline{3},2,1,\overline{1},\overline{2},3}{.5}{.25}{white}{1}}\hspace*{-.3cm},
		&& \hspace*{-.3cm}\raisebox{\rsp}{\AlphaPerm{2,1,1,2}{c2,c1,c1,c2}{\overline{2},3,1,\overline{1},\overline{3},2}{.5}{.25}{white}{1}}\hspace*{-.3cm},
		&& \hspace*{-.3cm}\raisebox{\rsp}{\AlphaPerm{2,1,1,2}{c2,c1,c1,c2}{2,3,1,\overline{1},\overline{3},\overline{2}}{.5}{.25}{white}{1}}\hspace*{-.3cm},\\
		& \hspace*{-.3cm}\raisebox{\rsp}{\AlphaPerm{2,1,1,2}{c2,c1,c1,c2}{\overline{3},\overline{1},\overline{2},2,1,3}{.5}{.25}{white}{1}}\hspace*{-.3cm},
		&& \hspace*{-.3cm}\raisebox{\rsp}{\AlphaPerm{2,1,1,2}{c2,c1,c1,c2}{\overline{3},1,\overline{2},2,\overline{1},3}{.5}{.25}{white}{1}}\hspace*{-.3cm},
		&& \hspace*{-.3cm}\raisebox{\rsp}{\AlphaPerm{2,1,1,2}{c2,c1,c1,c2}{\overline{1},3,\overline{2},2,\overline{3},1}{.5}{.25}{white}{1}}\hspace*{-.3cm},
		&& \hspace*{-.3cm}\raisebox{\rsp}{\AlphaPerm{2,1,1,2}{c2,c1,c1,c2}{1,3,\overline{2},2,\overline{3},\overline{1}}{.5}{.25}{white}{1}}\hspace*{-.3cm},\\
		& \hspace*{-.3cm}\raisebox{\rsp}{\AlphaPerm{2,1,1,2}{c2,c1,c1,c2}{\overline{3},\overline{1},2,\overline{2},1,3}{.5}{.25}{white}{1}}\hspace*{-.3cm},
		&& \hspace*{-.3cm}\raisebox{\rsp}{\AlphaPerm{2,1,1,2}{c2,c1,c1,c2}{\overline{3},1,2,\overline{2},\overline{1},3}{.5}{.25}{white}{1}}\hspace*{-.3cm},
		&& \hspace*{-.3cm}\raisebox{\rsp}{\AlphaPerm{2,1,1,2}{c2,c1,c1,c2}{\overline{1},3,2,\overline{2},\overline{3},1}{.5}{.25}{white}{1}}\hspace*{-.3cm},
		&& \hspace*{-.3cm}\raisebox{\rsp}{\AlphaPerm{2,1,1,2}{c2,c1,c1,c2}{1,3,2,\overline{2},\overline{3},\overline{1}}{.5}{.25}{white}{1}}\hspace*{-.3cm},\\
		& \hspace*{-.3cm}\raisebox{\rsp}{\AlphaPerm{2,1,1,2}{c2,c1,c1,c2}{\overline{2},\overline{1},\overline{3},3,1,2}{.5}{.25}{white}{1}}\hspace*{-.3cm},
		&& \hspace*{-.3cm}\raisebox{\rsp}{\AlphaPerm{2,1,1,2}{c2,c1,c1,c2}{\overline{2},1,\overline{3},3,\overline{1},2}{.5}{.25}{white}{1}}\hspace*{-.3cm},
		&& \hspace*{-.3cm}\raisebox{\rsp}{\AlphaPerm{2,1,1,2}{c2,c1,c1,c2}{\overline{1},2,\overline{3},3,\overline{2},1}{.5}{.25}{white}{1}}\hspace*{-.3cm},
		&& \hspace*{-.3cm}\raisebox{\rsp}{\AlphaPerm{2,1,1,2}{c2,c1,c1,c2}{1,2,\overline{3},3,\overline{2},\overline{1}}{.5}{.25}{white}{1}}\hspace*{-.3cm},\\
		& \hspace*{-.3cm}\raisebox{\rsp}{\AlphaPerm{2,1,1,2}{c2,c1,c1,c2}{\overline{2},\overline{1},3,\overline{3},1,2}{.5}{.25}{white}{1}}\hspace*{-.3cm},
		&& \hspace*{-.3cm}\raisebox{\rsp}{\AlphaPerm{2,1,1,2}{c2,c1,c1,c2}{\overline{2},1,3,\overline{3},\overline{1},2}{.5}{.25}{white}{1}}\hspace*{-.3cm},
		&& \hspace*{-.3cm}\raisebox{\rsp}{\AlphaPerm{2,1,1,2}{c2,c1,c1,c2}{\overline{1},2,3,\overline{3},\overline{2},1}{.5}{.25}{white}{1}}\hspace*{-.3cm},
		&& \hspace*{-.3cm}\raisebox{\rsp}{\AlphaPerm{2,1,1,2}{c2,c1,c1,c2}{1,2,3,\overline{3},\overline{2},\overline{1}}{.5}{.25}{white}{1}}\hspace*{-.3cm}.\\
	   \end{aligned}$
    \end{center}

    If we choose $\alpha'=(1,2)$, then $\alpha'$ is join and we obtain $J_{\alpha'}=\{s_{1}\}$ which gives us $S\setminus\{J_{\alpha'}\}=\{s_{0},s_{2}\}$. We may partition the long one-line notation into three parts $\{-3,-2\},\{,-1,1\},\{2,3\}$, and we highlight the first and the last part in green and the middle part in orange. The set $\Hyper_{\alpha'}$ consists of the following twelve elements:
    \begin{center}
	   $\begin{aligned}
		& \hspace*{-.3cm}\raisebox{\rsp}{\AlphaPerm{2,2,2}{c2,c1,c2}{\overline{3},\overline{2},\overline{1},1,2,3}{.5}{.25}{white}{1}}\hspace*{-.3cm},
		&& \hspace*{-.3cm}\raisebox{\rsp}{\AlphaPerm{2,2,2}{c2,c1,c2}{\overline{3},2,\overline{1},1,\overline{2},3}{.5}{.25}{white}{1}}\hspace*{-.3cm},
		&& \hspace*{-.3cm}\raisebox{\rsp}{\AlphaPerm{2,2,2}{c2,c1,c2}{\overline{2},3,\overline{1},1,\overline{3},2}{.5}{.25}{white}{1}}\hspace*{-.3cm},
		&& \hspace*{-.3cm}\raisebox{\rsp}{\AlphaPerm{2,2,2}{c2,c1,c2}{2,3,\overline{1},1,\overline{3},\overline{2}}{.5}{.25}{white}{1}}\hspace*{-.3cm},\\
		& \hspace*{-.3cm}\raisebox{\rsp}{\AlphaPerm{2,2,2}{c2,c1,c2}{\overline{3},\overline{1},\overline{2},2,1,3}{.5}{.25}{white}{1}}\hspace*{-.3cm},
		&& \hspace*{-.3cm}\raisebox{\rsp}{\AlphaPerm{2,2,2}{c2,c1,c2}{\overline{3},1,\overline{2},2,\overline{1},3}{.5}{.25}{white}{1}}\hspace*{-.3cm},
		&& \hspace*{-.3cm}\raisebox{\rsp}{\AlphaPerm{2,2,2}{c2,c1,c2}{\overline{1},3,\overline{2},2,\overline{3},1}{.5}{.25}{white}{1}}\hspace*{-.3cm},
		&& \hspace*{-.3cm}\raisebox{\rsp}{\AlphaPerm{2,2,2}{c2,c1,c2}{1,3,\overline{2},2,\overline{3},\overline{1}}{.5}{.25}{white}{1}}\hspace*{-.3cm},\\
		& \hspace*{-.3cm}\raisebox{\rsp}{\AlphaPerm{2,2,2}{c2,c1,c2}{\overline{2},\overline{1},\overline{3},3,1,2}{.5}{.25}{white}{1}}\hspace*{-.3cm},
		&& \hspace*{-.3cm}\raisebox{\rsp}{\AlphaPerm{2,2,2}{c2,c1,c2}{\overline{2},1,\overline{3},3,\overline{1},2}{.5}{.25}{white}{1}}\hspace*{-.3cm},
		&& \hspace*{-.3cm}\raisebox{\rsp}{\AlphaPerm{2,2,2}{c2,c1,c2}{\overline{1},2,\overline{3},3,\overline{2},1}{.5}{.25}{white}{1}}\hspace*{-.3cm},
		&& \hspace*{-.3cm}\raisebox{\rsp}{\AlphaPerm{2,2,2}{c2,c1,c2}{1,2,\overline{3},3,\overline{2},\overline{1}}{.5}{.25}{white}{1}}\hspace*{-.3cm}.
	   \end{aligned}$
    \end{center}
\end{example}

%

\subsection{Longest elements in parabolic quotients}
    \label{sec:parabolic_longest_elements}
The next result establishes that the set of minimal length representatives of parabolic quotients behave well with respect to the left weak order.

\begin{theorem}[{\cite[Theorem~4.1]{bjorner88generalized}}]\label{thm:weak_order_quotients}
    Let $W$ be a finite Coxeter group, and let $J\subseteq S$ be a set of simple reflections. There exists an element $w_{\circ}^J\in W^{J}$ such that the weak order $\Weak(W^{J})$ is isomorphic to the interval $[\id,\woa]$ in $\Weak(W)$.  
\end{theorem}

\begin{example}
    We continue Example~\ref{ex-parab} with $\alpha=(0,1,2)$, and observe that its maximal element is $\hspace*{-.2cm}\raisebox{\rsp}{\AlphaPerm{2,1,1,2}{c2,c1,c1,c2}{2,3,1,\overline{1},\overline{3},\overline{2}}{.5}{.25}{white}{1}}\hspace*{-.3cm}$. With $\alpha'=(1,2)$, the maximal element is $\hspace*{-.2cm}\raisebox{\rsp}{\AlphaPerm{2,2,2}{c2,c1,c2}{2,3,\overline{1},1,\overline{3},\overline{2}}{.5}{.25}{white}{1}}\hspace*{-.3cm}$.
\end{example}

It is our declared goal to describe the $\wboa(\linc)$-aligned elements of $\Hyper_{\alpha}$, and we therefore need a good understanding of the \defn{parabolic longest element} $\woa$ (which is the element $w_{\circ}^J$ from Theorem~\ref{thm:weak_order_quotients} for the parabolic quotient $\Hyper_\alpha$) and its $\linc$-sorting word.
The long one-line notation of $\woa$ is obtained by filling the blocks of $\Part(\alpha)$ from left to right with the largest available values so that the increasing property per block from Lemma~\ref{lem:increasing_blocks} is satisfied. This is made precise in the following lemma.

\begin{lemma}\label{lem:parabolic_long_element}
    Let $\alpha$ be a type-$B$ composition, and let $\woa$ be the longest element of $\Hyper_{\alpha}$.  For $a\in[n]$ with $\indicator_{\alpha}(a)=i$, we have
    \begin{displaymath}
        \woa(a) = \begin{cases}a, & \text{if}\;\alpha\;\text{is join and}\;i=1,\\
        -(p_{i}+p_{i-1}+1-a), & \text{otherwise}.\end{cases}
    \end{displaymath}
\end{lemma}
\begin{proof}
    This is a straightforward computation.
\end{proof}

For the description of the $\linc$-sorting word of $\woa$, it will be convenient to model it combinatorially using certain skew shapes.  This is foreshadowed in the arrangement of the inversions of $\wo$ for the full hyperoctahedral group $\Hyper_{n}$ after Corollary~\ref{cor:linc_sorting_longest}.

Given integers $\lambda_{1}\geq\lambda_{2}\geq\cdots$, a \defn{Ferrers diagram} of shape $(\lambda_{1},\lambda_{2},\ldots)$ is a left-aligned arrangement of unit boxes, where the $i\th$ row consists of $\lambda_{i}$ boxes. We use the English notation, \ie the first row is at the top. We usually do not distinguish between the Ferrers diagram and the integer partition it represents.

For two Ferrers diagrams $\lambda,\mu$, the associated \defn{skew diagram} $\lambda/\mu$ consists of all boxes of $\lambda$ that are not boxes of $\mu$. See Figure~\ref{fig:skew_diagram_5} for an illustration.

\begin{figure}[ht]
    \centering
    \begin{tikzpicture}
		\def\x{.4};
		\foreach \k in {0,...,6}{
			\draw(\k*\x,0*\x) -- (\k*\x,5*\x);
		}
		\foreach \k in {0,...,4}{
			\draw(0*\x,\k*\x) -- ({(6+\k)*\x},\k*\x) -- ({(6+\k)*\x},5*\x);
		}
		\draw(0*\x,5*\x) -- (10*\x,5*\x);
		\begin{pgfonlayer}{background}
			\fill[white!50!gray](1*\x,0*\x) -- (6*\x,0*\x) -- (6*\x,1*\x) -- (7*\x,1*\x) -- (7*\x,2*\x) -- (8*\x,2*\x) -- (8*\x,3*\x) -- (9*\x,3*\x) -- (9*\x,4*\x) -- (10*\x,4*\x) -- (10*\x,5*\x) -- (5*\x,5*\x) -- (5*\x,4*\x) -- (4*\x,4*\x) -- (4*\x,3*\x) -- (3*\x,3*\x) -- (3*\x,2*\x) -- (2*\x,2*\x) -- (2*\x,1*\x) -- (1*\x,1*\x) -- (1*\x,0*\x);
		\end{pgfonlayer}
	\end{tikzpicture}
    \caption{The skew shape $\lambda/\mu$, where $\lambda=(10,9,8,7,6)$ and $\mu=(5,4,3,2,1)$, consists of the gray boxes.}
    \label{fig:skew_diagram_5}
\end{figure}

Fix a type-$B$ composition of $n$ with non-zero components $\alpha_{1},\alpha_{2},\ldots,\alpha_{r}$.  We first consider the Ferrers diagram $\mu_{(\alpha)}=(\mu_{1},\mu_{2},\ldots,\mu_{n})$ given by
\begin{displaymath}
    \mu_{k} \defs \alpha_{\indicator_{\alpha}(k)}+\alpha_{\indicator_{\alpha}(k)+1}+\cdots+\alpha_{r}.
\end{displaymath}
In other words, the $k\th$ entry of $\mu$ is the sum of the sizes of the $\alpha$-regions starting from the one containing $k$.  Now, we consider the Ferrers diagram $\lambda_{(\alpha)}=(\lambda_{1},\lambda_{2},\ldots,\lambda_{n})$, where
\begin{displaymath}
    \lambda_{k} \defs
    \begin{cases}
        2n-\alpha_{1}, & \text{if}\;\alpha\;\text{is join and}\;k\leq\alpha_{1},\\
        2n+1-k, & \text{otherwise}.
    \end{cases}
\end{displaymath}
We write $\Skew(\alpha)\defs\lambda_{(\alpha)}/\mu_{(\alpha)}$.

We now construct a filling of $\Skew(\alpha)$ as follows.  Let $k\in[n]$.  If $\alpha$ is split or $k>\alpha_{1}$, then we fill the $k\th$ row of $\Skew(\alpha)$ with the Coxeter generators $s_{0},s_{1},\ldots$ \textit{from right to left} until all the boxes in that row are exhausted.  If $\alpha$ is join and $k\leq\alpha_{1}$, then we fill the $k\th$ row of $\Skew(\alpha)$ with the Coxeter generators $s_{\alpha_{1}-k+1},s_{\alpha_{1}-k+2},\ldots$ \textit{also from right to left} until all the boxes are exhausted. We denote by $\wb_{\alpha}$ the word obtained by reading the Coxeter generators in $\Skew(\alpha)$ from bottom to top, left to right.

\begin{example}
    \label{ex:longest_shapes}
    Consider the split composition $\alpha=(0,3,1,2,1)$. Then $n=7$, $\mu=(7,7,7,4,3,3,1)$, and $\lambda=(14,13,12,11,10,9,8)$. The skew shape $\Skew(\alpha)$ together with its filling is shown below.
    
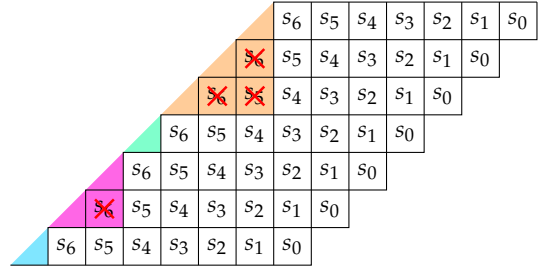
\begin{center}
        \begin{tikzpicture}
		\small
		\def\x{.5};
		\def\n{7};	
		\def\m{6};	
		\foreach \a in {0,...,\m}{
			\draw(\a*\x,0*\x) -- (\a*\x,\a*\x) -- ({(\a+\n+1)*\x)},\a*\x) -- ({(\a+\n+1)*\x},\n*\x);
		}
		\draw(\n*\x,0*\x) -- (\n*\x,\n*\x) -- (2*\n*\x,\n*\x);
		\foreach \a in {0,...,\m}{
			\foreach \b in {0,...,\m}{
				\draw({(\n+\a+1-\b-.5)*\x},{(\a+.5)*\x}) node{$s_{\b}$};
			}
		}
		\crossout{2.5}{1.5}{\x}
		\crossout{5.5}{4.5}{\x}
		\crossout{6.5}{4.5}{\x}
		\crossout{6.5}{5.5}{\x}
		\begin{pgfonlayer}{background}
			\fill[c4](0*\x,0*\x) -- (1*\x,0*\x) -- (1*\x,1*\x) -- cycle;
			\fill[c3](1*\x,1*\x) -- (3*\x,1*\x) -- (3*\x,3*\x) -- cycle;
			\fill[c2](3*\x,3*\x) -- (4*\x,3*\x) -- (4*\x,4*\x) -- cycle;
			\fill[c1](4*\x,4*\x) -- (7*\x,4*\x) -- (7*\x,7*\x) -- cycle;
		\end{pgfonlayer}
	   \end{tikzpicture}
    \end{center}
    We obtain the following reading word (bottom to top, left to right), where we have separated the rows by vertical bars:
    \begin{multline*}
        \wb_{\alpha} = s_{6}s_{5}s_{4}s_{3}s_{2}s_{1}s_{0}
        \mid s_{5}s_{4}s_{3}s_{2}s_{1}s_{0}
        \mid s_{6}s_{5}s_{4}s_{3}s_{2}s_{1}s_{0}
        \mid s_{6}s_{5}s_{4}s_{3}s_{2}s_{1}s_{0}
        \mid\\
            s_{4}s_{3}s_{2}s_{1}s_{0}
        \mid s_{5}s_{4}s_{3}s_{2}s_{1}s_{0}
        \mid s_{6}s_{5}s_{4}s_{3}s_{2}s_{1}s_{0}
    \end{multline*}
    that may also be rewritten as
    \begin{equation}
        \begin{split}
            \wb_{\alpha} &= s_{6\ldots0}\, s_{5\ldots0}\, s_{6\ldots0}\,s_{6\ldots0}\, s_{4\ldots0}\, s_{5\ldots0}\, s_{6\ldots0} 
        \end{split}
    \end{equation}
    This word represents the following sign-symmetric permutation
    \begin{displaymath}
        \woa = \hspace*{-.25cm}\raisebox{\rsp}{\AlphaPerm{1,2,1,3,3,1,2,1}{c4,c3,c2,c1,c1,c2,c3,c4}{7,5,6,4,1,2,3,\overline{3},\overline{2},\overline{1},\overline{4},\overline{6},\overline{5},\overline{7}}{.5}{.25}{white}{1}}\hspace*{-.1cm},
    \end{displaymath}
    which is the longest element of $\Hyper_{\alpha}$.
	
    Next we consider the join composition $\alpha'=(4,2,2)$. In that case, $n=8$, $\mu=(8,8,8,8,4,4,2,2)$, and $\lambda=(12,12,12,12,12,11,10,9)$. The skew shape $\Skew(\alpha')$ together with its filling is shown below.
    \begin{center}
		\begin{tikzpicture}\small
		\def\x{.5};
		\def\n{8};	
		\def\m{7};	
		\def\r{3};
		\def\s{4};
		\foreach \a in {0,...,\r}{
			\draw(\a*\x,0*\x) -- (\a*\x,\a*\x) -- ({(\a+\n+1)*\x)},\a*\x) -- ({(\a+\n+1)*\x},\n*\x);
		}
		\foreach \a in {\s,...,\m}{
			\draw(\a*\x,0*\x) -- (\a*\x,\a*\x) -- (12*\x,\a*\x);
		}
		\draw(\n*\x,0*\x) -- (\n*\x,\n*\x) -- (12*\x,\n*\x);
		\foreach \a in {0,...,\r}{
			\foreach \b in {0,...,\m}{
				\draw({(\n+\a+1-\b-.5)*\x},{(\a+.5)*\x}) node{$s_{\b}$};
			}
		}
		\foreach \a in {\s,...,\m}{
			\setcounter{i}{\n-\a+\s-1}
			\foreach \b in {1,...,\thei}{
				\setcounter{n}{\a+\b-\s}
				\draw({(2*\n-\s+1-\b-.5)*\x},{(\a+.5)*\x}) node{$s_{\then}$};
			}
		}
		\draw(12*\x,3.9*\x) -- (12*\x,8*\x);
		\crossout{1.5}{.5}{\x}
		\crossout{3.5}{2.5}{\x}
		\crossout{5.5}{4.5}{\x}
		\crossout{6.5}{4.5}{\x}
		\crossout{7.5}{4.5}{\x}
		\crossout{6.5}{5.5}{\x}
		\crossout{7.5}{5.5}{\x}
		\crossout{7.5}{6.5}{\x}
		\begin{pgfonlayer}{background}
			\fill[c3](0*\x,0*\x) -- (2*\x,0*\x) -- (2*\x,2*\x) -- cycle;
			\fill[c2](2*\x,2*\x) -- (4*\x,2*\x) -- (4*\x,4*\x) -- cycle;
			\fill[c1](4*\x,4*\x) -- (8*\x,4*\x) -- (8*\x,8*\x) -- cycle;
		\end{pgfonlayer}
	   \end{tikzpicture}
    \end{center}
    The reading word is 
    \begin{multline*}
    \wb_{\alpha'} = s_{6}s_{5}s_{4}s_{3}s_{2}s_{1}s_{0}
        \mid s_{7}s_{6}s_{5}s_{4}s_{3}s_{2}s_{1}s_{0}
        \mid s_{6}s_{5}s_{4}s_{3}s_{2}s_{1}s_{0}
        \mid s_{7}s_{6}s_{5}s_{4}s_{3}s_{2}s_{1}s_{0}
        \mid\\
    s_{4}s_{3}s_{2}s_{1}
        \mid s_{5}s_{4}s_{3}s_{2}
        \mid s_{6}s_{5}s_{4}s_{3}
        \mid s_{7}s_{6}s_{5}s_{4},
    \end{multline*}
    that may also be rewritten as
    \begin{equation}
      \wb_{\alpha'}
        = s_{6\ldots0}\, s_{7\ldots0}\, s_{6\ldots0}\, s_{7\ldots0}\, s_{4\ldots1}\, s_{5\ldots2}\, s_{6\ldots3}\, s_{7\ldots4}.
    \end{equation}
    It represents
    \begin{displaymath}
        \woab = \hspace*{-.2cm}\raisebox{\rsp}{\AlphaPerm{2,2,8,2,2}{c3,c2,c1,c2,c3}{7,8,5,6,\overline{4},\overline{3},\overline{2},\overline{1},1,2,3,4,\overline{6},\overline{5},\overline{8},\overline{7}}{.5}{.25}{white}{1}}\hspace*{-.1cm},
    \end{displaymath}
    which is the longest element of $\Hyper_{\alpha'}$.
\end{example}

The previous examples suggest that $\wb_{\alpha}$ is a candidate for $\wboa(\linc)$. This is indeed true as the next proposition shows, whose proof makes use of Lemma~\ref{lem:sorting_word_negative}.

\begin{proposition}\label{prop:wbwoa}
    For every type-$B$ composition $\alpha$, the word $\wb_{\alpha}$ is the $\linc$-sorting word of $\woa$.
\end{proposition}
\begin{proof}
    If $\alpha$ is split, we already are in the case of Lemma~\ref{lem:sorting_word_negative} with $i_1=0$. If $\alpha$ is join, then the first values on the right part of $w$ are positive, and the first generator one can apply is at the end of the first $\alpha$-region. Then this value goes to the extreme right of the permutation and the same idea applies again to the other elements of the first $\alpha$-region: they move to the right until they get next to a value greater than themselves. This exactly accounts for the first $i_1$ rows of $\Skew(\alpha)$ read from right to left.  Again, each product decreases the length of the permutation, hence proving that we have a reduced word. After that product of generators, we are back with a permutation of the type shown in Lemma~\ref{lem:sorting_word_negative}, hence the result.
\end{proof}

\begin{corollary}\label{cor:parabolic_longest_length}
    If $\alpha=(\alpha_{1},\alpha_{2},\ldots,\alpha_{r})$ is a type-$B$ composition of $n$, then
    \begin{displaymath}
        \ell_S(\woa) = n^2 - \sum_{i=1}^r \binom{\alpha_i}{2} -
        \begin{cases}
            \binom{\alpha_{1}+1}{2}, & \text{if}\;\alpha\;\text{is join},\\
            0, & \text{if}\;\alpha\;\text{is split}.
        \end{cases}
    \end{displaymath}
\end{corollary}
\begin{proof}
    By definition, $\ell_{S}(\woa)$ is equal to the number of letters in the $\linc$-sorting word $\wboa(\linc)$, and by Proposition~\ref{prop:wbwoa} this number is equal to the number of boxes in $\Skew(\alpha)$.
    Then the statement is immediate since we start with a skew arrangement of $n$ rows, each consisting of $n$ boxes, and then we remove a triangular part for each component of $\alpha$ on the left. If $\alpha$ is join, we remove another triangular part on the right.
\end{proof}

Let us now exploit the skew shape used in the construction of $\wboa(\linc)$ to determine the corresponding inversion order $\invorder\bigl(\wboa(\linc)\bigr)$. We start with an example.

\begin{example}\label{ex:inversion_orders}
    We continue on Example~\ref{ex:longest_shapes}. First, let us consider $\alpha=(0,3,1,2,1)$. If we write the $i\th$ inversion in $\invorder(\wboa(\linc)\bigr)$ in the cell, which corresponds to the $(\ell_{S}(\woa)-i+1)\st$ letter of $\wb_{\alpha}$ (recall that the last letter of $\wbo$ is the right-most one on the top row), we obtain the following filling of $\Skew(\alpha)$.
    \begin{center}
	   \begin{tikzpicture}\small
		\def\x{.85};
		\def\y{.5};
		\def\n{7};	
		\def\m{6};	
		\foreach \a in {0,...,\m}{
		\draw(\a*\x,0*\y) -- (\a*\x,\a*\y) -- ({(\a+\n+1)*\x)},\a*\y)
			-- ({(\a+\n+1)*\x},\n*\y);
		}
		\draw(\n*\x,0*\y) -- (\n*\x,\n*\y) -- (2*\n*\x,\n*\y);
		\filldraw[draw=c4,fill=c4](0*\x,0*\y) -- (1*\x,0*\y) -- (1*\x,1*\y) -- cycle;
		\filldraw[draw=c3,fill=c3](1*\x,1*\y) -- (3*\x,1*\y) -- (3*\x,3*\y) -- cycle;
		\filldraw[draw=c2,fill=c2](3*\x,3*\y) -- (4*\x,3*\y) -- (4*\x,4*\y) -- cycle;
		\filldraw[draw=c1,fill=c1](4*\x,4*\y) -- (7*\x,4*\y) -- (7*\x,7*\y) -- cycle;
		\draw(0*\x,0*\y) -- (1*\x,0*\y) -- (1*\x,1*\y) -- (3*\x,1*\y) -- (3*\x,3*\y) -- (4*\x,3*\y) -- (4*\x,4*\y) -- (7*\x,4*\y) -- (7*\x,7*\y) -- (14*\x,7*\y);
		\draw(13.5*\x,6.5*\y) node[scale=.7]{$\sleft 1\sright$};
		\draw(12.5*\x,6.5*\y) node[scale=.7]{$\lleft {-}2\;1\rright$};
		\draw(11.5*\x,6.5*\y) node[scale=.7]{$\lleft {-}3\;1\rright$};
		\draw(10.5*\x,6.5*\y) node[scale=.7]{$\lleft {-}4\;1\rright$};
		\draw(9.5*\x,6.5*\y) node[scale=.7]{$\lleft {-}5\;1\rright$};
		\draw(8.5*\x,6.5*\y) node[scale=.7]{$\lleft {-}6\;1\rright$};
		\draw(7.5*\x,6.5*\y) node[scale=.7]{$\lleft {-}7\;1\rright$};
		\draw(12.5*\x,5.5*\y) node[scale=.7]{$\sleft 2\sright$};
		\draw(11.5*\x,5.5*\y) node[scale=.7]{$\lleft {-}3\;2\rright$};
		\draw(10.5*\x,5.5*\y) node[scale=.7]{$\lleft {-}4\;2\rright$};
		\draw(9.5*\x,5.5*\y) node[scale=.7]{$\lleft {-}5\;2\rright$};
		\draw(8.5*\x,5.5*\y) node[scale=.7]{$\lleft {-}6\;2\rright$};
		\draw(7.5*\x,5.5*\y) node[scale=.7]{$\lleft {-}7\;2\rright$};
		\draw(11.5*\x,4.5*\y) node[scale=.7]{$\sleft 3\sright$};
		\draw(10.5*\x,4.5*\y) node[scale=.7]{$\lleft {-}4\;3\rright$};
		\draw(9.5*\x,4.5*\y) node[scale=.7]{$\lleft {-}5\;3\rright$};
		\draw(8.5*\x,4.5*\y) node[scale=.7]{$\lleft {-}6\;3\rright$};
		\draw(7.5*\x,4.5*\y) node[scale=.7]{$\lleft {-}7\;3\rright$};
		\draw(10.5*\x,3.5*\y) node[scale=.7]{$\sleft 4\sright$};
		\draw(9.5*\x,3.5*\y) node[scale=.7]{$\lleft {-}5\;4\rright$};
		\draw(8.5*\x,3.5*\y) node[scale=.7]{$\lleft {-}6\;4\rright$};
		\draw(7.5*\x,3.5*\y) node[scale=.7]{$\lleft {-}7\;4\rright$};
		\draw(6.5*\x,3.5*\y) node[scale=.7]{$\lleft 3\;4\rright$};
		\draw(5.5*\x,3.5*\y) node[scale=.7]{$\lleft 2\;4\rright$};
		\draw(4.5*\x,3.5*\y) node[scale=.7]{$\lleft 1\;4\rright$};
		\draw(9.5*\x,2.5*\y) node[scale=.7]{$\sleft 5\sright$};
		\draw(8.5*\x,2.5*\y) node[scale=.7]{$\lleft {-}6\;5\rright$};
		\draw(7.5*\x,2.5*\y) node[scale=.7]{$\lleft {-}7\;5\rright$};
		\draw(6.5*\x,2.5*\y) node[scale=.7]{$\lleft 3\;5\rright$};
		\draw(5.5*\x,2.5*\y) node[scale=.7]{$\lleft 2\;5\rright$};
		\draw(4.5*\x,2.5*\y) node[scale=.7]{$\lleft 1\;5\rright$};
		\draw(3.5*\x,2.5*\y) node[scale=.7]{$\lleft 4\;5\rright$};
		\draw(8.5*\x,1.5*\y) node[scale=.7]{$\sleft 6\sright$};
		\draw(7.5*\x,1.5*\y) node[scale=.7]{$\lleft {-}7\;6\rright$};
		\draw(6.5*\x,1.5*\y) node[scale=.7]{$\lleft 3\;6\rright$};
		\draw(5.5*\x,1.5*\y) node[scale=.7]{$\lleft 2\;6\rright$};
		\draw(4.5*\x,1.5*\y) node[scale=.7]{$\lleft 1\;6\rright$};
		\draw(3.5*\x,1.5*\y) node[scale=.7]{$\lleft 4\;6\rright$};
		\draw(7.5*\x,.5*\y) node[scale=.7]{$\sleft 7\sright$};
		\draw(6.5*\x,.5*\y) node[scale=.7]{$\lleft 3\;7\rright$};
		\draw(5.5*\x,.5*\y) node[scale=.7]{$\lleft 2\;7\rright$};
		\draw(4.5*\x,.5*\y) node[scale=.7]{$\lleft 1\;7\rright$};
		\draw(3.5*\x,.5*\y) node[scale=.7]{$\lleft 4\;7\rright$};
		\draw(2.5*\x,.5*\y) node[scale=.7]{$\lleft 6\;7\rright$};
		\draw(1.5*\x,.5*\y) node[scale=.7]{$\lleft 5\;7\rright$};
		\draw(.5*\x,.25*\y)    node[text=gray!50!black,scale=.7]{$-7$};
		\draw(1.5*\x,1.25*\y)  node[text=gray!50!black,scale=.7]{$-5$};
		\draw(2.5*\x,1.25*\y)  node[text=gray!50!black,scale=.7]{$-6$};
		\draw(3.5*\x,3.25*\y)  node[text=gray!50!black,scale=.7]{$-4$};
		\draw(4.5*\x,4.25*\y)  node[text=gray!50!black,scale=.7]{$-1$};
		\draw(5.5*\x,4.25*\y)  node[text=gray!50!black,scale=.7]{$-2$};
		\draw(6.5*\x,4.25*\y)  node[text=gray!50!black,scale=.7]{$-3$};
		\draw(7.5*\x,7.25*\y)  node[text=gray!50!black,scale=.7]{$7$};
		\draw(8.5*\x,7.25*\y)  node[text=gray!50!black,scale=.7]{$6$};
		\draw(9.5*\x,7.25*\y)  node[text=gray!50!black,scale=.7]{$5$};
		\draw(10.5*\x,7.25*\y) node[text=gray!50!black,scale=.7]{$4$};
		\draw(11.5*\x,7.25*\y) node[text=gray!50!black,scale=.7]{$3$};
		\draw(12.5*\x,7.25*\y) node[text=gray!50!black,scale=.7]{$2$};
		\draw(13.5*\x,7.25*\y) node[text=gray!50!black,scale=.7]{$1$};
		\draw(14*\x,6.5*\y)  node[text=gray!50!black,scale=.7,anchor=west]{$-1$};
		\draw(13*\x,5.5*\y)  node[text=gray!50!black,scale=.7,anchor=west]{$-2$};
		\draw(12*\x,4.5*\y)  node[text=gray!50!black,scale=.7,anchor=west]{$-3$};
		\draw(11*\x,3.5*\y)  node[text=gray!50!black,scale=.7,anchor=west]{$-4$};
		\draw(10*\x,2.5*\y)  node[text=gray!50!black,scale=.7,anchor=west]{$-5$};
		\draw(9*\x,1.5*\y)   node[text=gray!50!black,scale=.7,anchor=west]{$-6$};
		\draw(8*\x,.5*\y)    node[text=gray!50!black,scale=.7,anchor=west]{$-7$};
	   \end{tikzpicture}
    \end{center}
    Note that each cell corresponds to the pair $(i,j)$ where $i$ is the label of its row and $j$ the label of its column (written outside of the shape in the picture above) up to taking the opposite of both elements.

    \medskip
    
    Now, let us consider $\alpha'=(4,2,2)$.  As before, we fill $\Skew(\alpha')$ with the inversions according to $\invorder\bigl(\mathbf{w}_{\mathsf{o};\alpha'}(\linc)\bigr)$.
    \begin{center}
	   \begin{tikzpicture}\small
		\def\x{.85};
		\def\y{.5};
		\def\n{8};	
		\def\m{7};	
		\def\r{3};
		\def\s{4};
		\foreach \a in {0,...,\r}{
			\draw(\a*\x,0*\y) -- (\a*\x,\a*\y) -- ({(\a+\n+1)*\x)},\a*\y) -- ({(\a+\n+1)*\x},\n*\y);
		}
		\foreach \a in {\s,...,\m}{
			\draw(\a*\x,0*\y) -- (\a*\x,\a*\y) -- (12*\x,\a*\y);
		}
		\draw(\n*\x,0*\y) -- (\n*\x,\n*\y) -- (12*\x,\n*\y);
		\filldraw[draw=c3,fill=c3](0*\x,0*\y) -- (2*\x,0*\y) -- (2*\x,2*\y) -- cycle;
		\filldraw[draw=c2,fill=c2](2*\x,2*\y) -- (4*\x,2*\y) -- (4*\x,4*\y) -- cycle;
		\filldraw[draw=c1,fill=c1](4*\x,4*\y) -- (8*\x,4*\y) -- (8*\x,8*\y) -- cycle;
		\draw(0*\x,0*\y) -- (2*\x,0*\y) -- (2*\x,2*\y) -- (4*\x,2*\y) -- (4*\x,4*\y) -- (8*\x,4*\y) -- (8*\x,8*\y) -- (12*\x,8*\y);
		\draw(11.5*\x,7.5*\y) node[scale=.7]{$\lleft 4\;5\rright$};
		\draw(10.5*\x,7.5*\y) node[scale=.7]{$\lleft 4\;6\rright$};
		\draw(9.5*\x,7.5*\y) node[scale=.7]{$\lleft 4\;7\rright$};
		\draw(8.5*\x,7.5*\y) node[scale=.7]{$\lleft 4\;8\rright$};
		\draw(11.5*\x,6.5*\y) node[scale=.7]{$\lleft 3\;5\rright$};
		\draw(10.5*\x,6.5*\y) node[scale=.7]{$\lleft 3\;6\rright$};
		\draw(9.5*\x,6.5*\y) node[scale=.7]{$\lleft 3\;7\rright$};
		\draw(8.5*\x,6.5*\y) node[scale=.7]{$\lleft 3\;8\rright$};
		\draw(11.5*\x,5.5*\y) node[scale=.7]{$\lleft 2\;5\rright$};
		\draw(10.5*\x,5.5*\y) node[scale=.7]{$\lleft 2\;6\rright$};
		\draw(9.5*\x,5.5*\y) node[scale=.7]{$\lleft 2\;7\rright$};
		\draw(8.5*\x,5.5*\y) node[scale=.7]{$\lleft 2\;8\rright$};
		\draw(11.5*\x,4.5*\y) node[scale=.7]{$\lleft 1\;5\rright$};
		\draw(10.5*\x,4.5*\y) node[scale=.7]{$\lleft 1\;6\rright$};
		\draw(9.5*\x,4.5*\y) node[scale=.7]{$\lleft 1\;7\rright$};
		\draw(8.5*\x,4.5*\y) node[scale=.7]{$\lleft 1\;8\rright$};
		\draw(11.5*\x,3.5*\y) node[scale=.7]{$\sleft 5\sright$};
		\draw(10.5*\x,3.5*\y) node[scale=.7]{$\lleft {-}6\;5\rright$};
		\draw(9.5*\x,3.5*\y) node[scale=.7]{$\lleft {-}7\;5\rright$};
		\draw(8.5*\x,3.5*\y) node[scale=.7]{$\lleft {-}8\;5\rright$};
		\draw(7.5*\x,3.5*\y) node[scale=.7]{$\lleft {-}5\;1\rright$};
		\draw(6.5*\x,3.5*\y) node[scale=.7]{$\lleft {-}5\;2\rright$};
		\draw(5.5*\x,3.5*\y) node[scale=.7]{$\lleft {-}5\;3\rright$};
		\draw(4.5*\x,3.5*\y) node[scale=.7]{$\lleft {-}5\;4\rright$};
		\draw(10.5*\x,2.5*\y) node[scale=.7]{$\sleft 6\sright$};
		\draw(9.5*\x,2.5*\y) node[scale=.7]{$\lleft {-}7\;6\rright$};
		\draw(8.5*\x,2.5*\y) node[scale=.7]{$\lleft {-}8\;6\rright$};
		\draw(7.5*\x,2.5*\y) node[scale=.7]{$\lleft {-}6\;1\rright$};
		\draw(6.5*\x,2.5*\y) node[scale=.7]{$\lleft {-}6\;2\rright$};
		\draw(5.5*\x,2.5*\y) node[scale=.7]{$\lleft {-}6\;3\rright$};
		\draw(4.5*\x,2.5*\y) node[scale=.7]{$\lleft {-}6\;4\rright$};
		\draw(9.5*\x,1.5*\y) node[scale=.7]{$\sleft 7\sright$};
		\draw(8.5*\x,1.5*\y) node[scale=.7]{$\lleft {-}8\;7\rright$};
		\draw(7.5*\x,1.5*\y) node[scale=.7]{$\lleft {-}7\;1\rright$};
		\draw(6.5*\x,1.5*\y) node[scale=.7]{$\lleft {-}7\;2\rright$};
		\draw(5.5*\x,1.5*\y) node[scale=.7]{$\lleft {-}7\;3\rright$};
		\draw(4.5*\x,1.5*\y) node[scale=.7]{$\lleft {-}7\;4\rright$};
		\draw(3.5*\x,1.5*\y) node[scale=.7]{$\lleft 6\;7\rright$};
		\draw(2.5*\x,1.5*\y) node[scale=.7]{$\lleft 5\;7\rright$};
		\draw(8.5*\x,.5*\y) node[scale=.7]{$\sleft 8\sright$};
		\draw(7.5*\x,.5*\y) node[scale=.7]{$\lleft {-}8\;1\rright$};
		\draw(6.5*\x,.5*\y) node[scale=.7]{$\lleft {-}8\;2\rright$};
		\draw(5.5*\x,.5*\y) node[scale=.7]{$\lleft {-}8\;3\rright$};
		\draw(4.5*\x,.5*\y) node[scale=.7]{$\lleft {-}8\;4\rright$};
		\draw(3.5*\x,.5*\y) node[scale=.7]{$\lleft 6\;8\rright$};
		\draw(2.5*\x,.5*\y) node[scale=.7]{$\lleft 5\;8\rright$};
		\draw(.5*\x,.25*\y)    node[text=gray!50!black,scale=.7]{$-7$};
		\draw(1.5*\x,.25*\y)   node[text=gray!50!black,scale=.7]{$-8$};
		\draw(2.5*\x,2.25*\y)  node[text=gray!50!black,scale=.7]{$-5$};
		\draw(3.5*\x,2.25*\y)  node[text=gray!50!black,scale=.7]{$-6$};
		\draw(4.5*\x,4.25*\y)  node[text=gray!50!black,scale=.7]{$4$};
		\draw(5.5*\x,4.25*\y)  node[text=gray!50!black,scale=.7]{$3$};
		\draw(6.5*\x,4.25*\y)  node[text=gray!50!black,scale=.7]{$2$};
		\draw(7.5*\x,4.25*\y)  node[text=gray!50!black,scale=.7]{$1$};
		\draw(8.5*\x,8.25*\y)  node[text=gray!50!black,scale=.7]{$8$};
		\draw(9.5*\x,8.25*\y)  node[text=gray!50!black,scale=.7]{$7$};
		\draw(10.5*\x,8.25*\y) node[text=gray!50!black,scale=.7]{$6$};
		\draw(11.5*\x,8.25*\y) node[text=gray!50!black,scale=.7]{$5$};
		\draw(12*\x,7.5*\y)  node[text=gray!50!black,scale=.7,anchor=west]{$4$};
		\draw(12*\x,6.5*\y)  node[text=gray!50!black,scale=.7,anchor=west]{$3$};
		\draw(12*\x,5.5*\y)  node[text=gray!50!black,scale=.7,anchor=west]{$2$};
		\draw(12*\x,4.5*\y)  node[text=gray!50!black,scale=.7,anchor=west]{$1$};
		\draw(12*\x,3.5*\y)  node[text=gray!50!black,scale=.7,anchor=west]{$-5$};
		\draw(11*\x,2.5*\y)  node[text=gray!50!black,scale=.7,anchor=west]{$-6$};
		\draw(10*\x,1.5*\y)  node[text=gray!50!black,scale=.7,anchor=west]{$-7$};
		\draw(9*\x,.5*\y)    node[text=gray!50!black,scale=.7,anchor=west]{$-8$};
	   \end{tikzpicture}
    \end{center}
    Note that again each cell corresponds to the pair $(i,j)$ where $i$ is the label of its row and $j$ the label of its column, up to taking the opposite of both elements.
\end{example}

This property is true in general. To state it precisely before proving it, let us construct a filling of $\Skew(\alpha)$ by the inversions of $\woa$. First, we label the rows and the columns of $\Skew(\alpha)$ as follows.
\begin{enumerate}[(i)]
    \item The first $n$ columns (from left to right) are labeled by $\woa(n)$,$\woa(n-1)$, $\ldots$, $\woa(1)$.  The next $n-\alpha_{1}$ columns are labeled by $n, n-1, \ldots, \alpha_1+1$.  If $\alpha$ is split, then we label the remaining $\alpha_{1}$ columns by $\alpha_{1},\alpha_{1}-1,\ldots,1$.  If $\alpha$ is join, there are no columns left.
    \item If $\alpha$ is split, then we label the rows (from top to bottom) by $-1$, $-2$, $\ldots$, $-n$.  If $\alpha$ is join, then we label the first $\alpha_{1}$ rows by $\alpha_{1},\alpha_{1}-1,\ldots,1$, and the remaining $n-\alpha_{1}$ rows by $-\alpha_{1}-1,-\alpha_{1}-2,\ldots,-n$.
\end{enumerate}

\medskip

Now we fill the cells of $\Skew(\alpha)$, and we consider a cell in a row labeled by $r$ and a column labeled by $c$. 
\begin{enumerate}[(i)]
    \item If $r>0$ and $c>0$, then we fill this cell by $\lleft r\;c\rright$. This can only happen if $\alpha$ is join, and then necessarily $r<c$.
    \item If $r<0$ and $c=-r$, then we fill this cell by $\sleft r\sright$. 
    \item If $r<0$ and $c>0$ with $c>-r$, then we fill this cell by $\lleft{-}c\;{-}r\rright$.
    \item If $r<0$ and $c>0$ with $c<-r$, then we fill this cell by $\lleft r\;c\rright$.
    \item If $r<0$ and $c<0$, then we fill this cell by $\lleft {-}c\;{-}r\rright$.  By construction, we are in the first $n$ columns, thus $r<c$.
\end{enumerate}

\begin{proposition}\label{prop:alpha_inversion_order}
    The filling of $\Skew(\alpha)$, read from top to bottom and from right to left, yields the inversion order $\invorder\bigl(\wboa(\linc)\bigr)$.
\end{proposition}

We shall need the following lemma whose proof is immediate by direct computation.

\begin{lemma}
    Consider the product
    \begin{equation}
        \sbo_i^{(n)} \defs s_{n-i\ldots0}\, s_{n-i+1\ldots 0}\, \cdots s_{n-1\ldots 0}.
    \end{equation}
    It sends $k>0$ to 
    \begin{equation}
        \left\{
        \begin{split}
            -n+k-1 & \text{\ \ if\ $k\leq i$}, \\
            k-i    & \text{\ \ if\ $k>i$}. \\
        \end{split}
    \right.
    \end{equation}
\end{lemma}

For example, with $n=7$ and $i=3$, one gets $\sbo_{3}^{(7)}=\overline{4}\;\overline{3}\;\overline{2}\;\overline{1}\;5\;6\;7\mid\overline{7}\;\overline{6}\;\overline{5}\;1\;2\;3\;4$. Its inverse is $1\;2\;3\;\overline{7}\;\overline{6}\;\overline{5}\;\overline{4}\mid 4\;5\;6\;7\;\overline{3}\;\overline{2}\;\overline{1}$.

\begin{proof}[Proof of Proposition~\ref{prop:alpha_inversion_order}]
    Recall that each inversion is obtained by conjugating an $s_i$ and so is itself an reflection, and in particular an involution. And an involution that is not the identity is characterized by the image of an element not sent to itself.

    Let us start with a split composition $\alpha=(0,\alpha_1,\alpha_2,\ldots,\alpha_r)$ of $n$ and compute the inversion corresponding to an $s_i$ in the $k$-th row of part $\alpha_j$ (so that $k\leq \alpha_j$). Note that $i+k\leq n$ since each row has one element less than the previous one inside $\sbo_{\alpha_j}$.

    Let us denote by $\pi$ the permutation defined by the product
    \begin{equation}
        \pi = \sbo^{(n)}_{k-1} \sbo^{(n)}_{\alpha_{j-1}} \dots \sbo^{(n)}_{\alpha_1}.
    \end{equation}
    Then the involution we want to compute is $w=\pi^{-1} s_{0\ldots i-1} s_i s_{i-1\ldots 0} \pi$.

    We are considering an element in row $t=\alpha_1+\alpha_2+\cdots+\alpha_{j-1}+k$ (from the top). Let us compute $w(t)$. First, $\pi(t)=1$: each product $s_{p\ldots0}$ decreases the value by $1$, and since there are $t-1$ such products, we get $1$. Then $s_{0\ldots i-1} s_i s_{i-1\ldots 0}(1)$ is $-i-1$.

    Now, either $k=1$ or $k>1$. If $k>1$ then $s_{0\ldots n-k+1}(-i-1)=-i-2$ since $i+1\leq n-k+1$. We now need to evaluate the remaining rows of $\pi^{-1}$ on $-i-2$ which is equal to the image of $t-1$ by the conjugate of the generator $s_{i+1}$ in row $k-1$ of $\alpha_j$, hence showing that the columns of the skew shape of $\alpha$ are filled with the same values within the same part of the composition.

    Now, if $k=1$, there are again two cases. Either $i+1\leq n-\alpha_{j-1}$ or not. If it is the case, that exactly means that the cell we consider has a cell above it and the same reasoning as before applies. Otherwise, we are in the topmost cell of its column. In that case, $s_{0\ldots \alpha_{j-1}}(-i-1)$ is either $1$ (if $i+1=\alpha_{j-1}+1$) or $-i-1$ itself. Applying the other elements $s_{0\ldots n-1}\cdots s_{0\ldots\alpha_{j-1}-1}$ of $\sbo_{j_1-1}^{-1}$ to this result will eventually get the value $1$ and then increment it at each product by an $s_{0\ldots p}$. Moreover, applying then the other elements of $\pi^{-1}$, it will still increment on each row, finally giving the value $n-i+\alpha_1+\alpha_2+\cdots+\alpha_{j-2}$, which is exactly the value of its column in $\Skew(\alpha)$.

    \medskip
    
    Let us now consider the join case. The labels of the rows of the first part go backwards, which is easily explained (see the proof of Proposition~\ref{prop:wbwoa}). Once that is noted, if the cell under consideration has a cell above it, the same argument as before works and one easily shows that the skew shape predicts the correct involution. If the cell is the topmost of its column, the explanation we saw before in the case of the split composition will apply again with one small difference: the first part behaves differently from the next ones since it is not multiplied on its left by a sequence $\sbo_k$ but by a sequence that does not contain the generator $s_0$. That explains why their images do not change sign, in contrast to other parts of $\alpha$.
\end{proof}

\subsection{$\wboa(\linc)$-aligned elements for parabolic quotients of $\Hyper_{n}$}
    \label{sec:parabolic_aligned_elements}
In this section we want to generalize Lemma~\ref{lem:type_b_forcing} to parabolic quotients of $\Hyper_{n}$.  So, fix a type-$B$ composition $\alpha$, and let $\pi\in\Hyper_{\alpha}$.  Let $t\in\covset(\pi)$.  We use the inversion order $\invorder\bigl(\wboa(\linc)\bigr)$ to describe the required implications.  

In the following proofs we will simply write ``aligned'' when we mean ``$\wboa(\linc)$-aligned''.  Moreover, most of the time $\invset(\woa)$ is a proper subset of the set of all reflections of $\Hyper_{n}$.  Therefore, if we have a reflection $t_{b\beta+c\gamma}\in\invset(\woa)$, where $t_{\beta}\notin\invset(\woa)$ or $t_{\gamma}\notin\invset(\woa)$, we do not need to consider the relative order of $\{t_{\beta},t_{b\beta+c\gamma},t_{\gamma}\}\cap\invset(\woa)$ with respect to $\invorder\bigl(\wboa(\linc)\bigr)$.  We then say that we \defn{discard} this decomposition.

\subsubsection{The split case}

\begin{lemma}\label{lem:split_forcing}
    If $\alpha$ is split, then $\pi\in\Hyper_{\alpha}$ is $\wboa(\linc)$-aligned if and only if for every cover inversion $t\in\covset(\pi)$ the following implications hold.
    \begin{itemize}
        \item If $t=\sleft i\sright$ for $0<i$, then $\sleft j\sright\in\invset(\pi)$ for all $1\leq j<i$ with $\indicator_{\alpha}(j)<\indicator_{\alpha}(i)$.
        \item If $t=\lleft i\;k\rright$ for $0<i<k$, then $\lleft i\;j\rright\in\invset(\pi)$ for all $i<j<k$ with $\indicator_{\alpha}(i)<\indicator_{\alpha}(j)<\indicator_{\alpha}(k)$.
        \item If $t=\lleft {-k}\;i\rright$ for $0<i<k$, then:
        \begin{itemize}
            \item $\sleft i\sright\in\invset(\pi)$,
            \item $\lleft {-}j\;i\rright\in\invset(\pi)$ for $1\leq j<k$ with $\indicator_{\alpha}(j)<\indicator_{\alpha}(k)$,
            \item $\lleft {-}k\;j\rright\in\invset(\pi)$ for $1\leq j<i$ with $\indicator_{\alpha}(j)<\indicator_{\alpha}(i)$.
        \end{itemize}
    \end{itemize}
\end{lemma}
\begin{proof}
    If $t=\sleft i\sright$ for $0<i$, then, since $t$ is a cover inversion of $\pi$, we have $\pi(i)=-1$ by Lemma~\ref{lem:type_b_inversions}.  By Lemma~\ref{lem:type_b_decompositions} we have $t=\lleft j\;i\rright+\sleft j\sright$ for $1\leq j<i$.  If $\indicator_{\alpha}(j)=\indicator_{\alpha}(i)$, then $\lleft j\;i\rright\notin\invset(\woa)$ and we may discard this decomposition. Otherwise $\sleft j\sright$ appears in a row above $\sleft i\sright$ and $\lleft j\;i\rright$ is to the left of $\sleft i\sright$ in the same row. Therefore, we have $\sleft j\sright\prec t\prec\lleft j\;i\rright$, and for $\pi$ to be aligned, we need to have $\sleft j\sright\in\invset(\pi)$.

    \medskip

    If $t=\lleft i\;k\rright$ for $0<i<k$, then $\pi(i)=\pi(k)^+$ by Lemma~\ref{lem:type_b_inversions}, which means that $\indicator_{\alpha}(i)<\indicator_{\alpha}(k)$.  By Lemma~\ref{lem:type_b_decompositions}, we have $t=\lleft i\;j\rright+\lleft j\;k\rright$ for $i<j<k$.  By construction, $\lleft i\;j\rright$ is in the row labeled $-j$ and the column labeled by $-i$, while $\lleft j\;k\rright$ is in the row labeled $-k$ and in the column labeled by $-j$.  Now, $j<k$ implies that $\lleft i\;j\rright$ lies above $t$ in $\Skew(\alpha)$, while $\lleft j\;k\rright$ lies to the left of $t$.  We thus get the order $\lleft i\;j\rright\prec t\prec\lleft j\;k\rright$.  So, for $\pi$ to be aligned we need that $\lleft i\;j\rright\in\invset(\pi)$ whenever $\indicator_{\alpha}(i)<\indicator_{\alpha}(j)<\indicator_{\alpha}(k)$. Indeed, if $\indicator_{\alpha}(i)=\indicator_{\alpha}(j)$, then $\lleft i\;j\rright\notin\invset(\woa)$ so that we may discard the corresponding decomposition.  Analogously, we may discard the case $\indicator_{\alpha}(j)=\indicator_{\alpha}(k)$.

    \medskip

    If $t=\lleft {-}k\;i\rright$ for $0<i<k$, then $\pi(-i)=\pi(k)^+$ by Lemma~\ref{lem:type_b_inversions}.  By Lemma~\ref{lem:type_b_decompositions}, there are four decompositions to consider.
	
    If we take $t=\sleft i\sright+\sleft k\sright$, then $i<k$ implies $\sleft i\sright\prec t\prec\sleft k\sright$, so for $\pi$ to be aligned we need $\sleft i\sright\in\invset(\pi)$.  
	
    If we take $t=2\sleft i\sright+\lleft i\;k\rright$, then $i<k$ implies $\sleft i\sright\prec t\prec\lleft i\;k\rright$. Once again, for $\pi$ to be aligned we need $\sleft i\sright\in\invset(\pi)$.

    Say that $t=\lleft {-}j\;i\rright+\lleft j\;k\rright$ for $1\leq j<k$ and $j\neq i$.  If $\indicator_{\alpha}(j)=\indicator_{\alpha}(k)$, then $\lleft j\;k\rright\notin\invset(\woa)$, so that we may discard this decomposition. Otherwise, $\lleft {-}j\;i\rright$ appears in a row above $\lleft j\;k\rright$, because $i<k$.  Indeed, if $i<j$, then $\lleft {-}j\;i\rright$ appears in the row labeled by $-i$, and if $j<i$ then $\lleft {-}j\;i\rright$ appears in the row labeled by $-j$, which is above the row labeled $-i$. Therefore, we have $\lleft {-}j\;i\rright\prec t\prec\lleft j\;k\rright$ and for $\pi$ to be aligned, we need to have $\lleft {-}j\;i\rright\in\invset(\pi)$.
	
    Lastly, say that $t=\lleft j\;i\rright+\lleft {-}k\;j\rright$ for $1\leq j<i$. If $\indicator_{\alpha}(j)=\indicator_{\alpha}(i)$, then $\lleft j\;i\rright\notin\invset(\woa)$ and we may discard this decomposition.  If $\indicator_{\alpha}(j)<\indicator_{\alpha}(i)$, then $\lleft j\;i\rright$ is in the row labeled $-i$ and in the column labeled $-j$, and therefore it is to the left of $t$, which is in the row labeled $-i$ and in the column labeled $k$. Since $j<i<k$, we have that $\lleft {-}k\;j\rright$ is in the row labeled $-j$, which is above the row containing $t$.  Thus, we get $\lleft {-}k\;j\rright\prec t\prec \lleft j\;i\rright$.  For $\pi$ to be aligned we need to have $\lleft {-}k\;j\rright\in\invset(\pi)$.
\end{proof}

\begin{note}\label{note:split_forcing}
    When a relation is implied by another, we also say that it is \textit{forced} by that. Since the inversions are put in a skew partition by rows from the top-right corner cell to the left-bottom one, the alignment condition of Lemma~\ref{lem:split_forcing} can be translated to: any inversion forced by a given inversion is either in the same row as the given inversion and to its right, or in the same column and above it, as one can check on Example~\ref{ex:inversion_orders}.
\end{note}

Combinatorially speaking, we model aligned elements for $\Hyper_{\alpha}$ with $\alpha$ split as follows.

\begin{definition}\label{def:split_pattern}
    Let $\alpha$ be a split type-$B$ composition of $n$. A \defn{type-$B$ $(\alpha,231)$-split pattern} of a sign-symmetric permutation $\pi\in\Hyper_{\alpha}$ is a triple of indices $-n\leq i<j<k\leq n$ such that
    \begin{itemize}
        \item $i, j, k$ are in different regions, with $j > 0$;
        \item $\pi(i) = \pi(k)^+$, $\pi(i) < \pi(j)$.
    \end{itemize}
\end{definition}

\begin{proposition}\label{prop:split_pattern}
    Let $\alpha$ be a split type-$B$ composition of $n$. A sign-symmetric permutation $\pi\in\Hyper_{\alpha}$ is $\wboa(\linc)$-aligned if and only if it does not have a type-$B$ $(\alpha,231)$-split pattern.
\end{proposition}
\begin{proof}
    According to Lemma~\ref{lem:split_forcing}, we only need to show that the violation of any condition is equivalent to the existence of a type-$B$ $(\alpha,231)$-split pattern. We first consider the violation of conditions in Lemma~\ref{lem:split_forcing}.
    \begin{itemize}
        \item There are some $0 < j < i$ with $\indicator_\alpha(j) < \indicator_\alpha(i)$ such that $\sleft i \sright \in \covset(\pi)$ but $\sleft j \sright \notin \invset(\pi)$. Then we have $\pi(i) = -1$, $\pi(-i) = 1$ and $\pi(j) > 0$. Thus, $(-i, j, i)$ is a type-$B$ $(\alpha,231)$-split pattern.
        \item There are some $0 < i < j < k$ with $\indicator_\alpha(i) < \indicator_\alpha(j) < \indicator_\alpha(k)$ such that $\invd{i}{k} \in \covset(\pi)$ but $\invd{i}{j} \notin \invset(\pi)$. In this case, we have $\pi(i) = \pi(k)^+$, but $\pi(i) < \pi(j)$, and $(i, j, k)$ is thus a type-$B$ $(\alpha,231)$-split pattern.
        \item There are some $0 < i < k$ such that $\invd{-k}{i} \in \covset(\pi)$. In this case we have $\pi(-k) = \pi(i)^+$ and $\pi(-i) = \pi(k)^+$, then we have several subcases.
        \begin{itemize}
            \item Either $\invs{i} \notin \invset(\pi)$, in which we have $\pi(-i) < \pi(i)$, then we have $\indicator_\alpha(i) < \indicator_\alpha(k)$ as $\pi(i) > \pi(-i) > \pi(k)$. Therefore, $(-i, i, k)$ is a type-$B$ $(\alpha,231)$-split pattern.
            \item Or there is some $0 < j < k$ such that $\invd{-j}{i} \notin \invset(\pi)$. We thus have $\pi(-i) < \pi(j)$, and therefore $\indicator_\alpha(j) < \indicator_\alpha(k)$, as $\pi(j) > \pi(-i) > \pi(k)$. Therefore, $(-i, j, k)$ is a type-$B$ $(\alpha,231)$-split pattern.
            \item Or there is some $0 < j < i$ such that $\invd{-k}{j} \notin \invset(\pi)$. We thus have $\pi(j) > \pi(-k) > \pi(i)$, and as $\pi(-k) = \pi(i)^+$, we also have $\indicator_\alpha(j) < \indicator_\alpha(i)$. Therefore, $(-k, i, j)$ is a type-$B$ $(\alpha,231)$-split pattern.
        \end{itemize}
    \end{itemize}

    Now assume that we have a type-$B$ $(\alpha,231)$-split pattern $(i, j, k)$. Then we have $\invd{i}{k} \in \covset(\pi)$ but $\invd{i}{j} \notin \invset(\pi)$. There are several cases: either $i > 0$, or when $i < 0$, the positive integer $-i$ may fall in one of the intervals $0 < -i < j$, $-i = j$, $j < -i < k$, $-i = k$, or $-i > k$. We routinely check that each case leads to a violation of conditions in Lemma~\ref{lem:split_forcing}. We thus have the equivalence.
\end{proof}

\begin{example}\label{ex:split_patterns}
    Let us illustrate Proposition~\ref{prop:split_pattern}, with examples of the three types of type-$B$ $(\alpha,231)$-patterns with respect to the split composition $\alpha=(0,3,1,2,1)$.  The inversion order $\invorder\bigl(\wboa(\linc)\bigr)$ is shown in Example~\ref{ex:inversion_orders}.  We consider the following three members of $\Hyper_{\alpha}$:
    \begin{align*}
        \pi_{1} & = \hspace*{-.25cm}\raisebox{\rsp}{\AlphaPerm{1,2,1,3,3,1,2,1}{c4,c3,c2,c1,c1,c2,c3,c4}{4,\overline{7},\overline{3},\hgl{1},\overline{6},\overline{2},5,\overline{5},2,\hgl{6},\hgl{\overline{1}},3,7,\overline{4}}{.5}{.25}{white}{1}}\hspace*{-.1cm},\\
        \pi_{2} & = \hspace*{-.25cm}\raisebox{\rsp}{\AlphaPerm{1,2,1,3,3,1,2,1}{c4,c3,c2,c1,c1,c2,c3,c4}{4,\overline{5},\overline{1},6,\overline{2},3,7,\overline{7},\hgl{\overline{3}},2,\overline{6},1,\hgl{5},\hgl{\overline{4}}}{.5}{.25}{white}{1}}\hspace*{-.1cm},\\
        \pi_{3} & = \hspace*{-.25cm}\raisebox{\rsp}{\AlphaPerm{1,2,1,3,3,1,2,1}{c4,c3,c2,c1,c1,c2,c3,c4}{\overline{7},\overline{6},\hgl{3},\overline{2},\overline{5},\overline{4},\overline{1},1,\hgl{4},5,\hgl{2},\overline{3},6,7}{.5}{.25}{white}{1}}\hspace*{-.1cm}.
    \end{align*}

    In each case, the highlighted positions constitute a type-$B$ $(\alpha,231)$-split pattern.

    \begin{itemize}
        \item For $\pi_{1}$, we have $\sleft 4\sright\in\covset(\pi_{1})$ and $\lleft 3\;4\rright\in\invset(\pi_{1})$, but $\sleft 3\sright\notin\invset(\pi_{1})$. Moreover, we have $\sleft 4\sright=\lleft3\;4\rright+\sleft 3\sright$ and $\sleft 3\sright\prec\sleft 4\sright\prec\lleft 3\;4\rright$.  It follows that $\pi_{1}$ is not $\wboa(\linc)$-aligned.
        \item For $\pi_{2}$, we have $\lleft 2\;7\rright\in\covset(\pi_{2})$ and $\lleft 6\;7\rright\in\invset(\pi_{2})$, while $\lleft 2\;6\rright\notin\invset(\pi_{2})$.  Since $\lleft 2\;7\rright=\lleft 2\;6\rright+\lleft 6\;7\rright$ and $\lleft 2\;6\rright\prec\lleft 2\;7\rright\prec\lleft 6\;7\rright$, $\pi_{2}$ is not $\wboa(\linc)$-aligned. 
        \item For $\pi_{3}$, we have $\lleft{-}5\;4\rright\in\covset(\pi_{3})$ and $\lleft2\;4\rright\in\invset(\pi_{3})$, while $\lleft{-}5\;2\rright\notin\invset(\pi_{3})$.  At the same time, we have $\lleft{-}5\;4\rright=\lleft {-}5\;2\rright+\lleft 2\;4\rright$ and $\lleft{-}5\;2\rright\prec\lleft {-}5\;4\rright\prec\lleft 2\;3\rright$ in $\invorder\bigl(\wboa(\linc)\bigr)$.  Thus, $\pi_{3}$ is not $\wboa(\linc)$-aligned.
    \end{itemize}
\end{example}

\subsubsection{The join case}

Let us now consider the join case.

\begin{lemma}\label{lem:join_first_block_positive}
    If $\alpha$ is join, then for any $\pi\in\Hyper_{\alpha}$ we have $\pi(i)>0$ for $i\in[\alpha_{1}]$.
\end{lemma}
\begin{proof}
    Let $i\in[\alpha_{1}]$. By definition, we have $\pi(i)=-\pi(-i)$. Since $\alpha$ is join, $-i$ and $i$ lie in the same part of $\Part(\alpha)$, and by Lemma~\ref{lem:increasing_blocks} we must have $\pi(i)> \pi(-i) = -\pi(i)$. This implies $\pi(i)>0$.
\end{proof}

\begin{lemma}\label{lem:join_forcing}
    If $\alpha$ is join, then $\pi\in\Hyper_{\alpha}$ is $\wboa(\linc)$-aligned if and only if for every cover inversion $t\in\covset(\pi)$ the following implications hold.
    \begin{itemize}
        \item If $t=\sleft i\sright$ for $\alpha_{1}<i$, then $\sleft j\sright\in\invset(\pi)$ for $\alpha_{1}<j<i$ with $\indicator_{\alpha}(j)<\indicator_{\alpha}(i)$.
        \item If $t=\lleft i\;k\rright$ for $0<i<k$, then $\lleft i\;j\rright\in\invset(\pi)$ for $i<j<k$ with $\indicator_{\alpha}(i)<\indicator_{\alpha}(j)<\indicator_{\alpha}(k)$.
        \item If $t=\lleft {-}k\;i\rright$ for $0<i<k$, then:
        \begin{itemize}
            \item $\sleft i\sright\in\invset(\pi)$ for $i>\alpha_{1}$,
            \item $\lleft j\;k\rright\in\invset(\pi)$ for $1\leq j\leq \alpha_{1}$, $j\neq i$,
            \item $\lleft {-j}\;i\rright\in\invset(\pi)$ for $j>\alpha_{1}$, $j\neq i$,
            \item $\lleft j\;i\rright\in\invset(\pi)$ for $1\leq j\leq \alpha_{1}<i$,
            \item $\lleft {-k}\;j\rright\in\invset(\pi)$ for $i>j>\alpha_{1}$.
        \end{itemize}
    \end{itemize}
\end{lemma}
\begin{proof}
    If $t=\sleft i\sright$ for $0<i$, then $\pi(i)=-1$ by Lemma~\ref{lem:type_b_inversions}.  By Lemma~\ref{lem:join_first_block_positive}, we have $i>\alpha_{1}$, and by Lemma~\ref{lem:type_b_decompositions}, we have $t=\lleft j\;i\rright+\sleft j\sright$ for $1\leq j<i$.  If $\indicator_{\alpha}(j)=\indicator_{\alpha}(i)$, then we may discard this decomposition.  If $j\leq\alpha_{1}$, then by Lemma~\ref{lem:join_first_block_positive}, we have $\sleft j\sright\notin\invset(\woa)$, and we may discard this decomposition as well. For $\alpha_{1}<j<i$, we see that $\sleft j\sright$ lies in the row labeled by $-j$, which is above the row labeled $-i$ that contains $\sleft i\sright$. Moreover, $\lleft j\;i\rright$ lies in the row labeled $-i$, to the left of $\sleft i\sright$, which yields $\sleft j\sright\prec t\prec\lleft j\;i\rright$.  Thus, for $\pi$ to be aligned, we must have $\sleft j\sright\in\invset(\pi)$.

    \medskip

    If $t=\lleft i\;k\rright$ for $0<i<k$, then $\pi(i)=\pi(k)^+$ by Lemma~\ref{lem:type_b_inversions}, and therefore $\indicator_{\alpha}(i)<\indicator_{\alpha}(k)$ by Lemma~\ref{lem:increasing_blocks}.  By Lemma~\ref{lem:type_b_decompositions}, $\lleft i\;k\rright=\lleft i\;j\rright+\lleft j\;k\rright$ for $i<j<k$.  As in the proof of Lemma~\ref{lem:split_forcing}, we may discard the cases where $\indicator_{\alpha}(i)=\indicator_{\alpha}(j)$ or $\indicator_{\alpha}(j)=\indicator_{\alpha}(k)$.  Now, there are two options where $\lleft i\;k\rright$ can be.  If $i\leq \alpha_{1}$, then it is in the column labeled by $k$ and in the row labeled by $i$.  By construction, $j>\alpha_{1}$, such that $\lleft i\;j\rright$ lies to the right of $\lleft i\;k\rright$ in the same row, and $\lleft j\;k\rright$ lies in the row labeled $-k$ and in the column labeled $-j$, which is below $\lleft i\;k\rright$.  If $i>\alpha_{1}$, then $\lleft i\;k\rright$ lies in the row labeled $-k$ and in the column labeled $-i$.  Moreover, $\lleft i\;j\rright$ lies in the row labeled $-j$ (which is above the row labeled $-k$) and $\lleft j\;k\rright$ lies in the row labeled $-k$ and in the column labeled $-j$, to the left of $\lleft i\;k\rright$.  Thus, in both cases we have $\lleft i\;j\rright\prec t\prec\lleft j\;k\rright$, so that for $\pi$ to be aligned, we need $\lleft i\;j\rright\in\invset(\pi)$.

    \medskip

    If $t=\lleft {-}k\;i\rright$ for $0<i<k$, then $\pi(-i)=\pi(k)^+$ by Lemma~\ref{lem:type_b_inversions}.  By Lemma~\ref{lem:type_b_decompositions}, there are four decompositions to consider.
	
    If we take $t=\sleft i\sright+\sleft k\sright$, then Lemma~\ref{lem:join_first_block_positive} allows us to discard the case $i\leq\alpha_{1}$.  If $i>\alpha_{1}$, then by construction $\sleft i\sright$ lies above $\sleft k\sright$, and $\sleft i\sright$ lies in the same row as $\lleft {-}k\;i\rright$ but to the right.  Thus we have $\sleft i\sright\prec t\prec \sleft k\sright$, and for $\pi$ to be aligned, we need to have $\sleft i\sright\in\invset(\pi)$.
	
    If we take $t=2\sleft i\sright+\lleft i\;k\rright$, then we may again discard the case $i\leq\alpha_{1}$.  Since $k>i>\alpha_{1}$ we conclude that $\sleft i\sright\prec t\prec \lleft i\;k\rright$.
	
    Say that $t=\lleft {-}j\;i\rright+\lleft j\;k\rright$ for $1\leq j<k$ and $j\neq i$.  Once more, we can discard the case $\indicator_{\alpha}(j)=\indicator_{\alpha}(k)$.  If $j\leq \alpha_{1}$, then $i>\alpha_{1}$, because otherwise $\woa(-j)<\woa(i)$ by Lemma~\ref{lem:increasing_blocks}, so that we can discard this decomposition. Thus, $\lleft {-}k\;i\rright$ is in the row labeled $-i$.  Now, $\lleft j\;k\rright$ is in the row labeled $j$, and thus above $t$.  Since $j\leq\alpha_{1}$, $\lleft {-}j\;i\rright$ is in the row labeled $-i$ but to the left of $t$.  Therefore, we have $\lleft j\;k\rright\prec t\prec\lleft{-}j\;i\rright$.  If $j>\alpha_{1}$, then $\lleft{-}j\;i\rright$ is in the column labeled by $-j$ and $\lleft j\;k\rright$ is in the column labeled by $-k$.  If $i\leq\alpha_{1}$, then $\lleft{-}k\;i\rright$ is in the row labeled $-k$ but to the right of $\lleft j\;k\rright$, and if $i>\alpha_{1}$, then $\lleft{-}k\;i\rright$ is in the row labeled $-i$, which is above the row labeled $-j$.  Thus, we get $\lleft {-}j\;i\rright\prec t\prec\lleft j\;k\rright$.  Concluding, for $\pi$ to be aligned we need $\lleft j\;k\rright\in\invset(\pi)$ when $j\leq\alpha_{1}$, and $\lleft{-j}\;i\rright\in\invset(\pi)$ when $j>\alpha_{1}$.
	
    Lastly, say that $t=\lleft j\;i\rright+\lleft{-}k\;j\rright$ for $1\leq j<i$. If $\indicator_{\alpha}(i)=\indicator_{\alpha}(j)$, then we can discard this decomposition.  Thus, $\indicator_{\alpha}(j)<\indicator_{\alpha}(i)$.  In any case, $i>\alpha_{1}$, which means that $\lleft{-}k\;i\rright$ is in the row labeled by $-i$.  If $j\leq\alpha_{1}$, then $\lleft j\;i\rright$ is in the row labeled $j$, and $\lleft{-k}\;j\rright$ is in the row labeled by $-k$. This yields $\lleft j\;i\rright\prec t\prec\lleft{-}k\;i\rright$.  If $j>\alpha_{1}$, then $\lleft j\;i\rright$ is in the row labeled $-i$ but to the left of $t$, and $\lleft{-}k\;j\rright$ is in the row labeled by $-j$. Thus, we get $\lleft{-}k\;j\rright\prec t\prec\lleft j\;i\rright$.  Hence, for $\pi$ to be aligned we need $\lleft j\;i\rright\in\invset(\pi)$ if $j\leq\alpha_{1}$ and $\lleft{-}k\;j\rright \in \invset(\pi)$ if $j>\alpha_{1}$.
\end{proof}

\begin{note}\label{note:join_forcing}
    As in the split case, when a relation is implied by another, we also say that it is \textit{forced} by that one. Since the inversions are put in a skew partition by rows from the top-right corner cell to the left-bottom one, the alignment condition of Lemma~\ref{lem:join_forcing} may be translated to: any inversion forced by a given inversion is either on the same row as the given inversion and to its right, or in the same column and above it, as one can check on Example~\ref{ex:inversion_orders}.
\end{note}

Combinatorially speaking, we model aligned elements for $\Hyper_\alpha$ when
$\alpha$ is join as follows.

\begin{definition}\label{def:join_pattern}
    Let $\alpha$ be a join type-$B$ composition of $n$. A \defn{type-$B$ $(\alpha,231)$-join pattern} in a sign-symmetric permutation $\pi\in\Hyper_{\alpha}$ is a triple of indices $-n\leq i<j<k\leq n$ such that 
    \begin{itemize}
        \item $i, j, k$ are in different regions, with $j > 0$;
        \item $\pi(i) = \pi(k)^+$;
        \item Either $j > \alpha_1$ and $\pi(j) > \pi(i)$, or $0 < j \leq \alpha_1$ and $\pi(j) < \pi(k)$.
    \end{itemize}
\end{definition}

Comparing with the definition of split patterns, we see that the only difference in the join case is that, when the middle index $j$ is in the first region, the join one, $\pi(j)$ must be \textit{smaller} than the two consecutive values $\pi(i), \pi(k)$ instead of being larger. In this situation, we actually see a $312$-pattern.

\begin{proposition} \label{prop:join_pattern}
    Let $\alpha$ be a join type-$B$ composition of $n$. A sign-symmetric permutation $\pi\in\Hyper_{\alpha}$ is $\wboa(\linc)$-aligned if and only if it does not have a type-$B$ $(\alpha,231)$-join pattern.
\end{proposition}
\begin{proof}
    Comparing Lemma~\ref{lem:split_forcing} and Lemma~\ref{lem:join_forcing} and using the same arguments in the proof of Proposition~\ref{prop:split_pattern} and the fact that $\pi \in \Hyper_\alpha$, we observe that $\pi$ has a type-$B$ $(\alpha,231)$-join pattern $(i, j, k)$ with $i < 0$ and $j > \alpha_1$ if and only if $\pi$ violates one of the conditions in Lemma~\ref{lem:join_forcing}, except for the following.
    \begin{itemize}
        \item There are some $0 < i' < k'$ and $0 < j' \leq \alpha_1$ with $j' \neq i'$ such that $\invd{-k'}{i'} \in \covset(\pi)$ but $\invd{j'}{k'} \notin \invset(\pi)$.
        \item There are some $0 < i' < k'$ and $0 < j' \leq \alpha_1$ with $i' > \alpha_1$ such that $\invd{-k'}{i'} \in \covset(\pi)$ but $\invd{j'}{i'} \notin \invset(\pi)$.
    \end{itemize}
    We now only need to show that the violation of the conditions above is equivalent to the existence of a type-$B$ $(\alpha,231)$-join pattern $(i, j, k)$ with $i < 0$ and $0 < j \leq \alpha_1$. 

    When the first condition is violated, we have $\pi(-i') = \pi(k')^+$ and $\pi(j') < \pi(k')$. It is not possible that $k'$ is also in the first region with $j'$, as then $i'$ and $-i'$ would also be in the same region, violating $\pi(-i') > \pi(k')$. Therefore, $(-i', j', k')$ is a type-$B$ $(\alpha,231)$-join pattern for $\pi$. When the second condition is violated, we have $\pi(-k') = \pi(i')^+$ and $\pi(j') < \pi(i')$. As $i' > \alpha_1$, we know that $i'$ and $j'$ are not in the same region. We thus have a type-$B$ $(\alpha,231)$-join pattern $(-k', j', i')$ for $\pi$.

    Now assume that we have in $\pi$ a type-$B$ $(\alpha,231)$-join pattern $(i, j, k)$ with $i < 0$ and $0 < j \leq \alpha_1$. By definition, we have $\invd{-k}{-i} \in \covset(\pi)$ but $\invd{j}{k} \notin \invset(\pi)$. If $0 < -i \leq k$, then $\pi$ violating the first condition above with $i' = -i$, $j' = j$ and $k' = k$. Note that in this case we cannot have $j = -i$, as then $i, -i$ would be in the first region, and we would have $\pi(k) < \pi(i) < \pi(j)$, which is impossible. If $-i > k$, then it is the second condition that is violated with $i' = k$, $j' = j$ and $k' = -i$. We thus have the equivalence.
\end{proof}

\begin{example}
    We now illustrate Proposition~\ref{prop:join_pattern} using the composition $\alpha'=(4,2,2)$.  The corresponding inversion order is shown in Example~\ref{ex:inversion_orders}.  It remains to illustrate examples of $(\alpha',231)$-join patterns, where the middle index is in the first $\alpha$-region.  The other cases are analogous to the split case.  We consider the following permutations of $\Hyper_{\alpha'}$:
    \begin{align*}
        \pi_{1} & =
        \hspace*{-.25cm}\raisebox{\rsp}{\AlphaPerm{2,2,8,2,2}{c3,c2,c1,c2,c3}{\hgl{\overline{6}},\overline{4},\overline{8},7,\overline{5},\overline{3},\overline{2},\overline{1},\hgl{1},2,3,5,\hgl{\overline{7}},8,4,6}{.5}{.25}{white}{1}}\hspace*{-.1cm},\\
        \pi_{2} & =
        \hspace*{-.25cm}\raisebox{\rsp}{\AlphaPerm{2,2,8,2,2}{c3,c2,c1,c2,c3}{\overline{7},\hgl{5},\overline{4},3,\overline{8},\overline{6},\overline{2},\overline{1},1,\hgl{2},6,8,\overline{3},\hgl{4},\overline{5},7}{.5}{.25}{white}{1}}\hspace*{-.1cm}.
    \end{align*}
    \begin{itemize}
        \item We have $\lleft{-}8\;5\rright\in\covset(\pi_{1})$ and $\lleft{-}8\;1\rright\in\invset(\pi_{1})$, but $\lleft 1\;8\rright\notin\invset(\pi_{1})$.  We have $\lleft{-}8\;5\rright=\lleft{-}5\;1\rright+\lleft 1\;8\rright$, but $\lleft 1\;8\rright\prec\lleft{-}8\;5\rright\prec\lleft{-}1\;5\rright$, so that $\pi_{1}$ is not $\wboa(\linc)$-aligned.
        \item We see that $\lleft{-}7\;6\rright\in\covset(\pi_{2})$ and $\lleft{-}7\;2\rright\in\invset(\pi_{2})$, but $\lleft2\;6\rright\notin\invset(\pi_{2})$.  It holds that $\lleft{-}7\;6\rright=\lleft{-}7\;2\rright+\lleft 2\;6\rright$ and $\lleft 2\;6\rright\prec\lleft{-}7\;6\rright\prec\lleft{-}7\;2\rright$.  Thus, $\pi_{2}$ is not $\wbo(\linc)$-aligned.
    \end{itemize}
\end{example}

\begin{example}\label{ex:aligned_elements_subtlety}
	Let us come back to the subtlety mentioned in Remark~\ref{rem:cover_vs_inversions}, namely that for the definition of $\wb$-aligned elements, where $\wb$ is the $c$-sorting word of the longest element in a parabolic quotient of $W$, only the cover inversions must pass a certain check.  Let us start with the observation that, in the case where the parabolic quotient is the whole group (\ie $J=\emptyset$), this difference is irrelevant because the conditions are equivalent, see \cite[Section~6.2]{muehle19tamari}.
	
	When $J\neq\emptyset$, however, this difference is indeed crucial. Consider the group $W=\Hyper_3$, and choose $\alpha=(0,2,1)$.  By Proposition~\ref{prop:alpha_inversion_order}, the inversion order $\invorder\bigl(\wboa(\linc)\bigr)$ is 
	\begin{displaymath}
		\invs{1} \prec \invd{-2}{1} \prec \invd{-3}{1} \prec \invs{2} \prec \invd{-3}{2} \prec \invs{3} \prec \invd{2}{3} \prec \invd{1}{3}.
	\end{displaymath}
	Let $w=\hspace*{-.25cm}\raisebox{\rsp}{\AlphaPerm{1,2,2,1}{c2,c1,c1,c2}{2,\overline{3},1,\overline{1},3,\overline{2}}{.5}{.25}{white}{1}}\hspace*{-.2cm}$. Then, we have
	\begin{align*}
		\invset(w) & = \Bigl\{\invs{1},\invd{-3}{1},\invs{3},\invd{2}{3},\invd{1}{3}\Bigr\},\\
		\covset(w) & = \Bigl\{\invs{1},\invd{1}{3}\Bigr\}.
	\end{align*}
	As $\invs{1}$ is a simple reflection and the only decomposition $\invd{1}{3}=\invd{1}{2}+\invd{2}{3}$ must be discarded because $\invd{1}{2}\notin\invset(\woa)$, we conclude that $w$ is $\wboa(\linc)$-aligned. However, the inversion $t=\invs{3}$ can be decomposed as 
	\begin{displaymath}
		\invs{3} = \invs{1}+\invd{1}{3} = \invs{2} + \invd{2}{3},
	\end{displaymath}
	and neither of these decompositions can be discarded, because $\Bigl\{\invs{1},\invs{2},\invd{2}{3},\invd{1}{3}\Bigr\}\subseteq\invset(\woa)$. As we have $\invs{2}\prec\invd{2}{3}$ and $\invs{2}\notin\invset(w)$, the condition from Definition~\ref{def:w_alignment} would fail for $w$ if all inversions were considered. 
	
	It can be verified that the $\wboa(\linc)$-aligned permutation $\hspace*{-.25cm}\raisebox{\rsp}{\AlphaPerm{1,2,2,1}{c2,c1,c1,c2}{3,\overline{1},2,\overline{2},1,\overline{3}}{.5}{.25}{white}{1}}\hspace*{-.2cm}$ exhibits the same behavior, so that only fourteen of the sixteen $\wboa(\linc)$-aligned elements pass the more restrictive condition of checking all inversions in Definition~\ref{def:w_alignment}.
\end{example}

\begin{remark}
	In light of Example~\ref{ex:aligned_elements_subtlety}, we notice that changing the condition from Definition~\ref{def:w_alignment} to checking \emph{all} inversions would in general produce a smaller set of ``aligned'' elements for parabolic quotients of a Coxeter group $W$. The reason why it is proposed in \cite{muehle19tamari} to only verify the defining conditions for cover inversions lies in the numerical coincidence mentioned in \cite[Conjecture~41]{muehle19tamari}, namely that we can construct other sets of ``Coxeter--Catalan''-like objects which share the cardinality of $\Align\bigl(W^J,\wbo^J(c)\bigr)$ but not that of the smaller collection.
\end{remark}

\section{The parabolic Tamari lattice in linear type $B$}
    \label{sec:parabolic_tamari_lattice}
\subsection{A projection map}

Let $\alpha$ be a type-$B$ composition $\alpha$ of $n$. A sign-symmetric permutation $\pi \in \Hyper_\alpha$ is \defn{$(\alpha, 231)$-avoiding} if $\pi$ has no type-$B$ $(\alpha, 231)$-split (resp. $(\alpha, 231)$-join) pattern when $\alpha$ is split (resp. join). We denote by $\Hyper_\alpha(231)$ the set of such sign-symmetric permutations. By Propositions~\ref{prop:split_pattern}~and~\ref{prop:join_pattern}, the set $\Hyper_\alpha(231)$ contains exactly the $\wboa(\linc)$-aligned elements in $\Hyper_\alpha$. This fact leads naturally to the following definition.

\begin{definition} \label{def:parabolic_Tamari}
    For a type-$B$ composition $\alpha$ of $n$, we define the associated \defn{type-$B$ parabolic Tamari lattice} by $\Tamari_{B}(\alpha) \defs \Weak(\Hyper_\alpha(231))$.
\end{definition}

To show that $\Tamari_B(\alpha)$ is indeed a lattice, we shall construct a surjective map from $\Hyper_{\alpha}$ to $\Hyper_{\alpha}(231)$ which we then use to prove the lattice property. We need the following lemma in our type-$B$ case, which is similar to Lemma~12 in \cite{muehle19tamari} for the type-$A$ case.

\begin{lemma} \label{lem:unique_maximal}
    For every $\pi \in \Hyper_\alpha$, there is a unique element $\pi_\downarrow \in \Hyper_\alpha(231)$ with $\invset(\pi_\downarrow) \subseteq \invset(\pi)$ such that $\invset(\pi_\downarrow)$ is maximal by inclusion for any such element.
\end{lemma}
\begin{proof}
    We proceed by induction on $|\invset(\pi)|$. When $|\invset(\pi)| = 0$, then $\pi$ is the identity, and the claim is clearly satisfied since $\pi_\downarrow = \pi$. Assume that the claim holds for $\pi \in \Hyper_\alpha$ with $k - 1$ inversions, and consider $\pi \in \Hyper_\alpha$ with $\lvert\invset(\pi)\rvert = k$. If $\pi \in \Hyper_\alpha(231)$, then the claim is again satisfied since $\pi_\downarrow = \pi$. Otherwise, there is some $t \in \covset(\pi)$ that violates the conditions in Lemma~\ref{lem:split_forcing} or Lemma~\ref{lem:join_forcing} according to $\alpha$ being split or join.
    Define $\sigma = t \cdot \pi$, so that $|\invset(\sigma)| = |\invset(\pi)| - 1$. Then define $\pi_\downarrow = \sigma_\downarrow$. Let us show that $\pi_\downarrow$ satisfies our claim. By definition, we always have $\pi_\downarrow \in \Hyper_\alpha(231)$. For $\pi' \in \Hyper_\alpha(231)$ with $\invset(\pi') \subset \invset(\pi)$, if $t \notin \invset(\pi')$, then we have $\pi' \weakorder \sigma$, and thus $\invset(\pi') \subseteq \invset(\sigma_\downarrow) = \invset(\pi_\downarrow)$ by induction hypothesis. We shall show that $t \in \invset(\pi')$ is impossible. In the following, by abuse of notation and in accordance with Lemma~\ref{lem:type_b_inversions}, we say that $\invd{a}{b}$ is an inversion of $\pi$ whenever $a - b$ and $\pi(a) - \pi(b)$ take different signs. We suppose that $t = \invd{i}{k}$.

    We first deal with the case where $\alpha$ is split. By Proposition~\ref{prop:split_pattern}, we have a type-$B$ ($\alpha$, 231)-split pattern $(i, j, k)$ in $\pi$ . Without loss of generality, we may assume that $j$ is the smallest such index in its region. We have $\indicator_\alpha(i) < \indicator_\alpha(j) < \indicator_\alpha(k)$. We know that $\invd{i}{k} \notin \covset(\pi')$, as the absent inversion that makes $\invd{i}{k}$ violating the conditions in Lemma~\ref{lem:split_forcing} is still absent in $\pi'$. We thus have $\pi'(i) - \pi'(k) \geq 2$. As $\invd{i}{j} \notin \invset(\pi)$, we also have $\invd{i}{j} \notin \invset(\pi')$, meaning that $\pi'(j) > \pi'(i) > \pi'(k)$.
  
    Let $k_d$ be the index such that $\pi'(k_d) = \pi'(k) + d$. We have $k_0 = k$. Let $d_* = \pi'(i) - \pi'(k) - 1$. We have $\pi'(k_{d_*}) = \pi'(i) - 1$. We now show inductively that $\indicator_\alpha(k_d) > \indicator_\alpha(j)$ for all $0 \leq d \leq d_*$. The case $d=0$ follows from $\indicator_\alpha(j) < \indicator_\alpha(k)$. Suppose that our claim holds for some $0 \leq d < d_*$. It is clear that $k_{d+1} \neq i$. If $\indicator_\alpha(k_{d+1}) < \indicator_\alpha(j)$, then we have $k_{d+1} < j < k_d$, thus $\invd{k_{d+1}}{k_d} \in \invset(\pi')$. As $\pi'(j) > \pi'(i) > \pi'(k_{d+1}) = \pi'(k_d)^+$ and $\indicator_\alpha(j) < \indicator_\alpha(k_d)$ by induction hypothesis, $(k_{d+1}, j, k_d)$ is a type-$B$ ($\alpha$, 231)-split pattern in $\pi'$, which is impossible. If $\indicator_\alpha(k_{d+1}) = \indicator_\alpha(j)$, then we have $\indicator_\alpha(k_{d+1}) < \indicator_\alpha(k)$, but $\pi'(k_{d+1}) > \pi'(k)$. Thus, $\invd{k_{d+1}}{k} \in \invset(\pi') \subseteq \invset(\pi)$, meaning that $\pi(k_{d+1}) > \pi(k)$. We thus have $(i, k_{d+1}, k)$ as a type-$B$ ($\alpha$, 231)-split pattern in $\pi$. However, as $\pi'(k_{d+1}) < \pi'(i) < \pi'(j)$ and $\indicator_\alpha(k_{d+1}) = \indicator_\alpha(j)$, we have $k_{d+1} < j$, which is impossible due to the minimality of $j$. We thus have $\indicator_\alpha(k_{d+1}) > \indicator_\alpha(j)$ to conclude the induction. But this is absurd, as it implies $\indicator_\alpha(i) < \indicator_\alpha(j) < \indicator_\alpha(k_{d_*})$, meaning that $(i, j, k_{d_*})$ is a type-$B$ $(\alpha, 231)$-split pattern in $\pi'$, which is impossible.
  
    When $\alpha$ is join, the only difference is when $\indicator_\alpha(j) = 1$ in the type-$B$ ($\alpha$, 231)-join pattern $(i, j, k)$. In this case, we take $j$ the maximal such index in the same region, and now we have $\pi'(j) < \pi'(k) < \pi'(i)$, as $\invd{j}{k} \notin \invset(\pi) \supseteq \invset(\pi')$. We observe that all patterns in the proof for $\alpha$ split above involving $j$ remain valid, except that for the case $\indicator_\alpha(k_{d+1}) = \indicator_\alpha(j)$. In this case, as $\pi'(k_{d+1}) < \pi'(i)$, we have $\invd{i}{k_{d+1}} \in \invset(\pi') \subset \invset(\pi)$, thus $\pi(k_{d+1}) < \pi(i)$. Therefore, $(i, k_{d+1}, k)$ is still a type-$B$ ($\alpha$, 231)-join pattern in $\pi$. However, as $\pi'(k_{d+1}) > \pi'(k) > \pi'(j)$ and $\indicator_\alpha(k_{d+1}) = \indicator_\alpha(j)$, we have $k_{d+1} > j$, violating the maximality of $j$ here. The modified argument thus applies to $\alpha$ that is join. We thus conclude that such $\pi'$ with $t \in \invset(\pi')$ cannot exist, whether $\alpha$ is split or join. This concludes the induction.
\end{proof}

\begin{remark}
    We saw in the proof of Lemma~\ref{lem:unique_maximal} that, to construct $\pi_\downarrow$ from $\pi$, we only need to spot some type-$B$ ($\alpha$, 231)-split (or join, depending on $\alpha$) pattern $(i, j, k)$, apply the inversion that exchanges $\pi(i)$ and $\pi(k)$, and then repeat the same procedure until there is no longer any such pattern. The result is $\pi_\downarrow$. Lemma~\ref{lem:unique_maximal} also ensures that all choices of inversion lead to the same $\pi_\downarrow$.
\end{remark}

By Lemma~\ref{lem:unique_maximal}, we see $\pi \mapsto \pi_\downarrow$ as a surjective map from $\Hyper_\alpha$ to $\Hyper_\alpha(231)$.

\begin{definition} \label{def:downward_projection}
  We define the downward projection map
  \begin{displaymath}
    \Pi^\alpha_\downarrow \colon \Hyper_\alpha \to \Hyper_\alpha(231), \quad \pi \mapsto \pi_\downarrow,
  \end{displaymath}
  where $\pi_\downarrow$ is the unique element in $\Hyper_\alpha(231)$ with the inversion set $\invset(\pi_\downarrow) \subseteq \invset(\pi)$ being maximal, whose existence is ensured by Lemma~\ref{lem:unique_maximal}.
\end{definition}

We have the following immediate consequence of the proof of Lemma~\ref{lem:unique_maximal}.

\begin{corollary} \label{coro:projection_fixpoint}
    A sign-symmetric permutation $\pi$ is in $\Hyper_\alpha(231)$ if and only if $\Pi_\downarrow^\alpha(\pi) = \pi$.
\end{corollary}

\begin{proposition} \label{prop:parabolic_tamari_lattice}
  The poset $\Tamari_B(\alpha)$ is a lattice.
\end{proposition}

\begin{proof}
    Let $\pi, \pi' \in \Hyper_\alpha(231)$, and let $\sigma$ denote the meet of $\pi$ and $\pi'$ in $\Weak(\Hyper_\alpha)$ (which exists by Theorems~\ref{thm:weak_order_lattice} and \ref{thm:weak_order_quotients}). Then, $\sigma_{\downarrow}\weakorder \pi$ and $\Pi_{\downarrow}^{\alpha}(\sigma)\weakorder \pi'$.

    By Lemma~\ref{lem:unique_maximal}, $\invset(\sigma_{\downarrow})$ is the unique maximal element of the set
    \begin{displaymath}
	   \text{INV}(\sigma) = \bigl\{X\colon\;X\subseteq\invset(\sigma)\;\text{and there is}\:\tau'\in\Hyper_\alpha(231)\;\text{such that}\;\invset(\tau')=X\bigr\}.
    \end{displaymath}
    Let $\tau\in\Hyper_\alpha(231)$ with $\tau\weakorder\pi$ and $\tau\weakorder\pi'$.  Then, necessarily, $\tau\weakorder\sigma$, which also implies that $\invset(\tau)\subseteq\invset(\sigma)$.  It follows that $\invset(\tau)\in\text{INV}(\sigma)$, and thus $\invset(\tau)\subseteq\invset(\sigma_{\downarrow})$. We conclude $\tau\weakorder\sigma_{\downarrow}$ and that $\sigma_{\downarrow}$ is the meet of $\pi$ and $\pi'$ in $\Tamari_{B}(\alpha)$.

    Furthermore, as $\woa$ contains all possible inversions, we have $\Pi^\alpha_\downarrow(\woa) = \woa$. Therefore, $\Tamari_B(\alpha)$ is a finite meet-semilattice with greatest element $\woa$. A classical result \cite[Chapter I, Example~1.27]{gratzer11lattice} thus implies that $\Tamari_B(\alpha)$ is a lattice.
\end{proof}

\begin{figure}[!tbh]
	\centering
	\begin{tikzpicture}
		\def\x{1.25};
		\def\y{1};
		\def\s{.6};
		\def\t{1};	
		\coordinate(n1) at (3*\x,1*\y);
		\coordinate(n2) at (2*\x,2*\y);
		\coordinate(n3) at (3*\x,3*\y);
		\coordinate(n4) at (1*\x,3*\y);
		\coordinate(n5) at (2*\x,4*\y);
		\coordinate(n6) at (0*\x,4*\y);
		\coordinate(n7) at (1*\x,5*\y);
		\coordinate(n8) at (-1*\x,5*\y);
		\coordinate(n9) at (2*\x,6*\y);
		\coordinate(n10) at (0*\x,6*\y);
		\coordinate(n11) at (1*\x,7*\y);
		\coordinate(n12) at (2*\x,8*\y);
		\draw[thick](n1) -- (n2);
		\draw[thick](n2) -- (n3);
		\draw[thick](n2) -- (n4);
		\draw[thick](n3) -- (n5);
		\draw[thick](n4) -- (n5);
		\draw[thick](n4) -- (n6);
		\draw[thick](n5) -- (n7);
		\draw[thick](n6) -- (n8);
		\draw[thick](n7) -- (n9);
		\draw[thick](n7) -- (n10);
		\draw[thick](n8) -- (n10);
		\draw[thick](n9) -- (n11);
		\draw[thick](n10) -- (n11);
		\draw[thick](n11) -- (n12);
		\draw(n1) node[scale=\s]{\AlphaPerm{2,2,2}{c2,c1,c2}{\overline{3},\overline{2},\overline{1},1,2,3}{.5}{.25}{white!50!gray}{\t}};
		\draw(n2) node[scale=\s]{\AlphaPerm{2,2,2}{c2,c1,c2}{\overline{3},\overline{1},\overline{2},2,1,3}{.5}{.25}{white!50!gray}{\t}};
		\draw(n3) node[scale=\s]{\AlphaPerm{2,2,2}{c2,c1,c2}{\overline{2},\overline{1},\overline{3},3,1,2}{.5}{.25}{white!50!gray}{\t}};
		\draw(n4) node[scale=\s]{\AlphaPerm{2,2,2}{c2,c1,c2}{\overline{3},1,\overline{2},2,\overline{1},3}{.5}{.25}{white!50!gray}{\t}};
		\draw(n5) node[scale=\s]{\AlphaPerm{2,2,2}{c2,c1,c2}{\overline{2},1,\overline{3},3,\overline{1},2}{.5}{.25}{white!50!gray}{\t}};
		\draw(n6) node[scale=\s]{\AlphaPerm{2,2,2}{c2,c1,c2}{\overline{3},2,\overline{1},1,\overline{2},3}{.5}{.25}{white!50!gray}{\t}};
		\draw(n7) node[scale=\s]{\AlphaPerm{2,2,2}{c2,c1,c2}{\overline{1},2,\overline{3},3,\overline{2},1}{.5}{.25}{white!50!gray}{\t}};
		\draw(n8) node[scale=\s]{\AlphaPerm{2,2,2}{c2,c1,c2}{\overline{2},3,\overline{1},1,\overline{3},2}{.5}{.25}{white!50!gray}{\t}};
		\draw(n9) node[scale=\s]{\AlphaPerm{2,2,2}{c2,c1,c2}{1,2,\overline{3},3,\overline{2},\overline{1}}{.5}{.25}{white!50!gray}{\t}};
		\draw(n10) node[scale=\s]{\AlphaPerm{2,2,2}{c2,c1,c2}{1,3,\overline{2},2,\overline{3},\overline{1}}{.5}{.25}{white!50!gray}{\t}};
		\draw(n11) node[scale=\s]{\AlphaPerm{2,2,2}{c2,c1,c2}{1,3,\overline{2},2,\overline{3},\overline{1}}{.5}{.25}{white!50!gray}{\t}};
		\draw(n12) node[scale=\s]{\AlphaPerm{2,2,2}{c2,c1,c2}{2,3,\overline{1},1,\overline{3},\overline{2}}{.5}{.25}{white!50!gray}{\t}};
		\begin{pgfonlayer}{background}
			\filldraw[draw=gray,fill=white!50!gray,rounded corners](2.25*\x,.85*\y) -- (3.8*\x,.85*\y) -- (3.8*\x,1.25*\y) -- (2.25*\x,1.25*\y) -- cycle;
			\filldraw[draw=gray,fill=white!50!gray,rounded corners](1.25*\x,1.85*\y) -- (2.8*\x,1.85*\y) -- (2.8*\x,2.25*\y) -- (1.25*\x,2.25*\y) -- cycle;
			\filldraw[draw=gray,fill=white!50!gray,rounded corners](.25*\x,2.85*\y) -- (1.8*\x,2.85*\y) -- (1.8*\x,3.25*\y) -- (.25*\x,3.25*\y) -- cycle;
			\filldraw[draw=gray,fill=white!50!gray,rounded corners](2.25*\x,2.85*\y) -- (3.8*\x,2.85*\y) -- (3.8*\x,3.25*\y) -- (2.25*\x,3.25*\y) -- cycle;
			\filldraw[draw=gray,fill=white!50!gray,rounded corners](1.25*\x,3.85*\y) -- (2.8*\x,3.85*\y) -- (2.8*\x,4.25*\y) -- (1.25*\x,4.25*\y) -- cycle;
			\filldraw[draw=gray,fill=white!50!gray,rounded corners](.25*\x,4.85*\y) -- (1.8*\x,4.85*\y) -- (1.8*\x,5.25*\y) -- (.25*\x,5.25*\y) -- cycle;
			\filldraw[draw=gray,fill=white!50!gray,rounded corners](-.75*\x,5.85*\y) -- (.8*\x,5.85*\y) -- (.8*\x,6.25*\y) -- (-.75*\x,6.25*\y) -- cycle;
			\filldraw[draw=gray,fill=white!50!gray,rounded corners](1.25*\x,5.85*\y) -- (2.8*\x,5.85*\y) -- (2.8*\x,6.25*\y) -- (1.25*\x,6.25*\y) -- cycle;
			\filldraw[draw=gray,fill=white!50!gray,rounded corners](.25*\x,6.85*\y) -- (1.8*\x,6.85*\y) -- (1.8*\x,7.25*\y) -- (.25*\x,7.25*\y) -- cycle;
			\filldraw[draw=gray,fill=white!50!gray,rounded corners](1.25*\x,7.85*\y) -- (2.8*\x,7.85*\y) -- (2.8*\x,8.25*\y) -- (1.25*\x,8.25*\y) -- cycle;
			\filldraw[draw=gray,fill=white!50!gray,rounded corners](-.75*\x,3.85*\y) -- (-.75*\x,4.25*\y) -- (-1.75*\x,4.85*\y) -- (-1.75*\x,5.25*\y) -- (-.2*\x,5.25*\y) -- (-.2*\x,4.85*\y) -- (.8*\x,4.25*\y) -- (.8*\x,3.85*\y) -- cycle;
		\end{pgfonlayer}
	\end{tikzpicture}
    \caption{The weak order on $\Hyper_{(1,2)}$ with the congruence classes with respect to $\Theta$ highlighted.}
    \label{fig:weak_order_12}
\end{figure}

\begin{figure}
    \centering
	\begin{tikzpicture}
		\def\x{1.25};
		\def\y{1};
		\def\s{.6};
		\def\t{1};	
		\draw(5*\x,0*\y) coordinate(n1);
		\draw(3*\x,2*\y) coordinate(n2);
		\draw(7*\x,2*\y) coordinate(n3);
		\draw(2*\x,3*\y) coordinate(n4);
		\draw(6*\x,3*\y) coordinate(n5);
		\draw(8*\x,3*\y) coordinate(n6);
		\draw(1*\x,4*\y) coordinate(n7);
		\draw(3*\x,4*\y) coordinate(n8);
		\draw(5*\x,4*\y) coordinate(n9);
		\draw(7*\x,4*\y) coordinate(n10);
		\draw(2*\x,5*\y) coordinate(n11);
		\draw(4*\x,5*\y) coordinate(n12);
		\draw(6*\x,5*\y) coordinate(n13);
		\draw(8*\x,5*\y) coordinate(n14);
		\draw(3*\x,6*\y) coordinate(n15);
		\draw(5*\x,6*\y) coordinate(n16);
		\draw(7*\x,6*\y) coordinate(n17);
		\draw(9*\x,6*\y) coordinate(n18);
		\draw(2*\x,7*\y) coordinate(n19);
		\draw(4*\x,7*\y) coordinate(n20);
		\draw(8*\x,7*\y) coordinate(n21);
		\draw(3*\x,8*\y) coordinate(n22);
		\draw(7*\x,8*\y) coordinate(n23);
		\draw(5*\x,10*\y) coordinate(n24);
		\draw[edge](n1) -- (n2);
		\draw[edge](n1) -- (n3);
		\draw[edge](n2) -- (n4);
		\draw[edge](n3) -- (n5);
		\draw[edge](n3) -- (n6);
		\draw[edge](n4) -- (n7);
		\draw[edge](n4) -- (n8);
		\draw[edge](n5) -- (n9);
		\draw[edge](n5) -- (n10);
		\draw[edge](n6) -- (n10);
		\draw[edge](n7) -- (n11);
		\draw[edge](n8) -- (n11);
		\draw[edge](n8) -- (n12);
		\draw[edge](n9) -- (n12);
		\draw[edge](n9) -- (n13);
		\draw[edge](n10) -- (n14);
		\draw[edge](n11) -- (n15);
		\draw[edge](n12) -- (n16);
		\draw[edge](n13) -- (n16);
		\draw[edge](n13) -- (n17);
		\draw[edge](n14) -- (n17);
		\draw[edge](n14) -- (n18);
		\draw[edge](n15) -- (n19);
		\draw[edge](n15) -- (n20);
		\draw[edge](n16) -- (n20);
		\draw[edge](n17) -- (n21);
		\draw[edge](n18) -- (n21);
		\draw[edge](n19) -- (n22);
		\draw[edge](n20) -- (n22);
		\draw[edge](n21) -- (n23);
		\draw[edge](n22) -- (n24);
		\draw[edge](n23) -- (n24);
		\draw(n1) node[scale=\s]{\AlphaPerm{2,1,1,2}{c2,c1,c1,c2}{\overline{3},\overline{2},\overline{1},1,2,3}{.5}{.25}{white!50!gray}{\t}};
		\draw(n2) node[scale=\s]{\AlphaPerm{2,1,1,2}{c2,c1,c1,c2}{\overline{3},\overline{2},1,\overline{1},2,3}{.5}{.25}{white!50!gray}{\t}};
		\draw(n3) node[scale=\s]{\AlphaPerm{2,1,1,2}{c2,c1,c1,c2}{\overline{3},\overline{1},\overline{2},2,1,3}{.5}{.25}{white!50!gray}{\t}};
		\draw(n4) node[scale=\s]{\AlphaPerm{2,1,1,2}{c2,c1,c1,c2}{\overline{3},\overline{1},2,\overline{2},1,3}{.5}{.25}{white!50!gray}{\t}};
		\draw(n5) node[scale=\s]{\AlphaPerm{2,1,1,2}{c2,c1,c1,c2}{\overline{3},1,\overline{2},2,\overline{1},3}{.5}{.25}{white!50!gray}{\t}};
		\draw(n6) node[scale=\s]{\AlphaPerm{2,1,1,2}{c2,c1,c1,c2}{\overline{2},\overline{1},\overline{3},3,1,2}{.5}{.25}{white!50!gray}{\t}};
		\draw(n7) node[scale=\s]{\AlphaPerm{2,1,1,2}{c2,c1,c1,c2}{\overline{2},\overline{1},3,\overline{3},1,2}{.5}{.25}{white!50!gray}{\t}};
		\draw(n8) node[scale=\s]{\AlphaPerm{2,1,1,2}{c2,c1,c1,c2}{\overline{3},1,2,\overline{2},\overline{1},3}{.5}{.25}{white!50!gray}{\t}};
		\draw(n9) node[scale=\s]{\AlphaPerm{2,1,1,2}{c2,c1,c1,c2}{\overline{3},2,\overline{1},1,\overline{2},3}{.5}{.25}{white!50!gray}{\t}};
		\draw(n10) node[scale=\s]{\AlphaPerm{2,1,1,2}{c2,c1,c1,c2}{\overline{2},1,\overline{3},3,\overline{1},2}{.5}{.25}{white!50!gray}{\t}};
		\draw(n11) node[scale=\s]{\AlphaPerm{2,1,1,2}{c2,c1,c1,c2}{\overline{2},1,3,\overline{3},\overline{1},2}{.5}{.25}{white!50!gray}{\t}};
		\draw(n12) node[scale=\s]{\AlphaPerm{2,1,1,2}{c2,c1,c1,c2}{\overline{3},2,1,\overline{1},\overline{2},3}{.5}{.25}{white!50!gray}{\t}};
		\draw(n13) node[scale=\s]{\AlphaPerm{2,1,1,2}{c2,c1,c1,c2}{\overline{2},3,\overline{1},1,\overline{3},2}{.5}{.25}{white!50!gray}{\t}};
		\draw(n14) node[scale=\s]{\AlphaPerm{2,1,1,2}{c2,c1,c1,c2}{\overline{1},2,\overline{3},3,\overline{2},1}{.5}{.25}{white!50!gray}{\t}};
		\draw(n15) node[scale=\s]{\AlphaPerm{2,1,1,2}{c2,c1,c1,c2}{\overline{1},2,3,\overline{3},\overline{2},1}{.5}{.25}{white!50!gray}{\t}};
		\draw(n16) node[scale=\s]{\AlphaPerm{2,1,1,2}{c2,c1,c1,c2}{\overline{2},3,1,\overline{1},\overline{3},2}{.5}{.25}{white!50!gray}{\t}};
		\draw(n17) node[scale=\s]{\AlphaPerm{2,1,1,2}{c2,c1,c1,c2}{\overline{1},3,\overline{2},2,\overline{3},1}{.5}{.25}{white!50!gray}{\t}};
		\draw(n18) node[scale=\s]{\AlphaPerm{2,1,1,2}{c2,c1,c1,c2}{1,2,\overline{3},3,\overline{2},\overline{1}}{.5}{.25}{white!50!gray}{\t}};
		\draw(n19) node[scale=\s]{\AlphaPerm{2,1,1,2}{c2,c1,c1,c2}{1,2,3,\overline{3},\overline{2},\overline{1}}{.5}{.25}{white!50!gray}{\t}};
		\draw(n20) node[scale=\s]{\AlphaPerm{2,1,1,2}{c2,c1,c1,c2}{\overline{1},3,2,\overline{2},\overline{3},1}{.5}{.25}{white!50!gray}{\t}};
		\draw(n21) node[scale=\s]{\AlphaPerm{2,1,1,2}{c2,c1,c1,c2}{1,3,\overline{2},2,\overline{3},\overline{1}}{.5}{.25}{white!50!gray}{\t}};
		\draw(n22) node[scale=\s]{\AlphaPerm{2,1,1,2}{c2,c1,c1,c2}{1,3,2,\overline{2},\overline{3},\overline{1}}{.5}{.25}{white!50!gray}{\t}};
		\draw(n23) node[scale=\s]{\AlphaPerm{2,1,1,2}{c2,c1,c1,c2}{2,3,\overline{1},1,\overline{3},\overline{2}}{.5}{.25}{white!50!gray}{\t}};
		\draw(n24) node[scale=\s]{\AlphaPerm{2,1,1,2}{c2,c1,c1,c2}{2,3,1,\overline{1},\overline{3},\overline{2}}{.5}{.25}{white!50!gray}{\t}};
		\begin{pgfonlayer}{background}
			\filldraw[draw=gray,fill=white!50!gray,rounded corners](4.25*\x,-.15*\y) -- (5.8*\x,-.15*\y) -- (5.8*\x,.25*\y) -- (4.25*\x,.25*\y) -- cycle;
			\filldraw[draw=gray,fill=white!50!gray,rounded corners](2.25*\x,1.85*\y) -- (3.8*\x,1.85*\y) -- (3.8*\x,2.25*\y) -- (2.25*\x,2.25*\y) -- cycle;
			\filldraw[draw=gray,fill=white!50!gray,rounded corners](6.25*\x,1.85*\y) -- (7.8*\x,1.85*\y) -- (7.8*\x,2.25*\y) -- (6.8*\x,2.85*\y) -- (6.8*\x,3.25*\y) -- (5.8*\x,3.85*\y) -- (5.8*\x,4.25*\y) -- (6.8*\x,4.85*\y) -- (6.8*\x,5.25*\y) -- (5.25*\x,5.25*\y) -- (5.25*\x,4.85*\y) -- (4.25*\x,4.25*\y) -- (4.25*\x,3.85*\y) -- (5.25*\x,3.25*\y) -- (5.25*\x,2.85*\y) -- (6.25*\x,2.25*\y) -- cycle;
			\filldraw[draw=gray,fill=white!50!gray,rounded corners](1.25*\x,2.85*\y) -- (2.8*\x,2.85*\y) -- (2.8*\x,3.25*\y) -- (1.25*\x,3.25*\y) -- cycle;    
			\filldraw[draw=gray,fill=white!50!gray,rounded corners](7.25*\x,2.85*\y) -- (8.8*\x,2.85*\y) -- (8.8*\x,3.25*\y) -- (7.8*\x,3.85*\y) -- (7.8*\x,4.25*\y) -- (8.8*\x,4.85*\y) -- (8.8*\x,5.25*\y) -- (9.8*\x,5.85*\y) -- (9.8*\x,6.25*\y) -- (8.8*\x,6.85*\y) -- (8.8*\x,7.25*\y) -- (7.8*\x,7.85*\y) -- (7.8*\x,8.25*\y) -- (6.25*\x,8.25*\y) -- (6.25*\x,7.85*\y) -- (7.25*\x,7.25*\y) -- (7.25*\x,6.85*\y) -- (8.25*\x,6.25*\y) -- (8.25*\x,5.75*\y) -- (7.25*\x,5.25*\y) -- (7.25*\x,4.85*\y) -- (6.25*\x,4.25*\y) -- (6.25*\x,3.85*\y) -- (7.25*\x,3.25*\y) -- cycle;
			\fill[white!50!gray,rounded corners](7.25*\x,4.85*\y) -- (7.25*\x,5.25*\y) -- (6.25*\x,5.85*\y) -- (6.25*\x,6.25*\y) -- (7.25*\x,6.85*\y) -- (7.25*\x,7.25*\y) -- (8.8*\x,7.25*\y) -- (8.8*\x,6.85*\y) -- (7.8*\x,6.25*\y) -- (7.8*\x,5.85*\y) -- (8.8*\x,5.25*\y) -- (8.8*\x,4.85*\y) -- cycle;
			\draw[gray,rounded corners](7.25*\x,5*\y) -- (7.25*\x,5.25*\y) -- (6.25*\x,5.85*\y) -- (6.25*\x,6.25*\y) -- (7.25*\x,6.85*\y) -- (7.25*\x,7*\y);
			\draw[gray,rounded corners](8.1*\x,5.68*\y) -- (7.8*\x,5.85*\y) -- (7.8*\x,6.25*\y) -- (8.03*\x,6.38*\y);
			\filldraw[draw=gray,fill=white!50!gray,rounded corners](.25*\x,3.85*\y) -- (1.8*\x,3.85*\y) -- (1.8*\x,4.25*\y) -- (.25*\x,4.25*\y) -- cycle;
			\filldraw[draw=gray,fill=white!50!gray,rounded corners](2.25*\x,3.85*\y) -- (3.8*\x,3.85*\y) -- (3.8*\x,4.25*\y) -- (2.25*\x,4.25*\y) -- cycle;
			\filldraw[draw=gray,fill=white!50!gray,rounded corners](1.25*\x,4.85*\y) -- (2.8*\x,4.85*\y) -- (2.8*\x,5.25*\y) -- (1.25*\x,5.25*\y) -- cycle;
			\filldraw[draw=gray,fill=white!50!gray,rounded corners](3.25*\x,4.85*\y) -- (4.8*\x,4.85*\y) -- (4.8*\x,5.25*\y) -- (5.8*\x,5.85*\y) -- (5.8*\x,6.25*\y) -- (4.25*\x,6.25*\y) -- (4.25*\x,5.85*\y) -- (3.25*\x,5.25*\y) -- cycle;
			\filldraw[draw=gray,fill=white!50!gray,rounded corners](2.25*\x,5.85*\y) -- (3.8*\x,5.85*\y) -- (3.8*\x,6.25*\y) -- (2.25*\x,6.25*\y) -- cycle;
			\filldraw[draw=gray,fill=white!50!gray,rounded corners](1.25*\x,6.85*\y) -- (2.8*\x,6.85*\y) -- (2.8*\x,7.25*\y) -- (1.25*\x,7.25*\y) -- cycle;
			\filldraw[draw=gray,fill=white!50!gray,rounded corners](3.25*\x,6.85*\y) -- (4.8*\x,6.85*\y) -- (4.8*\x,7.25*\y) -- (3.25*\x,7.25*\y) -- cycle;
			\filldraw[draw=gray,fill=white!50!gray,rounded corners](2.25*\x,7.85*\y) -- (3.8*\x,7.85*\y) -- (3.8*\x,8.25*\y) -- (2.25*\x,8.25*\y) -- cycle;
			\filldraw[draw=gray,fill=white!50!gray,rounded corners](4.25*\x,9.85*\y) -- (5.8*\x,9.85*\y) -- (5.8*\x,10.25*\y) -- (4.25*\x,10.25*\y) -- cycle;
		\end{pgfonlayer}
	\end{tikzpicture}
    \caption{The weak order on $\Hyper_{(0,1,2)}$ with the congruence classes with respect to $\Theta$ highlighted.}
    \label{fig:weak_order_012}
\end{figure}

\subsection{$\Tamari_{B}(\alpha)$ is a quotient lattice of the weak order}

To obtain stronger properties on $\Tamari_B(\alpha)$, we now show that it is also a quotient lattice of $\Weak(\Hyper_\alpha)$ using Lemma~\ref{lem:congruences_combin}, by defining an equivalence relation $\Theta$ on $\Hyper_\alpha$ which turns out to be a congruence, using the downward projection $\Pi^\alpha_\downarrow$.  This is illustrated in Figure~\ref{fig:weak_order_12} for $\alpha=(1,2)$ and in Figure~\ref{fig:weak_order_012} for $\alpha=(0,1,2)$.

\begin{definition} \label{def:congruence}
    Let $\Theta$ be the binary relation on $\Hyper_\alpha$ such that $(\pi, \pi') \in \Theta$ if and only if $\Pi^\alpha_\downarrow(\pi) = \Pi^\alpha_\downarrow(\pi')$. It is clear that $\Theta$ is an equivalence relation, with $\Pi^\alpha_\downarrow$ giving the smallest element of each equivalence class.
\end{definition}

We first show that each equivalence class of $\Theta$ is order-convex.

\begin{proposition} \label{prop:theta_order_convex}
    For any $\pi, \sigma, \pi'$ in $\Hyper_\alpha$ with $\pi \weakorder \sigma \weakorder \pi'$, if $\Pi^\alpha_\downarrow(\pi) = \Pi^\alpha_\downarrow(\pi')$, then $\Pi^\alpha_\downarrow(\sigma) = \Pi^\alpha_\downarrow(\pi)$.
\end{proposition}
\begin{proof}
    Let $\pi_\downarrow = \Pi^\alpha_\downarrow(\pi) = \Pi^\alpha_\downarrow(\pi')$ and $\sigma_\downarrow = \Pi^\alpha_\downarrow(\sigma)$. As $\sigma_\downarrow \weakorder \sigma \weakorder \pi'$, by Lemma~\ref{lem:unique_maximal}, we have $\sigma_\downarrow \weakorder \pi_\downarrow$. The same argument on $\pi_\downarrow \weakorder \pi \weakorder \sigma$ gives $\pi_\downarrow \weakorder \sigma_\downarrow$. We thus have $\sigma_\downarrow = \pi_\downarrow$.
\end{proof}

From a dual point of view, we also define an upward projection.

\begin{definition} \label{def:312_pattern}
    Let $\alpha$ be a split type-B composition of $n$. A \defn{type-B $(\alpha,312)$-split pattern} of a sign-symmetric permutation $\pi \in \Hyper_\alpha$ is a triple of indices $-n \leq i < j < k \leq n$ such that
    \begin{itemize}
        \item $i$, $j$, $k$ are in different regions, with $j > 0$;
        \item $\pi(i) = \pi(k)^+, \pi(k) > \pi(j)$.
    \end{itemize}
    For $\alpha$ a join type-B composition of $n$, a \defn{type-B $(\alpha,312)$-join pattern} of $\pi \in \Hyper_\alpha$ is a triple of indices $-n \leq i < j < k\leq n$ such that
    \begin{itemize}
        \item $i$, $j$, $k$ are in different regions, with $j > 0$;
        \item $\pi(i) = \pi(k)^+$;
        \item either $j > \alpha_1$ and $\pi(k) > \pi(j)$, or $0 < j \leq \alpha_1$ and $\pi(j) > \pi(i)$.
    \end{itemize}
    We denote by $\Hyper_\alpha(312)$ the subset of elements in $\Hyper_\alpha$ that avoid type-B $(\alpha,312)$-split or -join patterns, depending on $\alpha$ being split or join.
\end{definition}

Similar to the case of type-B $(\alpha,231)$ patterns, we have the following lemma, whose proof is essentially identical to that of Lemma~\ref{lem:unique_maximal}. The only difference is that we are now looking at cover inversions with a smaller element in the middle, instead of a greater one in Lemma~\ref{lem:unique_maximal} (except in the case when $0 < j \leq \alpha_1$, where it is reversed).

\begin{lemma} \label{lem:unique_maximal_alt}
    For every $\pi \in \Hyper_\alpha$, there is a unique element $\pi_\dagger \in \Hyper_\alpha(312)$ with $\invset(\pi_\dagger) \subseteq \invset(\pi)$ such that $\invset(\pi_\dagger)$ is maximal by inclusion for any such element.
\end{lemma}

Similarly, we define the map $\Pi^\alpha_\dagger : \Hyper_\alpha \to \Hyper_\alpha(312)$ by taking $\Pi^\alpha_\dagger(\pi) = \pi_\dagger$ in Lemma~\ref{lem:unique_maximal_alt}.

\begin{definition} \label{def:upward_projection}
    We define a map $\iota : \Hyper_\alpha \to \Hyper_\alpha$ by taking $\iota(\pi) = \pi\woa{}$. In other words, to obtain $\iota(\pi)$, we replace all elements $k$ by $-k$ in $\pi$, and we reverse each $\alpha$-region of $\pi$ so that the elements are sorted. We check that $\iota(\pi)$ is in $\Hyper_\alpha$. We then define the upward projection map
    \begin{displaymath}
        \Pi^\alpha_\uparrow : \Hyper_\alpha \to \iota(\Hyper_\alpha(312)), \quad \pi \mapsto \iota(\Pi^\alpha_\dagger(\iota(\pi))).
    \end{displaymath}
\end{definition}

\begin{proposition} \label{prop:iota_dual}
  The map $\iota$ is order-reversing on $\Weak(\Hyper_\alpha)$.
\end{proposition}
\begin{proof}
    We only need to show that, for $\pi, \pi' \in \Hyper_\alpha$, we have $\invset(\pi) \subseteq \invset(\pi')$ if and only if $\invset(\iota(\pi')) \subseteq \invset(\iota(\pi))$. We observe that $\pi(i) = \iota(\pi)(-i)$ for all $i$, except when $\alpha$ is join and $\indicator(i) > 1$, and in this case we have $\pi(i) = \iota(\pi)(i)$. Since $\pi, \pi', \iota(\pi), \iota(\pi')$ are all in $\Hyper_\alpha$, there is no inversion in the same $\alpha$-region. Now, for $i < j$ in different $\alpha$-regions, we have $\woa(i) > \woa(j)$, thus the pair $i, j$ contributes to $\invset(\pi)$ if and only if the pair $\woa(i), \woa(j)$ does not contribute to $\invset(\iota(\pi))$. We then have the equivalence.
\end{proof}

\begin{proposition} \label{prop:projection_order_preserving}
    The maps $\Pi^\alpha_\downarrow$ and $\Pi^\alpha_\uparrow$ are order-preserving on $\Weak(\Hyper_\alpha)$.
\end{proposition}
\begin{proof}
    For $\pi, \pi' \in \Hyper_\alpha$ with $\pi \weakorder \pi'$, we have $\Pi^\alpha_\downarrow(\pi) \weakorder \pi \weakorder \pi'$. By the maximality of $\Pi^\alpha_\downarrow(\pi')$ ensured by Lemma~\ref{lem:unique_maximal}, we have $\Pi^\alpha_\downarrow(\pi) \weakorder \Pi^\alpha_\downarrow(\pi')$, thus $\Pi^\alpha_\downarrow$ is order-preserving. We also obtain the case for $\Pi^\alpha_\uparrow$ from that of $\Pi^\alpha_\downarrow$ and Proposition~\ref{prop:iota_dual}.
\end{proof}

\begin{proposition}\label{prop:cover_equivalence}
    For $\sigma$ and $\pi\in\Hyper_\alpha$ with $\sigma \weakcover \pi$, let $t$ be the only inversion in
$\invset(\pi) \setminus \invset(\sigma) \subseteq \covset(\pi)$. The following are equivalent:
    \begin{enumerate}
        \item[(i)] The inversion $t$ violates the conditions in Lemma~\ref{lem:split_forcing} (when $\alpha$ is split) or in Lemma~\ref{lem:join_forcing} (when $\alpha$ is join). 
        \item[(ii)] $\Pi^\alpha_\downarrow(\pi) = \Pi^\alpha_\downarrow(\sigma)$.
        \item[(iii)] $\Pi^\alpha_\uparrow(\pi) = \Pi^\alpha_\uparrow(\sigma)$.
  \end{enumerate}
\end{proposition}
\begin{proof}
    Assume that (i) holds. Then $\pi \notin \Hyper_\alpha(231)$, and we are in the situation of the inductive step of the proof of Lemma~\ref{lem:unique_maximal}. We thus have (i) implies (ii). Now for (iii), by Proposition~\ref{prop:iota_dual}, we have $\iota(\pi) \weakcover \iota(\sigma)$, thus $\woa^{-1} \cdot t \cdot \woa$ is in $\covset(\iota(\sigma))$. We observe that, a type-$B$ $(\alpha, 231)$ pattern $(i, j, k)$ of $\pi$ induced by $t$ leads to a type-$B$ $(\alpha, 312)$ pattern in $\iota(\sigma)$ corresponding to $\woa^{-1} \cdot t \cdot \woa$, which is $(-\woa(i), -\woa(j), -\woa(k))$, except when $\alpha$ is join and $0 \leq j \leq \alpha_1$, where the pattern will be $(\woa(k), \woa(j), \woa(i))$. Using the same reasoning on (i) $\Rightarrow$ (ii) but applied to Lemma~\ref{lem:unique_maximal_alt} with $\iota(\sigma)$ and $\iota(\pi)$, we conclude that (i) implies (iii).
  
    Now we show that (ii) implies (i) by induction on $|\invset(\sigma)|$. As
    \begin{displaymath}
        \Pi^\alpha_\downarrow(\pi) = \Pi^\alpha_\downarrow(\sigma) \weakorder \sigma \weakless \pi,
    \end{displaymath}
    by Corollary \ref{coro:projection_fixpoint}, we have $\pi \notin \Hyper_\alpha(231)$. When $|\invset(\sigma)| = 0$, we have $\Pi^\alpha_\downarrow(\pi) = \Pi^\alpha_\downarrow(\sigma) = \sigma$, which can only occur when $t$ satisfies (i) to be removed from $\pi$ in the procedure of Lemma~\ref{lem:unique_maximal}. Suppose that (ii) implies (i) for any $\sigma'$ with $k-1$ inversions, and we now consider $\sigma$ with $|\invset(\sigma)| = k$. If $\sigma \in \Hyper_\alpha(231)$, then we conclude by the same reasoning as in the case $|\invset(\sigma)| = 0$. Otherwise, suppose that there is some $t' \in \covset(\pi)$ giving rise to a type-$B$ $(\alpha, 231)$ pattern $(i, j, k)$. If $t' = t$, then by the construction in Lemma~\ref{lem:unique_maximal}, we have $\Pi^\alpha_\downarrow(\pi) = \Pi^\alpha_\downarrow(t \cdot \pi) = \Pi^\alpha_\downarrow(\sigma)$, and the implication holds. Now suppose that $t \neq t'$. We check that $t \cdot t' \cdot t$ gives rise to a type-$B$ $(\alpha, 231)$ pattern in $\sigma$. If $t$ commutes with $t'$, then $(i, j, k)$ remains a valid pattern. Otherwise, $t$ may permute $\pi(i)$ or $\pi(k)$. Suppose that $t$ permutes $\pi(i)$. As $t \neq t'$, $t$ permutes $\pi(i)$ with $\pi(i)^+$ in $\sigma$. As $t$ is an inversion in $\pi$, the index $i'$ such that $\sigma(i') = \pi(i)$ satisfies $i' < i$, meaning that $(i', j, k)$ is a valid pattern in $\sigma$. The case where $t$ permutes $\pi(k)$ is similar. Now, applying the induction hypothesis to $t \cdot t' \cdot t \cdot \sigma$ and $t' \cdot \pi = t' \cdot t \cdot \sigma$, we have $\Pi^\alpha_\downarrow(t' \cdot \pi) = \Pi^\alpha_\downarrow(t \cdot t' \cdot t \cdot \sigma)$ implying $t$ satisfies (i). By Lemma~\ref{lem:unique_maximal}, we have $\Pi^\alpha_\downarrow(\pi) = \Pi^\alpha_\downarrow(t' \cdot \pi)$ and $\Pi^\alpha_\downarrow(\sigma) = \Pi^\alpha_\downarrow(t \cdot t' \cdot t \cdot \sigma)$. We thus have (ii) implies (i) in this case, which concludes the induction. The proof for (iii) implies (i) is similar.
\end{proof}

\begin{proposition} \label{prop:two-ends-meet}
    For $\sigma, \pi \in \Hyper_\alpha$, then $\Pi^\alpha_\downarrow(\sigma) = \Pi^\alpha_\downarrow(\pi)$ if and only if $\Pi^\alpha_\uparrow(\sigma) = \Pi^\alpha_\uparrow(\pi)$.
\end{proposition}
\begin{proof}
    For the comparable case, we may suppose that $\sigma \weakless \pi$. The case $\sigma \weakcover \pi$ follows from Proposition~\ref{prop:cover_equivalence}, and other cases follows from sucessive application of the cover case.
  
    For the incomparable case, we observe that $\Pi^\alpha_\downarrow(\sigma \wedge \pi) = \Pi^\alpha_\downarrow(\sigma)$, as by Lemma~\ref{lem:unique_maximal}, $\Pi^\alpha_\downarrow(\sigma) = \Pi^\alpha_\downarrow(\pi)$ is the unique maximal element in $\Hyper_\alpha(231)$ below both $\pi$ and $\sigma$. With the same argument, but passing by $\iota$ and Lemma~\ref{lem:unique_maximal_alt}, we also have $\Pi^\alpha_\uparrow(\sigma \wedge \pi) = \Pi^\alpha_\uparrow(\sigma) = \Pi^\alpha_\uparrow(\pi)$. The equivalence thus also holds.
\end{proof}

\begin{proposition} \label{prop:congruence}
    The equivalence relation $\Theta$ defined in Definition~\ref{def:congruence} is a lattice congruence on $\Weak(\Hyper_\alpha)$, which is the interval $[\id,\woa]$ in $\Weak(\Hyper_n)$, and the related quotient lattice is $\Tamari_B(\alpha)$.
\end{proposition}
\begin{proof}
    To show that $\Theta$ is a congruence on $\Weak(\Hyper_\alpha)$, we check that the conditions in Lemma~\ref{lem:congruences_combin} are satisfied. For each equivalence class of $\Theta$, Lemma~\ref{lem:unique_maximal} implies the existence of a minimal element, which is in $\Hyper_\alpha(231)$. Then Proposition~\ref{prop:two-ends-meet} along with Lemma~\ref{lem:unique_maximal_alt} shows the existence of a maximal element in each equivalence class. Along with Proposition~\ref{prop:theta_order_convex}, this concludes that equivalence classes of $\Theta$ are intervals, which is precisely the condition (i) in Lemma~\ref{lem:congruences_combin}. The conditions (ii) and (iii) of Lemma~\ref{lem:congruences_combin} are ensured by Proposition~\ref{prop:projection_order_preserving}. Taking the minimal element in each equivalence class, the quotient lattice is $\Weak(\Hyper_\alpha(231)) = \Tamari_B(\alpha)$.
\end{proof}

\begin{figure}
    \centering
	\begin{tikzpicture}
		\def\x{1.25};
		\def\y{1};
		\def\s{.6};
		\def\t{1};	
		\coordinate(n1) at (3*\x,1*\y);
		\coordinate(n2) at (2*\x,2*\y);
		\coordinate(n3) at (3*\x,3*\y);
		\coordinate(n4) at (1*\x,3*\y);
		\coordinate(n5) at (2*\x,4*\y);
		\coordinate(n6) at (-1*\x,5*\y);
		\coordinate(n7) at (1*\x,5*\y);
		\coordinate(n9) at (2*\x,6*\y);
		\coordinate(n10) at (0*\x,6*\y);
		\coordinate(n11) at (1*\x,7*\y);
		\coordinate(n12) at (2*\x,8*\y);
		\draw[thick](n1) -- (n2);
		\draw[thick](n2) -- (n3);
		\draw[thick](n2) -- (n4);
		\draw[thick](n3) -- (n5);
		\draw[thick](n4) -- (n5);
		\draw[thick](n4) -- (n6);
		\draw[thick](n5) -- (n7);
		\draw[thick](n6) -- (n10);
		\draw[thick](n7) -- (n9);
		\draw[thick](n7) -- (n10);
		\draw[thick](n9) -- (n11);
		\draw[thick](n10) -- (n11);
		\draw[thick](n11) -- (n12);
		\draw(n1) node[scale=\s]{\AlphaPerm{2,2,2}{c2,c1,c2}{\overline{3},\overline{2},\overline{1},1,2,3}{.5}{.25}{white}{\t}};
		\draw(n2) node[scale=\s]{\AlphaPerm{2,2,2}{c2,c1,c2}{\overline{3},\overline{1},\overline{2},2,1,3}{.5}{.25}{white}{\t}};
		\draw(n3) node[scale=\s]{\AlphaPerm{2,2,2}{c2,c1,c2}{\overline{2},\overline{1},\overline{3},3,1,2}{.5}{.25}{white}{\t}};
		\draw(n4) node[scale=\s]{\AlphaPerm{2,2,2}{c2,c1,c2}{\overline{3},1,\overline{2},2,\overline{1},3}{.5}{.25}{white}{\t}};
		\draw(n5) node[scale=\s]{\AlphaPerm{2,2,2}{c2,c1,c2}{\overline{2},1,\overline{3},3,\overline{1},2}{.5}{.25}{white}{\t}};
		\draw(n6) node[scale=\s]{\AlphaPerm{2,2,2}{c2,c1,c2}{\overline{3},2,\overline{1},1,\overline{2},3}{.5}{.25}{white}{\t}};
		\draw(n7) node[scale=\s]{\AlphaPerm{2,2,2}{c2,c1,c2}{\overline{1},2,\overline{3},3,\overline{2},1}{.5}{.25}{white}{\t}};
		\draw(n9) node[scale=\s]{\AlphaPerm{2,2,2}{c2,c1,c2}{1,2,\overline{3},3,\overline{2},\overline{1}}{.5}{.25}{white}{\t}};
		\draw(n10) node[scale=\s]{\AlphaPerm{2,2,2}{c2,c1,c2}{1,3,\overline{2},2,\overline{3},\overline{1}}{.5}{.25}{white}{\t}};
		\draw(n11) node[scale=\s]{\AlphaPerm{2,2,2}{c2,c1,c2}{1,3,\overline{2},2,\overline{3},\overline{1}}{.5}{.25}{white}{\t}};
		\draw(n12) node[scale=\s]{\AlphaPerm{2,2,2}{c2,c1,c2}{2,3,\overline{1},1,\overline{3},\overline{2}}{.5}{.25}{white}{\t}};
	\end{tikzpicture}
    \caption{The lattice $\Tamari_{B}\bigl((1,2)\bigr)$.}
    \label{fig:tamari_12}
\end{figure}

\begin{figure}
    \centering
	\begin{tikzpicture}
		\def\x{1.25};
		\def\y{1};
		\def\s{.6};
		\def\t{1};	
		\draw(4*\x,1*\y) coordinate(n1);
		\draw(3*\x,2*\y) coordinate(n2);
		\draw(7*\x,4*\y) coordinate(n3);
		\draw(2*\x,3*\y) coordinate(n4);
		\draw(8*\x,5*\y) coordinate(n6);
		\draw(1*\x,4*\y) coordinate(n7);
		\draw(3*\x,4*\y) coordinate(n8);
		\draw(2*\x,5*\y) coordinate(n11);
		\draw(5*\x,6*\y) coordinate(n12);
		\draw(3*\x,6*\y) coordinate(n15);
		\draw(2*\x,7*\y) coordinate(n19);
		\draw(4*\x,7*\y) coordinate(n20);
		\draw(3*\x,8*\y) coordinate(n22);
		\draw(4*\x,9*\y) coordinate(n24);
		\draw[edge](n1) -- (n2);
		\draw[edge](n1) -- (n3);
		\draw[edge](n2) -- (n4);
		\draw[edge](n3) -- (n12);
		\draw[edge](n3) -- (n6);
		\draw[edge](n4) -- (n7);
		\draw[edge](n4) -- (n8);
		\draw[edge](n6) -- (n24);
		\draw[edge](n7) -- (n11);
		\draw[edge](n8) -- (n11);
		\draw[edge](n8) -- (n12);
		\draw[edge](n11) -- (n15);
		\draw[edge](n12) -- (n20);
		\draw[edge](n15) -- (n19);
		\draw[edge](n15) -- (n20);
		\draw[edge](n19) -- (n22);
		\draw[edge](n20) -- (n22);
		\draw[edge](n22) -- (n24);
		\draw(n1) node[scale=\s]{\AlphaPerm{2,1,1,2}{c2,c1,c1,c2}{\overline{3},\overline{2},\overline{1},1,2,3}{.5}{.25}{white}{\t}};
		\draw(n2) node[scale=\s]{\AlphaPerm{2,1,1,2}{c2,c1,c1,c2}{\overline{3},\overline{2},1,\overline{1},2,3}{.5}{.25}{white}{\t}};
		\draw(n3) node[scale=\s]{\AlphaPerm{2,1,1,2}{c2,c1,c1,c2}{\overline{3},\overline{1},\overline{2},2,1,3}{.5}{.25}{white}{\t}};
		\draw(n4) node[scale=\s]{\AlphaPerm{2,1,1,2}{c2,c1,c1,c2}{\overline{3},\overline{1},2,\overline{2},1,3}{.5}{.25}{white}{\t}};
		\draw(n6) node[scale=\s]{\AlphaPerm{2,1,1,2}{c2,c1,c1,c2}{\overline{2},\overline{1},\overline{3},3,1,2}{.5}{.25}{white}{\t}};
		\draw(n7) node[scale=\s]{\AlphaPerm{2,1,1,2}{c2,c1,c1,c2}{\overline{2},\overline{1},3,\overline{3},1,2}{.5}{.25}{white}{\t}};
		\draw(n8) node[scale=\s]{\AlphaPerm{2,1,1,2}{c2,c1,c1,c2}{\overline{3},1,2,\overline{2},\overline{1},3}{.5}{.25}{white}{\t}};
		\draw(n11) node[scale=\s]{\AlphaPerm{2,1,1,2}{c2,c1,c1,c2}{\overline{2},1,3,\overline{3},\overline{1},2}{.5}{.25}{white}{\t}};
		\draw(n12) node[scale=\s]{\AlphaPerm{2,1,1,2}{c2,c1,c1,c2}{\overline{3},2,1,\overline{1},\overline{2},3}{.5}{.25}{white}{\t}};
		\draw(n15) node[scale=\s]{\AlphaPerm{2,1,1,2}{c2,c1,c1,c2}{\overline{1},2,3,\overline{3},\overline{2},1}{.5}{.25}{white}{\t}};
		\draw(n19) node[scale=\s]{\AlphaPerm{2,1,1,2}{c2,c1,c1,c2}{1,2,3,\overline{3},\overline{2},\overline{1}}{.5}{.25}{white}{\t}};
		\draw(n20) node[scale=\s]{\AlphaPerm{2,1,1,2}{c2,c1,c1,c2}{\overline{1},3,2,\overline{2},\overline{3},1}{.5}{.25}{white}{\t}};
		\draw(n22) node[scale=\s]{\AlphaPerm{2,1,1,2}{c2,c1,c1,c2}{1,3,2,\overline{2},\overline{3},\overline{1}}{.5}{.25}{white}{\t}};
		\draw(n24) node[scale=\s]{\AlphaPerm{2,1,1,2}{c2,c1,c1,c2}{2,3,1,\overline{1},\overline{3},\overline{2}}{.5}{.25}{white}{\t}};
	\end{tikzpicture}
    \caption{The lattice $\Tamari_{B}\bigl((0,1,2)\bigr)$.}
    \label{fig:tamari_012}
\end{figure}

Figure~\ref{fig:tamari_12} shows $\Tamari_{B}\bigl((1,2)\bigr)$ and Figure~\ref{fig:tamari_012} shows $\Tamari_{B}\bigl((0,1,2)\bigr)$.  We have therefore just proved Theorem~\ref{thm:parabolic_tamari_lattice}.

\begin{proof}[Proof of Theorem~\ref{thm:parabolic_tamari_lattice}]
    Proposition~\ref{prop:parabolic_tamari_lattice} states that $\Tamari_{B}(\alpha)$ is a lattice and Proposition~\ref{prop:congruence} proves that it is a quotient lattice of the weak order on $\Hyper_{\alpha}$.
\end{proof}

\begin{remark}
	In contrast to Theorem~\ref{thm:tamari_type_b}, the lattice $\Tamari_{B}(\alpha)$ is in general \textit{not} a sublattice of $\Weak(\Hyper_{\alpha})$.  This can be witnessed for the composition $\alpha=(0,2,1)$.  We have drawn the weak order on $\Hyper_{(0,2,1)}$ in Figure~\ref{fig:weak_order_021}, where we have highlighted the congruence classes with respect to $\Pi^\alpha_\downarrow$, and we have drawn $\Tamari_{B}\bigl((0,2,1)\bigr)$ in Figure~\ref{fig:tamari_021}. 
	
	If we consider $\pi_{1}=\hspace*{-.3cm}\raisebox{\rsp}{\AlphaPerm{1,2,2,1}{c2,c1,c1,c2}{3,\overline{1},2,\overline{2},1,\overline{3}}{.5}{.25}{white}{1}}\hspace*{-.1cm}$ and $\pi_{2}=\hspace*{-.3cm}\raisebox{\rsp}{\AlphaPerm{1,2,2,1}{c2,c1,c1,c2}{2,1,3,\overline{3},\overline{1},\overline{2}}{.5}{.25}{white}{1}}\hspace*{-.1cm}$, then their meet in $\Weak\bigl(\Hyper\bigl((0,2,1)\bigr)\bigr)$ is $\sigma_{1}=\hspace*{-.3cm}\raisebox{\rsp}{\AlphaPerm{1,2,2,1}{c2,c1,c1,c2}{2,\overline{1},3,\overline{3},1,\overline{2}}{.5}{.25}{white}{1}}\hspace*{-.1cm}$, while their meet in $\Tamari_{B}\bigl((0,2,1)\bigr)$ is $\sigma_{2}=\hspace*{-.3cm}\raisebox{\rsp}{\AlphaPerm{1,2,2,1}{c2,c1,c1,c2}{\overline{1},\overline{2},3,\overline{3},2,1}{.5}{.25}{white}{1}}\hspace*{-.1cm}$.  It clearly holds that $\sigma_{1}\neq\sigma_{2}$ but we have $\Pi^{(0,2,1)}_\downarrow(\sigma_{1})=\sigma_{2}$.
\end{remark}

\begin{figure}
	\centering
	\begin{tikzpicture}
		\def\x{1.25};
		\def\y{1};
		\def\s{.6};
		\def\t{1};	
		\coordinate(n1) at (6*\x,1*\y);
		\coordinate(n2) at (5*\x,2*\y);
		\coordinate(n3) at (7*\x,2*\y);
		\coordinate(n4) at (4*\x,3*\y);
		\coordinate(n5) at (6*\x,3*\y);
		\coordinate(n6) at (8*\x,3*\y);
		\coordinate(n7) at (3*\x,4*\y);
		\coordinate(n8) at (5*\x,4*\y);
		\coordinate(n9) at (7*\x,4*\y);
		\coordinate(n10) at (9*\x,4*\y);
		\coordinate(n11) at (2*\x,5*\y);
		\coordinate(n12) at (4*\x,5*\y);
		\coordinate(n13) at (6*\x,5*\y);
		\coordinate(n14) at (8*\x,5*\y);
		\coordinate(n15) at (1*\x,6*\y);
		\coordinate(n16) at (3*\x,6*\y);
		\coordinate(n17) at (5*\x,6*\y);
		\coordinate(n18) at (7*\x,6*\y);
		\coordinate(n19) at (2*\x,7*\y);
		\coordinate(n20) at (4*\x,7*\y);
		\coordinate(n21) at (6*\x,7*\y);
		\coordinate(n22) at (3*\x,8*\y);
		\coordinate(n23) at (5*\x,8*\y);
		\coordinate(n24) at (4*\x,9*\y);
		\draw[thick](n1) -- (n2);
		\draw[thick](n1) -- (n3);
		\draw[thick](n2) -- (n4);
		\draw[thick](n2) -- (n5);
		\draw[thick](n3) -- (n5);
		\draw[thick](n3) -- (n6);
		\draw[thick](n4) -- (n7);
		\draw[thick](n5) -- (n8);
		\draw[thick](n6) -- (n9);
		\draw[thick](n6) -- (n10);
		\draw[thick](n7) -- (n11);
		\draw[thick](n8) -- (n12);
		\draw[thick](n8) -- (n13);
		\draw[thick](n9) -- (n13);
		\draw[thick](n9) -- (n14);
		\draw[thick](n10) -- (n14);
		\draw[thick](n11) -- (n16);
		\draw[thick](n11) -- (n15);
		\draw[thick](n12) -- (n16);
		\draw[thick](n12) -- (n17);
		\draw[thick](n13) -- (n17);
		\draw[thick](n14) -- (n18);
		\draw[thick](n15) -- (n19);
		\draw[thick](n16) -- (n19);
		\draw[thick](n17) -- (n20);
		\draw[thick](n18) -- (n21);
		\draw[thick](n19) -- (n22);
		\draw[thick](n20) -- (n22);
		\draw[thick](n20) -- (n23);
		\draw[thick](n21) -- (n23);
		\draw[thick](n22) -- (n24);
		\draw[thick](n23) -- (n24);
		\draw(n1) node[scale=\s]{\AlphaPerm{1,2,2,1}{c2,c1,c1,c2}{\overline{3},\overline{2},\overline{1},1,2,3}{.5}{.25}{white!50!gray}{\t}};
		\draw(n2) node[scale=\s]{\AlphaPerm{1,2,2,1}{c2,c1,c1,c2}{\overline{2},\overline{3},\overline{1},1,3,2}{.5}{.25}{white!50!gray}{\t}};
		\draw(n3) node[scale=\s]{\AlphaPerm{1,2,2,1}{c2,c1,c1,c2}{\overline{3},\overline{2},1,\overline{1},2,3}{.5}{.25}{white!50!gray}{\t}};
		\draw(n4) node[scale=\s]{\AlphaPerm{1,2,2,1}{c2,c1,c1,c2}{\overline{1},\overline{3},\overline{2},2,3,1}{.5}{.25}{white!50!gray}{\t}};
		\draw(n5) node[scale=\s]{\AlphaPerm{1,2,2,1}{c2,c1,c1,c2}{\overline{2},\overline{3},1,\overline{1},3,2}{.5}{.25}{white!50!gray}{\t}};
		\draw(n6) node[scale=\s]{\AlphaPerm{1,2,2,1}{c2,c1,c1,c2}{\overline{3},\overline{1},2,\overline{2},1,3}{.5}{.25}{white!50!gray}{\t}};
		\draw(n7) node[scale=\s]{\AlphaPerm{1,2,2,1}{c2,c1,c1,c2}{1,\overline{3},\overline{2},2,3,\overline{1}}{.5}{.25}{white!50!gray}{\t}};
		\draw(n8) node[scale=\s]{\AlphaPerm{1,2,2,1}{c2,c1,c1,c2}{\overline{1},\overline{3},2,\overline{2},3,1}{.5}{.25}{white!50!gray}{\t}};
		\draw(n9) node[scale=\s]{\AlphaPerm{1,2,2,1}{c2,c1,c1,c2}{\overline{2},\overline{1},3,\overline{3},1,2}{.5}{.25}{white!50!gray}{\t}};
		\draw(n10) node[scale=\s]{\AlphaPerm{1,2,2,1}{c2,c1,c1,c2}{\overline{3},1,2,\overline{2},\overline{1},3}{.5}{.25}{white!50!gray}{\t}};
		\draw(n11) node[scale=\s]{\AlphaPerm{1,2,2,1}{c2,c1,c1,c2}{2,\overline{3},\overline{1},1,3,\overline{2}}{.5}{.25}{white!50!gray}{\t}};
		\draw(n12) node[scale=\s]{\AlphaPerm{1,2,2,1}{c2,c1,c1,c2}{1,\overline{3},2,\overline{2},3,\overline{1}}{.5}{.25}{white!50!gray}{\t}};
		\draw(n13) node[scale=\s]{\AlphaPerm{1,2,2,1}{c2,c1,c1,c2}{\overline{1},\overline{2},3,\overline{3},2,1}{.5}{.25}{white!50!gray}{\t}};
		\draw(n14) node[scale=\s]{\AlphaPerm{1,2,2,1}{c2,c1,c1,c2}{\overline{2},1,3,\overline{3},\overline{1},2}{.5}{.25}{white!50!gray}{\t}};
		\draw(n15) node[scale=\s]{\AlphaPerm{1,2,2,1}{c2,c1,c1,c2}{3,\overline{2},\overline{1},1,2,\overline{3}}{.5}{.25}{white!50!gray}{\t}};
		\draw(n16) node[scale=\s]{\AlphaPerm{1,2,2,1}{c2,c1,c1,c2}{2,\overline{3},1,\overline{1},3,\overline{2}}{.5}{.25}{white!50!gray}{\t}};
		\draw(n17) node[scale=\s]{\AlphaPerm{1,2,2,1}{c2,c1,c1,c2}{1,\overline{2},3,\overline{3},2,\overline{1}}{.5}{.25}{white!50!gray}{\t}};
		\draw(n18) node[scale=\s]{\AlphaPerm{1,2,2,1}{c2,c1,c1,c2}{\overline{1},2,3,\overline{3},\overline{2},1}{.5}{.25}{white!50!gray}{\t}};
		\draw(n19) node[scale=\s]{\AlphaPerm{1,2,2,1}{c2,c1,c1,c2}{3,\overline{2},1,\overline{1},2,\overline{3}}{.5}{.25}{white!50!gray}{\t}};
		\draw(n20) node[scale=\s]{\AlphaPerm{1,2,2,1}{c2,c1,c1,c2}{2,\overline{1},3,\overline{3},1,\overline{2}}{.5}{.25}{white!50!gray}{\t}};
		\draw(n21) node[scale=\s]{\AlphaPerm{1,2,2,1}{c2,c1,c1,c2}{1,2,3,\overline{3},\overline{2},\overline{1}}{.5}{.25}{white!50!gray}{\t}};
		\draw(n22) node[scale=\s]{\AlphaPerm{1,2,2,1}{c2,c1,c1,c2}{3,\overline{1},2,\overline{2},1,\overline{3}}{.5}{.25}{white!50!gray}{\t}};
		\draw(n23) node[scale=\s]{\AlphaPerm{1,2,2,1}{c2,c1,c1,c2}{2,1,3,\overline{3},\overline{1},\overline{2}}{.5}{.25}{white!50!gray}{\t}};
		\draw(n24) node[scale=\s]{\AlphaPerm{1,2,2,1}{c2,c1,c1,c2}{3,1,2,\overline{2},\overline{1},\overline{3}}{.5}{.25}{white!50!gray}{\t}};
		\begin{pgfonlayer}{background}
			\filldraw[draw=gray,fill=white!50!gray,rounded corners](5.25*\x,.85*\y) -- (6.8*\x,.85*\y) -- (6.8*\x,1.25*\y) -- (5.25*\x,1.25*\y) -- cycle;
			\filldraw[draw=gray,fill=white!50!gray,rounded corners](4.25*\x,1.85*\y) -- (5.8*\x,1.85*\y) -- (5.8*\x,2.25*\y) -- (4.25*\x,2.25*\y) -- cycle;
			\filldraw[draw=gray,fill=white!50!gray,rounded corners](6.25*\x,1.85*\y) -- (7.8*\x,1.85*\y) -- (7.8*\x,2.25*\y) -- (6.25*\x,2.25*\y) -- cycle;
			\filldraw[draw=gray,fill=white!50!gray,rounded corners](3.25*\x,2.85*\y) -- (4.8*\x,2.85*\y) -- (4.8*\x,3.25*\y) -- (3.8*\x,3.85*\y) -- (3.8*\x,4.25*\y) -- (2.8*\x,4.85*\y) -- (2.8*\x,5.25*\y) -- (1.8*\x,5.85*\y) -- (1.8*\x,6.25*\y) -- (.25*\x,6.25*\y) -- (.25*\x,5.85*\y) -- (1.25*\x,5.25*\y) -- (1.25*\x,4.85*\y) -- (2.25*\x,4.25*\y) -- (2.25*\x,3.85*\y) -- (3.25*\x,3.25*\y) -- cycle;
			\filldraw[draw=gray,fill=white!50!gray,rounded corners](5.25*\x,2.85*\y) -- (6.8*\x,2.85*\y) -- (6.8*\x,3.25*\y) -- (5.8*\x,3.85*\y) -- (5.8*\x,4.25*\y) -- (4.8*\x,4.85*\y) -- (4.8*\x,5.25*\y) -- (3.25*\x,5.25*\y) -- (3.25*\x,4.85*\y) -- (4.25*\x,4.25*\y) -- (4.25*\x,3.85*\y) -- (5.25*\x,3.25*\y) -- cycle;
			\filldraw[draw=gray,fill=white!50!gray,rounded corners](7.25*\x,2.85*\y) -- (8.8*\x,2.85*\y) -- (8.8*\x,3.25*\y) -- (7.25*\x,3.25*\y) -- cycle;
			\filldraw[draw=gray,fill=white!50!gray,rounded corners](6.25*\x,3.85*\y) -- (7.8*\x,3.85*\y) -- (7.8*\x,4.25*\y) -- (6.25*\x,4.25*\y) -- cycle;
			\filldraw[draw=gray,fill=white!50!gray,rounded corners](8.25*\x,3.85*\y) -- (9.8*\x,3.85*\y) -- (9.8*\x,4.25*\y) -- (8.25*\x,4.25*\y) -- cycle;
			\filldraw[draw=gray,fill=white!50!gray,rounded corners](5.25*\x,4.85*\y) -- (6.8*\x,4.85*\y) -- (6.8*\x,5.25*\y) -- (5.8*\x,5.85*\y) -- (5.8*\x,6.25*\y) -- (4.8*\x,6.85*\y) -- (4.8*\x,7.25*\y) -- (3.25*\x,7.25*\y) -- (3.25*\x,6.85*\y) -- (4.25*\x,6.25*\y) -- (4.25*\x,5.85*\y) -- (5.25*\x,5.25*\y) -- cycle;
			\filldraw[draw=gray,fill=white!50!gray,rounded corners](7.25*\x,4.85*\y) -- (8.8*\x,4.85*\y) -- (8.8*\x,5.25*\y) -- (7.25*\x,5.25*\y) -- cycle;
			\filldraw[draw=gray,fill=white!50!gray,rounded corners](2.25*\x,5.85*\y) -- (3.8*\x,5.85*\y) -- (3.8*\x,6.25*\y) -- (2.8*\x,6.85*\y) -- (2.8*\x,7.25*\y) -- (1.25*\x,7.25*\y) -- (1.25*\x,6.8*\y) -- (2.25*\x,6.25*\y) -- cycle;
			\filldraw[draw=gray,fill=white!50!gray,rounded corners](6.25*\x,5.85*\y) -- (7.8*\x,5.85*\y) -- (7.8*\x,6.25*\y) -- (6.25*\x,6.25*\y) -- cycle;
			\filldraw[draw=gray,fill=white!50!gray,rounded corners](5.25*\x,6.85*\y) -- (6.8*\x,6.85*\y) -- (6.8*\x,7.25*\y) -- (5.25*\x,7.25*\y) -- cycle;
			\filldraw[draw=gray,fill=white!50!gray,rounded corners](2.25*\x,7.85*\y) -- (3.8*\x,7.85*\y) -- (3.8*\x,8.25*\y) -- (2.25*\x,8.25*\y) -- cycle;
			\filldraw[draw=gray,fill=white!50!gray,rounded corners](4.25*\x,7.85*\y) -- (5.8*\x,7.85*\y) -- (5.8*\x,8.25*\y) -- (4.25*\x,8.25*\y) -- cycle;
			\filldraw[draw=gray,fill=white!50!gray,rounded corners](3.25*\x,8.85*\y) -- (4.8*\x,8.85*\y) -- (4.8*\x,9.25*\y) -- (3.25*\x,9.25*\y) -- cycle;
		\end{pgfonlayer}
	\end{tikzpicture}
	\caption{The weak order on $\Hyper_{(0,2,1)}$.}
	\label{fig:weak_order_021}
\end{figure}

\begin{figure}
	\centering
	\begin{tikzpicture}
		\def\x{1.25};
		\def\y{1};
		\def\s{.6};
		\def\t{1};	
		\coordinate(n1) at (6*\x,1*\y);
		\coordinate(n2) at (2*\x,5*\y);
		\coordinate(n3) at (7*\x,2*\y);
		\coordinate(n4) at (1*\x,6*\y);
		\coordinate(n5) at (3*\x,6*\y);
		\coordinate(n6) at (8*\x,3*\y);
		\coordinate(n9) at (7*\x,4*\y);
		\coordinate(n10) at (9*\x,4*\y);
		\coordinate(n13) at (4*\x,7*\y);
		\coordinate(n14) at (8*\x,5*\y);
		\coordinate(n18) at (7*\x,6*\y);
		\coordinate(n21) at (6*\x,7*\y);
		\coordinate(n22) at (3*\x,8*\y);
		\coordinate(n23) at (5*\x,8*\y);
		\coordinate(n24) at (4*\x,9*\y);
		\draw[thick](n1) -- (n2);
		\draw[thick](n1) -- (n3);
		\draw[thick](n2) -- (n4);
		\draw[thick](n2) -- (n5);
		\draw[thick](n3) -- (n5);
		\draw[thick](n3) -- (n6);
		\draw[thick](n4) -- (n22);
		\draw[thick](n5) -- (n13);
		\draw[thick](n6) -- (n9);
		\draw[thick](n6) -- (n10);
		\draw[thick](n9) -- (n13);
		\draw[thick](n9) -- (n14);
		\draw[thick](n10) -- (n14);
		\draw[thick](n13) -- (n22);
		\draw[thick](n13) -- (n23);
		\draw[thick](n14) -- (n18);
		\draw[thick](n18) -- (n21);
		\draw[thick](n21) -- (n23);
		\draw[thick](n22) -- (n24);
		\draw[thick](n23) -- (n24);
		\draw(n1) node[scale=\s]{\AlphaPerm{1,2,2,1}{c2,c1,c1,c2}{\overline{3},\overline{2},\overline{1},1,2,3}{.5}{.25}{white}{\t}};
		\draw(n2) node[scale=\s]{\AlphaPerm{1,2,2,1}{c2,c1,c1,c2}{\overline{2},\overline{3},\overline{1},1,3,2}{.5}{.25}{white}{\t}};
		\draw(n3) node[scale=\s]{\AlphaPerm{1,2,2,1}{c2,c1,c1,c2}{\overline{3},\overline{2},1,\overline{1},2,3}{.5}{.25}{white}{\t}};
		\draw(n4) node[scale=\s]{\AlphaPerm{1,2,2,1}{c2,c1,c1,c2}{\overline{1},\overline{3},\overline{2},2,3,1}{.5}{.25}{white}{\t}};
		\draw(n5) node[scale=\s]{\AlphaPerm{1,2,2,1}{c2,c1,c1,c2}{\overline{2},\overline{3},1,\overline{1},3,2}{.5}{.25}{white}{\t}};
		\draw(n6) node[scale=\s]{\AlphaPerm{1,2,2,1}{c2,c1,c1,c2}{\overline{3},\overline{1},2,\overline{2},1,3}{.5}{.25}{white}{\t}};
		\draw(n9) node[scale=\s]{\AlphaPerm{1,2,2,1}{c2,c1,c1,c2}{\overline{2},\overline{1},3,\overline{3},1,2}{.5}{.25}{white}{\t}};
		\draw(n10) node[scale=\s]{\AlphaPerm{1,2,2,1}{c2,c1,c1,c2}{\overline{3},1,2,\overline{2},\overline{1},3}{.5}{.25}{white}{\t}};
		\draw(n13) node[scale=\s]{\AlphaPerm{1,2,2,1}{c2,c1,c1,c2}{\overline{1},\overline{2},3,\overline{3},2,1}{.5}{.25}{white}{\t}};
		\draw(n14) node[scale=\s]{\AlphaPerm{1,2,2,1}{c2,c1,c1,c2}{\overline{2},1,3,\overline{3},\overline{1},2}{.5}{.25}{white}{\t}};
		\draw(n18) node[scale=\s]{\AlphaPerm{1,2,2,1}{c2,c1,c1,c2}{\overline{1},2,3,\overline{3},\overline{2},1}{.5}{.25}{white}{\t}};
		\draw(n21) node[scale=\s]{\AlphaPerm{1,2,2,1}{c2,c1,c1,c2}{1,2,3,\overline{3},\overline{2},\overline{1}}{.5}{.25}{white}{\t}};
		\draw(n22) node[scale=\s]{\AlphaPerm{1,2,2,1}{c2,c1,c1,c2}{3,\overline{1},2,\overline{2},1,\overline{3}}{.5}{.25}{white}{\t}};
		\draw(n23) node[scale=\s]{\AlphaPerm{1,2,2,1}{c2,c1,c1,c2}{2,1,3,\overline{3},\overline{1},\overline{2}}{.5}{.25}{white}{\t}};
		\draw(n24) node[scale=\s]{\AlphaPerm{1,2,2,1}{c2,c1,c1,c2}{3,1,2,\overline{2},\overline{1},\overline{3}}{.5}{.25}{white}{\t}};
	\end{tikzpicture}
	\caption{The lattice $\Tamari_{B}\bigl((0,2,1)\bigr)$.}
	\label{fig:tamari_021}
\end{figure}

An immediate consequence is the following property of $\Tamari_{B}(\alpha)$.

\begin{proposition}\label{prop:alpha_tamari_congruence_uniform}
    For every type-$B$ composition $\alpha$ of $n$, the lattice $\Tamari_{B}(\alpha)$ is congruence uniform.
\end{proposition}
\begin{proof}
    By Theorem~\ref{thm:weak_order_congruence_uniform}, $\Weak(\Hyper_{n})$ is congruence uniform, and by Theorem~\ref{thm:weak_order_quotients}, $\Weak(\Hyper_{\alpha})$ is an interval of $\Weak(\Hyper_{n})$.  Since intervals are sublattices, Proposition~\ref{prop:congruence_uniform_preserving} implies that $\Weak(\Hyper_{\alpha})$ is congruence uniform, too. Proposition~\ref{prop:congruence} now states that $\Tamari_{B}(\alpha)$ is a quotient lattice of $\Weak(\Hyper_{\alpha})$, so that the claim follows by applying Proposition~\ref{prop:congruence_uniform_preserving} once more.
\end{proof}

\subsection{Extremality of $\Tamari_{B}(\alpha)$}

In this section, we prove that $\Tamari_{B}(\alpha)$ is an extremal
lattice, \ie it has the same number of join- and meet-irreducible elements
and this number agrees with its length.

\begin{proposition}\label{prop:alpha_tamari_length}
    For every type-$B$ composition $\alpha$ of $n$, the length of the lattice $\Tamari_{B}(\alpha)$ is
    \begin{displaymath}
        n^2 - \sum_{i=1}^r \binom{\alpha_i}{2} -
        \begin{cases}
            \binom{\alpha_{1}+1}{2}, & \text{if}\;\alpha\;\text{is join},\\
            0, & \text{if}\;\alpha\;\text{is split}.
        \end{cases}
    \end{displaymath}
More precisely, the sequence $C$ of suffixes of the $\linc$-sorting word of $\woa$ forms a maximal chain of $\Tamari_B(\alpha)$.
\end{proposition}
\begin{proof}
    We shall show the more precise statement since it implies the first one. First, since $\Tamari_B(\alpha)$ is a subposet of all type B permutations, its length cannot exceed the number of inversions of its maximal element $\woa$. So the considered chain $C$ is of maximal available length and it is a chain since the whole $\linc$-sorting word is a reduced decomposition so that each new letter exactly adds one inversion.

    Now, all the suffixes of $C$ are aligned. Indeed, let $\pi$ be such an element. Its inversions form a suffix of the list of inversions of $\woa$, or, in terms of the skew partition containing the inversions, they satisfy that if a cell $c$ is a inversion of $\pi$, then any cell above of $c$ or in its row to its right is also an inversion of $\pi$. But, thanks to Lemmas~\ref{lem:split_forcing} and~\ref{lem:join_forcing}, a permutation is aligned if and only if for any $t\in\covset(\pi)$, it implies other inversions in $\pi$ and thanks to Notes~\ref{note:split_forcing} and~\ref{note:join_forcing}, $\pi$ satisfies all forcing relations.
\end{proof}

\begin{proposition}\label{prop:alpha_tamari_irreducibles_count}
    For every type-$B$ composition $\alpha$ of $n$, the number of join-irreducible elements of $\Tamari_{B}(\alpha)$ is
    \begin{displaymath}
        n^2 - \sum_{i=1}^r \binom{\alpha_i}{2} -
        \begin{cases}
            \binom{\alpha_{1}+1}{2}, & \text{if}\;\alpha\;\text{is join},\\
            0, & \text{if}\;\alpha\;\text{is split}.
        \end{cases}
    \end{displaymath}
\end{proposition}

The proof of this proposition directly comes from the next lemma. Indeed, recall that join-irreducible elements have by definition exactly one cover inversion. Since a cover inversion is in particular an inversion, the set of possible cover inversions of the join-irreducibles of $\Tamari_{B}(\alpha)$ is the set of inversions of $\woa$. The next lemma shows that, given any such inversion $\invd{i}{j}$, there is exactly one join-irreducible element of $\Tamari_{B}(\alpha)$ such that $\invd{i}{j}$ is its cover inversion, hence showing that the number of join-irreducible is equal to the length of $\woa$.

\begin{lemma}
    Let $\alpha$ be a type-$B$ composition and let $\invd{i}{j}$ be an inversion of $\woa$. Then there is exactly one join-irreducible element $\pi(i,j)$ of $\Tamari_{B}(\alpha)$ such that $\invd{i}{j}$ is its unique cover inversion.

    Moreover, it can be computed directly as follows.
    \begin{itemize}
        \item If $i$ and $j$ are both positive, fill from left to right the right part of a sign-symmetric permutation with letters starting from $1$ in increasing order, forgetting about position $i$ and the right of its block and ending at position $j$. Then fill the remaining positions from left to right with the remaining positive values in increasing order.
        \item If $i=-j$, fill $\pi$ with the values from $1$ to $n$ in that order the following positions: all the left part of $\pi$ starting at $i$ (or, if $\alpha$ is join, all the left part except the central block and then the right half of the central block instead) and then the right part of $\pi$ starting at position $j+1$.
        \item If either $i$ or $j$ is negative and $|i|<|j|$, there are three different cases.
        \begin{itemize}
            \item If $\alpha$ is split, $i$ is the negative one. Put $k=i+j$. Then its right part is given by the sequence
            \begin{equation}
                \overline j\dots\overline{k\!+\!1}\ \ 1\dots k\ \ j\!+\!1\dots n.
            \end{equation}
            \item If $\alpha$ is join and $|i|>\alpha_1$, then $i$ is again the negative one. Put $k=i+j$ and $\ell=j-\alpha_1$. Then its right part is given by the sequence
            \begin{equation}
                \ell\!+\!1\dots\ell\!+\!\alpha_1\ \ \overline\ell\dots\overline{k\!+\!1}\ \ 1\dots k\ \ \ell\!+\!\alpha_1\!+\!1\dots n.
            \end{equation}
            \item If $\alpha$ is join and $|i|\leq \alpha_1$, then $i$ is the positive one. Put $k=i$ and $\ell=i-j-\alpha_1$. Then its right part is given by the sequence
            \begin{equation}
                1\dots k\ \ \ell\!+\!1\dots \ell\!+\!\alpha_1\!-\!i\ \ \overline\ell\dots\overline{k\!+\!1}\ \ \ell\!+\!\alpha_1\!-\!i\!+\!1\dots n.
            \end{equation}
        \end{itemize}
    \end{itemize}
\end{lemma}

\begin{example}\label{ex:irreducibles}
    Here are a few examples showing in all cases above how a join-irreducible is associated with a type-$B$ composition and an inversion. We shall only represent the right part of the permutation.
    \begin{displaymath}\begin{aligned}
        & \alpha=(0,3,1,2,1), && \invd{i}{j}=(2,6): && \pi  = \hspace*{-.25cm}\raisebox{\rsp}{\AlphaPerm{3,1,2,1}{c1,c2,c3,c4}{1,\hgl{5},6,2,3,\hgl{4},7}{.5}{.25}{white}{1}}\hspace*{-.1cm}\\
        & \alpha=(4,2,2), && \invd{i}{j}=(2,6): && \pi = \hspace*{-.25cm}\raisebox{\rsp}{\AlphaPerm{4,2,2}{c1,c2,c3}{1,\hgl{4},5,6,2,\hgl{3},7,8}{.5}{.25}{white}{1}}\hspace*{-.1cm}\\
        & \alpha=(0,3,1,2,1), && \invd{i}{j}=(\overline5,5): && \pi = \hspace*{-.25cm}\raisebox{\rsp}{\AlphaPerm{3,1,2,1}{c1,c2,c3,c4}{\overline5,\overline4,\overline3,\overline2,\hgl{\overline1},6,7}{.5}{.25}{white}{1}} \hspace*{-.1cm}\\
        & \alpha=(4,2,2), && \invd{i}{j}=(\overline6,6): && \pi  = \hspace*{-.25cm}\raisebox{\rsp}{\AlphaPerm{4,2,2}{c1,c2,c3}{3,4,5,6,\overline2,\hgl{\overline1},7,8}{.5}{.25}{white}{1}} \hspace*{-.1cm}\\
        & \alpha=(0,3,1,2,1), && \invd{i}{j}=(\overline2,6): && \pi  = \hspace*{-.25cm}\raisebox{\rsp}{\AlphaPerm{3,1,2,1}{c1,c2,c3,c4}{\overline6,\overline5,1,2,3,\hgl{4},7}{.5}{.25}{white}{1}} \hspace*{-.1cm}\\
        & \alpha=(4,2,2), && \invd{i}{j}=(\overline5,7): 
        && \pi  = \hspace*{-.25cm}\raisebox{\rsp}{\AlphaPerm{4,2,2}{c1,c2,c3}{4,5,6,7,\overline{3},1,\hgl{2},8}{.5}{.25}{white}{1}} \hspace*{-.1cm}\\
        & \alpha=(4,2,2), && \invd{i}{j}=(2,\overline5): 
        && \pi  = \hspace*{-.25cm}\raisebox{\rsp}{\AlphaPerm{4,2,2}{c1,c2,c3}{1,2,4,5,\hgl{\overline3},6,7,8}{.5}{.25}{white}{1}} \hspace*{-.1cm}\\
    \end{aligned}\end{displaymath}
\end{example}

\begin{proof}
    Let $\invd{i}{j}$ be an inversion of $\woa$ and let us consider $\pi$ a join-irreducible element with $\invd{i}{j}$ as its unique cover inversion. Let $k$ and $k+1$ be the values of this cover inversion, that is, the values in positions $i$ and $j$. Since there is only one cover inversion, that means that the values $\overline1$ up to $k$ are in increasing order from left to right in $\pi$ and so are the values $k+1$ up to $n$.

    We shall distinguish cases depending on the values $i$ and $j$, and show that $\pi$ is unique.

    First, consider the case $i>0$ and $j>0$. Then the positions of $\overline1$ and $1$ do not form a cover inversion, so $1$ is in the right part of $\pi$. Thus, the other positive values must also be in the right part of $\pi$: $1$ up to $k$ have to be on the right part, and since $k+1$ is also on the right, all the other values are too. So any type-$B$ $231$ pattern can only occur with three values on the right half. We are thus in the type-$A$ case and $\pi$ is unique (see~\cite{muehle19tamari}).

    Let us nonetheless prove by ourselves that $\pi$ is unique, since we shall prove uniqueness by building it explicitly using ideas that will be used in other cases of this proof. Since there is only one cover inversion between values $k$ and $k+1$, all values to the left of $k+1$ must be at most $k-1$ and all values to the right of $k$ must be at least $k+2$. Moreover, all positions between the positions of $k+1$ and $k$ are filled with elements that are either smaller or greater than $k$, but not both: the values to the right of $k+1$ in its block must be filled with greater elements whereas all others must be filled with smaller elements. In other words, given $\alpha$ and $\invd{i}{j}$, all positions are filled either with values greater or smaller than $k$, hence showing that the filling is unique and in particular that $k$ is fixed. One easily checks that the element described in the statement is the correct one.

    \medskip
    
    Now, consider the case $i=-j$. Since it is a cover inversion, it has to be that $\pi(j)=\overline1$ and $\pi(i)=1$. Since there are no other cover inversions, all positive values are in increasing order from left to right.

    \begin{itemize}
        \item If $\alpha$ is split, the values to the left of $\overline1$ in the right part of $\pi$ must be negative. Otherwise, we would have a type-$B$ $(\alpha,231)$-pattern. So there are exactly $n-1$ available positions to the right of $1$ for the $n-1$ positive values from $2$ to $n$. So $\pi$ is unique and it has no type-$B$ $(\alpha,231)$-pattern since the right part of $\pi$ is increasing. Moreover, it corresponds to the one in the statement.
        \item If $\alpha$ is join, the same holds except for the central one; but in that case, only the right half of this part can be filled with positive values, again implying that we have to put $n$ values in $n$ given spots, hence only one solution. Again, it avoids any type-$B$ $(\alpha,231)$-pattern (which is a $312$ if the '$1$' is in the first block right half).
    \end{itemize}

    \medskip

    Let us finally consider the case where $i$ and $j$ have different sign and $|i|<|j|$. Let us first solve the case $\alpha$ split. In that case, $i$ is the negative one; otherwise, we would have a type-$B$ $(\alpha,231)$-pattern $\overline k\dots k\dots \overline{k{+}1}$.
    As before, since $\invd{i}{j}$ is the only cover inversion, $1$ has to be in the right part of $\pi$, and so do all positive values up to $k$. Let now $\ell$ be the maximal value of the left part of $\pi$. Then the right part consists of three possibly intertwined sequences: $\overline\ell\dots\overline{k{+}1}$, $1\dots k$, and $\ell{+}1\dots n$. But $\ell+1$ must be to the right of $k$, or we would have a type-$B$ $(\alpha,231)$-pattern $k{+}1\dots\ell{+}1\dots k$; and $1$ must be to the right of $\overline{k+1}$, otherwise we would have a type-$B$ $(\alpha,231)$-pattern as $\overline k\dots1\dots\overline{k{+}1}$. So the three sequences are not intertwined at all: they have to be concatenated. But since the positions of $k$ and $k+1$ are given, one easily gets that $\ell=j$ and $k=\ell+i$, so that $\pi$ is unique and is the one described in the statement.

    \medskip
    
    The join case when $|i|>\alpha_1$ is essentially the same as in the split case up to one small difference: the first $\alpha_1$ values of the first half have to be positive and greater than $\ell$ in order to avoid the type-$B$ $(\alpha,231)$-pattern (which is actually 312). The other positions are filled as in the split case and one finds the sequence of the statement.

    \medskip
    
    Finally, the join case when $|i|<\alpha_1$ is a little different. First, in that case $i$ is the positive one. The values $1$ up to $k$ must still be in the right part and in that order, so that they form a prefix of the right part (so that $k=i$). Now, there still remains to intertwine the sequences $\overline\ell\dots\overline{k{+}1}$ and $\ell{+}1\dots n$. But again, there is no choice: the block containing $\alpha_1$ must be filled after $k$ by values greater than it, and the remaining part of the permutation must be increasing in order to avoid type-$B$ $(\alpha,231)$-patterns, hence the solution of the statement.
\end{proof}

\begin{proposition}\label{prop:alpha_tamari_trim}
    For every type-$B$ composition $\alpha$ of $n$, the lattice $\Tamari_{B}(\alpha)$ is trim.
\end{proposition}
\begin{proof}
    By Proposition~\ref{prop:alpha_tamari_congruence_uniform}, $\Tamari_{B}(\alpha)$ is congruence uniform, and thus by Lemma~\ref{lem:congruence_uniform_semidistributive} it is semidistributive. Now, Lemma~\ref{lem:semidistributive_irreducibles} states that in $\Tamari_{B}(\alpha)$ the number of join-irreducible elements agrees with the number of meet-irreducible elements.  Propositions~\ref{prop:alpha_tamari_length} and \ref{prop:alpha_tamari_irreducibles_count} thus imply that $\Tamari_{B}(\alpha)$ is extremal.  Using Theorem~\ref{thm:semidistributive_extremal_is_trim}, we conclude that it is in fact trim.
\end{proof}

We are now ready to conclude our second main result.

\begin{proof}[Proof of Theorem~\ref{thm:parabolic_tamari_properties}]
    Proposition~\ref{prop:alpha_tamari_congruence_uniform} states that $\Tamari_{B}(\alpha)$ is congruence uniform, which implies together with Lemma~\ref{lem:congruence_uniform_semidistributive} that it is semidistributive.  Proposition~\ref{prop:alpha_tamari_trim} states that $\Tamari_{B}(\alpha)$ is trim.
\end{proof}

\section{Conclusion and outlook}
    \label{sec:outlook}
The constructions and results presented in this article are a natural follow-up to \cite{muehle19tamari}, where combinatorial realizations for parabolic aligned elements in \text{linear type $A$} were considered.  Drawing inspiration from other results concerning the combinatorics of parabolic Catalan objects in linear type $A$ there are some natural research problems that may be studied subsequently, see~\cite{ceballos20the,fang21consecutive,muehle21noncrossing}.  
    
The first natural question to investigate is the exhibition of other combinatorial families associated with parabolic quotients of $\Hyper_{n}$.  In \cite{ceballos20the}, tree and lattice path models for elements of parabolic quotients of the symmetric group were introduced, and these constructions can be extended to produce the corresponding models in linear type $B$.  This is work in progress and will be addressed in a future paper.  

Following the approach from \cite{muehle21noncrossing} in linear type $A$, it will certainly be fruitful to study the secondary structures associated with the type-$B$ parabolic Tamari lattices, such as their Galois graph, canonical join complex or core label order.  This will also help in the research for a noncrossing partition model for parabolic quotients in linear type $B$.

One of the main results of \cite{ceballos20the} relates the parabolic Tamari lattices in type $A$ to the $\nu$-Tamari lattices of Pr{\'e}ville-Ratelle and Viennot~\cite{preville17extension}, see also \cite{fang21consecutive}.  The underlying geometric structure was generalized to type $B$ in \cite{ceballos19geometry}, and it is a natural question to study if there is a similar connection with our type-$B$ parabolic Tamari lattice.  More precisely, we may study the following research question.

\begin{question}
	Let $\alpha$ be a type-$B$ composition of an integer $n$.  Is there some integer $N$ such that the associated type-$B$ parabolic Tamari lattice $\Tamari_{B}(\alpha)$ is isomorphic to an interval in the type-$B$ ordinary Tamari lattice $\Tamari_{B}(N)$?  Is it even possible that $\Tamari_{B}(N)$ contains all $\Tamari_{B}(\alpha')$ as disjoint intervals, where $\alpha'$ ranges over all type-$B$ compositions of $n$?
\end{question}

Besides the structural considerations, it may also be fruitful to consider various enumerative aspects of type-$B$ parabolic Tamari lattices.  For instance, the total sum of the cardinalities of parabolic aligned elements with respect to all parabolic quotients of a fixed symmetric group yields the sequence \cite[A151498]{oeis} and enumerates the dimensions of a certain family of Hopf algebras, see \cite{bergeron22hopf,ceballos20the}.  The counterpart in type $B$ is the sequence $t_{1},t_{2},\ldots$, where
\begin{displaymath}
	t_{n} \defs \sum_{\text{$\alpha$ is a type-$B$ composition of $n$}}\bigl\lvert\Hyper_{\alpha}(231)\bigr\rvert.
\end{displaymath}
The first six terms of this sequence are $3,15,91,598,4109,29071$, which currently is not listed in the OEIS. Our work in progress on other combinatorial models for type $B$ may give a possible explanation of these numbers.

A second intriguing enumerative problem is the determination of the \textit{cover enumerator} of the type-$B$ parabolic Tamari lattices.  More precisely, if $\alpha$ is a type-$B$ composition of $n$, then we define
\begin{displaymath}
	c_{\alpha}(x) \defs \sum_{\pi\in\Hyper_{\alpha}(231)}x^{\lvert\covset(\pi)\rvert}.
\end{displaymath}
When $\alpha=(0,1,1,\ldots,1)$ is a type-$B$ composition of $n$, then we have
\begin{displaymath}
	c_{\alpha}(x) = \sum_{k=0}^{n}\binom{n}{k}^{2}x^{k},
\end{displaymath}
which follows for instance from \cite[Proposition~6]{reiner97non} and \cite[Theorem~6.1]{reading07clusters}.  The coefficients of this sequence may be considered as type-$B$ (parabolic) \textit{Narayana numbers}.  Computer experiments suggest the following.

\begin{conjecture}
	Let $t>0$ and let $\alpha=(t,1,1,\ldots,1)$ be a type-$B$ composition of $n$.  Then,
	\begin{displaymath}
		c_{\alpha}(x)  = \sum_{k=0}^{n-t}\binom{n-t}{k}\binom{n+t}{k}x^{k}.
	\end{displaymath}
	Consequently, $\bigl\lvert\Hyper_{\alpha}(231)\bigr\rvert=\binom{2n}{n-t}$.
\end{conjecture}

In type $A$, the cover enumerator has been determined for certain parabolic quotients of the symmetric group in \cite[Theorem~3.1]{krattenthaler22the} using a slightly different perspective.

We wish to end by observing that computational evidence leads us to suggest that the cardinality of $\Hyper_{\alpha}(231)$, where $\alpha=(0,1,1,\ldots,1,2)$ is a type-$B$ composition of $n$ is given by the type-$D$ Catalan number $\frac{3n-2}{n}\binom{2n-2}{n-1}$.

\subsection*{Acknowledgements}

HM has received funding from the European Research Council (Grant Agreement no. 681988, CSP-Infinity). 
JCN was partially supported by the CARPLO project of the Agence Nationale de la recherche (ANR-20-CE40-0007).


\end{document}